\documentclass[11pt]{article}
\usepackage[utf8]{inputenc}
\usepackage{amsmath}
\usepackage{amsthm}
\usepackage{amsfonts}
\usepackage[numbers]{natbib}
\usepackage{bm}
\usepackage{thm-restate}

\usepackage{wrapfig}
\usepackage{subcaption}
\usepackage{float}
\usepackage{authblk}

\usepackage{color}   
\usepackage{hyperref}
\hypersetup{
    linktoc=all,     
    linkcolor=blue,  
}

\def\reals{{\mathbb R}}
\def\complex{{\mathbb C}}

\def\cov{\mathrm{cov}}
\def\var{\mathrm{var}}
\newcommand{\eps}{\epsilon}
\newcommand{\E}{\mathbb{E}}

\usepackage{framed}
\usepackage{setspace}
\usepackage{listings}
\spacing{1.06}
\usepackage[margin=1in]{geometry}

\usepackage{todonotes}

\usepackage[ruled,vlined]{algorithm2e}

\usepackage{mathtools}

\def\pr{\mathbb{P}}
\newcommand{\ind}{\mathds{1}}

\newcommand\defeq{\,\,\stackrel{\mathclap{\normalfont\mbox{def}}}{=}\,\,}

\def\calD{\mathcal{D}}

\def\calM{\mathcal{M}}
\def\calN{\mathcal{N}}

\def\calR{\mathcal{R}}

\def\calX{\mathcal{X}}

\def\calY{\mathcal{Y}}
\def\norm#1{\mathopen\| #1 \mathclose\|}
\newcommand{\brackets}[1]{\langle #1\rangle}

\def\bx{\text{\textbf{x}}}
\def\by{\text{\textbf{y}}}
\def\be{\text{\textbf{e}}}

\theoremstyle{plain}
\newtheorem{theo}{Theorem}[section]
 
\newtheorem{theorem}[theo]{Theorem}
\newtheorem{lemma}[theo]{Lemma}
\newtheorem{claim}[theo]{Claim}
\newtheorem{corollary}[theo]{Corollary}

\theoremstyle{definition}
\newtheorem{definition}[theo]{Definition}

\theoremstyle{remark}

\usepackage{dsfont}

\DeclareMathOperator*{\argmin}{argmin}

\DeclareMathOperator*{\poly}{poly}

\DeclareMathOperator*{\Bin}{Bin}

\def\Bern{\text{Bern}}
\def\Unif{\text{Unif}}

\def\naturals{{\mathbb N}}

\def\XXAm{\overline{x^2}_1}
\def\XXA{n_1\XXAm}
\def\XXBm{\overline{x^2}_2}
\def\XXB{n_2\XXBm}
\def\XXm{\overline{x^2}}
\def\XX{n\XXm}

\def\XYAm{\overline{xy}_1}
\def\XYA{n_1\XYAm}
\def\XYBm{\overline{xy}_2}
\def\XYB{n_2\XYBm}
\def\XYm{\overline{xy}}
\def\XY{n\XYm}

\def\YYAm{\overline{y^2}_1}

\def\YYBm{\overline{y^2}_2}

\def\YYm{\overline{y^2}}

\def\dpstats{\texttt{DPStats}}
\def\dpkw{\texttt{DPKW}}
\def\dpbern{\texttt{DPBern}}

\newcommand\blfootnote[1]{%
  \begingroup
  \renewcommand\thefootnote{}\footnote{#1}%
  \addtocounter{footnote}{-1}%
  \endgroup
}

\makeatletter
\def\widebreve#1{\mathop{\vbox{\m@th\ialign{##\crcr\noalign{\kern3\p@}%
      \brevefill\crcr\noalign{\kern3\p@\nointerlineskip}%
      $\hfil\displaystyle{#1}\hfil$\crcr}}}\limits}

\def\brevefill{$\m@th \setbox\z@\hbox{$\braceld$}%
  \bracelu\leaders\vrule \@height\ht\z@ \@depth\z@\hfill\braceru$}
\makeatletter

\def\DataSampler{\mathrm{DataSampler}}

\def\DataSamplerList{\mathrm{DataSamplerList}}

\def\MCTester{\mathrm{MonteCarloTester}}

\def\MCTesterList{\mathrm{MonteCarloTesterList}}

\def\EstRejection{\mathrm{EstimateRejectionProb}}

\def\CompareAlgs{\mathrm{CompareAlgorithms}}

\begin{document}

\pagenumbering{roman}

\title{
Hypothesis Testing for
Differentially Private Linear Regression
}
\date{}
\author[1]{Daniel Alabi
\thanks{Email:~\url{alabid@g.harvard.edu}. D.A. was supported by a Fellowship
from Meta AI
and Cooperative Agreement CB20ADR0160001
with the Census Bureau.}
}
\author[2]{Salil Vadhan
\thanks{Email:~\url{salil_vadhan@harvard.edu}. S.V. was supported by Cooperative Agreement
CB20ADR0160001 and a Simons Investigator Award.}}
\affil[1,2]{Harvard School of Engineering and Applied Sciences}
\maketitle

\blfootnote{
The views expressed in this paper are those of the authors and not those of
the U.S. Census Bureau or any other sponsor.
This work will appear at the
2022 IMS (Institute of Mathematical Statistics) and the
2022 SEA (Southern Economic Association) Annual Meetings.}

\begin{abstract}

In this work, we design differentially private hypothesis tests
for the following problems in the general
linear model: testing a linear
relationship and testing for the presence of mixtures.
The majority of our hypothesis tests are based on
differentially private versions of the
$F$-statistic for the general linear model framework, which
are uniformly most powerful unbiased in the non-private setting.
We also present 
other tests for these problems, one of which is
based on the differentially private nonparametric tests of 
Couch, Kazan, Shi, Bray, and Groce (CCS 2019), which is
especially suited for the small dataset regime.
We show that the differentially
private $F$-statistic converges to the 
asymptotic distribution of
its non-private counterpart.
As a corollary, the statistical power of the differentially
private $F$-statistic converges to the statistical power of the
non-private $F$-statistic.
Through a suite of
Monte Carlo based experiments, we show that our tests achieve desired 
\textit{significance levels} and have a high \textit{power} 
that approaches the power of the
non-private tests as we increase sample sizes or the privacy-loss parameter.
We also
show when our tests
outperform existing methods in the literature.

\end{abstract}

\clearpage

\tableofcontents

\clearpage

\pagenumbering{arabic}
\section{Introduction}

Linear regression is one of the most fundamental statistical techniques available to social scientists and
economists (especially econometricians).
One of the goals of
performing regression analysis is for use in decision-making via
point estimation (i.e., getting a single predicted value for the dependent variable). To increase the
confidence of decision-makers and analysts in such estimates, it is often important to also release
accompanying uncertainty estimates for the point 
estimates~\citep{ChettyF19, NBERw25147, 10.1093/qje/qjz042, NBERw25456}.

In this work, we aim to provide
differentially private linear regression
\textit{uncertainty quantification} via the use of hypothesis tests. 
Given the realistic possibility of reconstruction,
membership, and inference attacks~\citep{Sweeney97, DSSU17},
we can rely on 
Differential Privacy (DP), a rigorous approach to quantifying
privacy loss~\citep{DworkMNS06, DKMMN06}.
The task of DP linear regression is to,
given datapoints $\{(x_i, y_i)\}_{i=1}^n$,
estimate
point or uncertainty estimates for linear regression while
satisfying differential privacy.
The majority of our tests will rely on generalized likelihood ratio test $F$-statistics. 

In previous works~\citep{Sheffet17, Wang18, CWZ20, AMSSV20}, 
differentially private estimators
for linear regression are explored
and key factors (such as sample size and variance of the independent variable)
that affect the accuracy of these estimators are identified. 
The focus of these previous works is for point estimate
prediction. The predictive accuracy of 
such estimators can be measured in 
terms of a confidence bound or mean-squared error.
We continue the study of the utility of such estimators for use in uncertainty quantification via
hypothesis 
testing~\citep{Sheffet17}.
(See Section~\ref{sec:prelims} for background on hypothesis testing.)

Earlier work on uncertainty quantification for linear regression
was done by Sheffet~\citep{Sheffet17}, who constructed confidence intervals and
hypothesis tests based on the $t$-test statistic,
and can be used to test a linear relationship.
The random projection routine
in~\citep{Sheffet17}, based on the
Johnson–Lindenstrauss (JL) transform, only starts to
correctly reject the null hypothesis when the sample
size is very large (or the variables have
a large spread). This observation
is also supported by the
work of~\citep{CouchKSBG19}. Furthermore, the
random projection routine
requires extra parameters 
(e.g., for specifying the dimensions of the random
matrix).
In our work, 
we use the $F$-statistic and our framework can be used to
test mixture models, amongst other tests. 
We provide 
a general framework for DP tests based on the $F$-statistic.
In addition, we will consider hypothesis testing for linear regression
coefficients on both small and large datasets.
For the mixture model tester, we additionally adapt and evaluate a nonparametric method, 
a DP analogue
of the Kruskal-Wallis test due to
Couch, Kazan, Shi, Bray, and Groce~\citep{CouchKSBG19}, which is especially suited 
for the small dataset regime.
To the best of our knowledge, our tests are the
first to differentially
privately detect mixtures in linear regression
models, with accompanying experimental validation.

\subsection{Our Contributions}

In this work, 
we show that for the problem of differentially private linear regression, we can perform
hypothesis testing for two problems in the general linear model: 

\begin{enumerate}

\item \textbf{Testing a Linear Relationship}: 
is the slope of the linear model equal to some constant (e.g., slope is 0)?

\item \textbf{Testing for Mixtures}: does the population 
consist of one or more sub-populations with different regression coefficients?

\end{enumerate}

We provide a differentially private analogue of the
$F$-statistic which we,
under the general linear model,
show converges in distribution to the 
asymptotic distribution of the $F$-statistic
(Theorem~\ref{thm:dpf}). Furthermore, the
DP regression coefficients converge, in probability, to
the true coefficients (Lemma~\ref{lem:t1conv2}).
In particular, in Lemma~\ref{lem:t1conv2}, 
we show a $1/\sqrt{n}$ statistical rate of convergence for
the DP regression coefficients used for our hypothesis
tests. This matches the optimal rate~\citep{Lei11}.
We then use our DP $F$-statistic and parametric estimates
to obtain DP hypothesis tests using a Monte Carlo
parametric bootstrap, following Gaboardi, Lim, Rogers, and Vadhan~\citep{RogersVLG16}. The
Monte Carlo parametric bootstrap is used to
ensure that our tests achieve a target
significance level of $\alpha$ (i.e., data generated under the null hypothesis is rejected
with probability $\alpha$).
We experimentally
compare these tests to their non-private counterparts
for univariate linear regression
(i.e., one independent variable and one dependent
variable).
To the best of our knowledge, our tests are the first
that use
the $F$-statistic to perform tests on the problem of
linear regression while ensuring privacy of the data subjects.
In addition, our $F$-statistic based tests can be adapted to
work on design matrices in any dimension.
i.e., the design matrix can be cast in the form
$X\in\reals^{n\times p}$ for any integer $p \geq 2$ where $n$ represents the
number of individuals in the dataset and $p$ represents the number
of features per individual; we leave experimental evaluation of the multivariate
case for future work.

Experimental evaluation of our hypothesis tests is done on:

\begin{enumerate}

\item\textbf{Synthetic Data}:
We generate synthetic datasets with different distributions on the
independent (or explanatory) variables. Specifically, we consider
uniform, normal, and exponential distributions on the 
independent variables. We also vary the noise distribution of the
dependent variable.

\item\textbf{Opportunity Insights (OI) Data}:
We use a simulated version of the data used by the Opportunity 
Insights team (an economics research lab) to release the
Opportunity Atlas tool, primarily used to predict social and
economic mobility. We chose to use these datasets since they come
from a real deployment of privacy-preserving statistics.
\footnote{
See~\citep{NBERw25147, 10.1093/qje/qjz042} for a more detailed description of the use of the
Opportunity Atlas tool in predicting social and economic
mobility.~\citep{ChettyF19, AMSSV20} evaluate privacy-protection methods
on Opportunity Atlas data.
}
The census tract-level datasets from these states can have a very small number of
datapoints.
\item\textbf{Bike Sharing Dataset}:
We use a real-world dataset publicly available in the
UCI machine learning repository.
The dataset consists of daily and hourly counts
(with other information such as seasonal and
weather information) of
bike rentals in the Capital bikeshare system in years
2011 and 2012.
\end{enumerate}

Our experimental findings are as follows:
\begin{enumerate}
\item \textbf{Significance}: Across a variety of experimental settings,
the significance is below the target significance level of 0.05.
Thus, we have a high confidence that we will not falsely reject the null hypothesis, 
should the null hypothesis be true. Our tests are designed to be conservative
in the sense that they err on the
side of failing to reject the null 
(e.g., when the DP estimate of a variance is negative).

\item \textbf{Power}: Consistently, the power of our tests increase as we increase sample sizes
(from hundreds up to tens of thousands) or as we relax the privacy parameter.
The behavior of our DP tests tends toward that of the non-DP tests.
The power of the DP OLS linear relationship tester increases as the
slope of the model increases and as the noise in the dependent variable decreases. But
when the DP estimate of the noise in the dependent variable is negative,
the tests err on the side of failing to reject the null, leading to a lower power.
The power of the mixture model tests also increases as the difference between 
the slopes in the two groups increases. And the more uneven the group sizes are, the
lower the power.

\item\textbf{Alternative Method 1}:
We compare our DP linear
relationship tester, based on the $F$-statistic, to a test that
builds on a
DP parametric bootstrap method for confidence interval
estimation due to
Ferrando, Wang, and Sheldon~\citep{FWS20}. They prove that
these intervals are consistent
(in the asymptotic regime) and experimentally
show that these intervals have good coverage.
We can rely on such confidence intervals to
decide to reject or fail to reject the null hypothesis.
Such methods achieve the desired significance levels.
However, we observe that
the method is less powerful
than the $F$-statistic approach.
This behavior could be attributed to the differences in
the bootstrap process:
whereas we use estimates of sufficient statistics under the
null in the bootstrap procedure,
Ferrando et al.~\citep{FWS20} 
uses the entirety of the sufficient
statistics estimated for the parametric model.

\item\textbf{Alternative Method 2}:
Inspired by the DP regression estimators of
~\citep{DworkL09, AMSSV20},
we also show how to reduce linear relationship 
testing to
the problem of
testing if data is generated from a Bernoulli
distribution with equal chance of outputting heads or tails.
This problem has simple DP tests and has been solved
optimally by Awan and Slavkovic~\citep{Awan_Slavkovic_2020} 
for pure DP
(whereas we use zCDP).
The resulting linear relationship tester is
nonparametric in that, in contrast to the $F$-statistic,
it does not assume that the residuals are normally
distributed. We find that this nonparametric test outperforms
our DP $F$-statistic on smaller privacy-loss
parameters or larger
slopes, but otherwise the $F$-statistic performs better.

\item\textbf{Alternative Method 3}:
For testing mixture models, we also give a reduction to
testing whether two groups have the same median,
which can be solved using
the DP nonparametric method of
Couch, Kazan, Shi, Bray, and Groce~\citep{CouchKSBG19}.
We find that this nonparametric test
has a higher power in the small dataset regime than the
DP $F$-statistic method. 
As the dataset
size increases or the difference in slopes between the groups
increases, the gap closes. In addition, as the variance of
the independent variable increases,
the $F$-statistic method outperforms the nonparametric method.

\end{enumerate}

\subsection{Overview of Techniques}

We provide Algorithm~\ref{alg:mct}, 
a generic framework for DP Monte Carlo tests via a parametric bootstrap routine for estimating
sufficient statistics.
Algorithm~\ref{alg:mct} crucially relies
on $\dpstats$, a procedure that uses statistics of the independent and dependent variables to produce DP statistics. These statistics can be used to
decide to reject or fail to reject the null hypothesis. $\dpstats$ must satisfy DP---in our case $\rho$-zCDP (Zero-Concentrated Differential Privacy).
Algorithms~\ref{alg:t1} and~\ref{alg:t2}
are SSP (Sufficient Statistic Perturbation)
implementations
of DPSS for the linear relationship tester and mixture model tester, respectively. (See Section~\ref{sec:dplinear} for more details.)
Algorithm~\ref{alg:t3}
is a DP nonparametric test framework based on the Kruskal-Wallis test statistic for testing mixtures.
We use the standard Gaussian mechanism to make these algorithms satisfy zCDP although the Laplace mechanism could be used instead
(i.e., for pure DP). 
To make Algorithm~\ref{alg:t3}
DP, we randomly pair datapoints so that
the resulting transformation is 1-Lipschitz.
i.e., changing one datapoint will change a single 
slope estimate. See Section~\ref{sec:dpmixture} for more details.

Theorem~\ref{thm:dpf} shows that the DP $F$-statistic converges to the asymptotic distribution of its non-private counterpart, the chi-squared distribution.
As a corollary, the statistical power of the DP
$F$-statistic converges to the statistical power of the
non-private $F$-statistic.
To prove this result, we first show the $1/\sqrt{n}$ statistical rate of convergence for the DP coefficients. Then we reformulate both the $F$-statistic and the
DP $F$-statistic in terms of the $1/\sqrt{n}$ convergent quantities, including convergent functions of the Gram Matrix of the independent variable.
Next, we construct random variables whose
$\ell_2$-norm squared is distributed as the (non-central) chi-squared distribution. Then we prove that
DP analogues of the squared $\ell_2$-norm of
such random variables converge to the (non-central) chi-squared distribution.
Finally, using such random variables,
the reformulated DP $F$-statistic, 
reformulated $F$-statistic, 
and the continuous mapping theorem,
we combine the convergent quantities to prove Theorem~\ref{thm:dpf}.
See Section~\ref{sec:dpf} for more details.

\subsection{Other Related Work}

\paragraph{Differentially Private Linear Regression}

Sheffet~\citep{Sheffet17} considers hypothesis testing for ordinary least squares for
a specific test: testing for a linear relationship under the assumption that
the independent variable is drawn from a normal distribution.
Wang
~\citep{Wang18} focuses on using adaptive algorithms
for linear regression prediction.

$M$-estimators~\citep{huber2011robust}, 
motivated by the field of robust statistics~\citep{Huber64},
are a simple class of statistical
estimators that present a general approach to
parametric inference.
Dwork and Lei~\citep{DworkL09, Lei11} and
Chaudhuri and Hsu~\citep{ChaudhuriH12} present
differentially private $M$-estimators with
near-optimal statistical rates of convergence.
Avella-Medina~\citep{Medina20} generalizes the $M$-estimator
approach to differentially private statistical inference
using an empirical notion of influence functions to
calibrate the Gaussian mechanism.
Alabi, McMillan, Sarathy, Smith, and 
Vadhan~\citep{AMSSV20} proposed median-based estimators for linear regression
and evaluated their performance for prediction.
All of these previous works show connections between
robust statistics, $M$-estimators, and differential
privacy. The ordinary
least squares estimator is a classical $M$-estimator
for prediction. Other examples include sample
quantiles and the maximum likelihood estimation (MLE)
objective.
However, for differentially private
hypothesis testing, as we show in this work,
the optimal test statistic for DP linear regression
depends on statistical
properties of the dataset (such as variance and sample
size). We also present novel differentially private
test statistics that
converge in distribution to the asymptotic distribution
of the $F$-statistic.

Bernstein and Sheldon~\citep{BernsteinS19} take a Bayesian approach to linear regression
prediction and credible interval estimation.
Through the Bayesian lens, there is also work on how to approximately
bias-correct some DP estimators while providing some uncertainty
estimates in terms of (private) standard errors~\citep{Evans2019StatisticallyVI}.
As a motivation for differentially private
simple linear regression point and uncertainty estimation,
Bleninger, Dreschsler, Ronning~\citep{BleningerDR10}
show how an attacker could use
background information to reveal sensitive attributes
about data subjects used in a simple linear regression
analysis.

\paragraph{General Differentially Private Hypothesis Testing}

Gaboardi, Lim, Rogers, and Vadhan
~\citep{RogersVLG16} study hypothesis testing subject to differential privacy constraints.
The tests they consider are: (1) \textit{goodness-of-fit} tests on multinomial data to
determine if data was drawn from a multinomial distribution with a certain probability vector, and
(2) \textit{independence} tests for checking whether two categorical random variables are
independent of each other. Both tests use the chi-squared test statistic.
Rogers and Kifer
~\citep{RogersK17} develop new test statistics for differentially private hypothesis testing on
categorical data while maintaining a given Type I error.
Through the use of the subsample-and-aggregate
framework,
Barrientos et al.~\citep{Barrientos2017DifferentiallyPS}
compute univariate $t$-statistics by first
partitioning the data into $M$ disjoint subsets,
estimating the statistic on each subset,
truncating the statistic at some threshold $a$,
and then adding noise from a Laplace distribution
to the average of the
truncated $t$-statistics. Our framework requires a
clipping parameter (similar to $a$) but does not
require any others (e.g., the number of
subsets). As noted in that work, the
parameter $M$ plays a significant role in the
performance---and tuning---of their tests while our tests 
require no such parameter tuning on partitioning of the
data.

A subset of previous work~\citep{TaskC16, CampbellBRG18, SwanbergGGRGB19, CouchKSBG19} focus on differentially private independence tests
between a categorical
and continuous variables. Some of these works produce
nonparametric tests which require little or no distributional assumptions on the data generation process.
Specifically, Couch, Kazan, Shi, Bray, and Groce~\citep{CouchKSBG19} develop DP analogues of
rank-based nonparametric tests such as 
Kruskal-Wallis and Mann-Whitney signed-rank tests.
The Kruskal-Wallis test, for example, can be used to
determine whether the medians of two or more groups are
the same. We adapt their test to the setting of linear regression by using it to
compare the distributions of slopes between the two groups.
Wang, Lee, and Kifer~\citep{WangLK15} develop DP versions of likelihood ratio and chi-squared tests, showing a modified
equivalence between chi-squared and likelihood ratio tests.

In the space of differentially private hypothesis testing, 
previous work introduce
methods for differentially private identity and equivalence
testing of discrete 
distributions~\citep{AcharyaCFT19, AcharyaSZ18, AliakbarpourDKR19, AliakbarpourDR18}.
A differentially private version of the
log-likelihood ratio test for the Neyman-Pearson lemma has also been shown to
exist~\citep{CanonneKMSU19}.
Furthermore, 
Awan and Slavkovic~\citep{Awan_Slavkovic_2020}
derive uniformly most powerful DP tests for
simple hypotheses for binomial data.
Suresh~\citep{Suresh20} proposes a hypothesis test,
which can be made to satisfy differential privacy,
that is robust to distribution perturbations under
Hellinger distance. Sheffet and Kairouz, Oh, Viswanath
~\citep{Sheffet18, KairouzOV16} also consider
hypothesis testing, although in the local setting of
differential privacy. 

\section{Preliminaries and Notation}
\label{sec:prelims}

\subsection{Differential Privacy}

For the definitions below, we say that two databases $\bx$ and $\bx'$ are
neighboring, expressed notationally as $d(\bx, \bx')=1$ for any $\bx, \bx'\in\calX^n$,
if
$\bx$ differs from $\bx'$ in exactly one row.

\begin{definition}[Differential Privacy~\citep{DworkMNS06, DKMMN06}]
Let $\calM:\calX^n\rightarrow\calR$ be a (randomized) mechanism.
For any $\eps\geq 0, \delta\in[0, 1]$, we say $\calM$ is 
\textbf{$(\eps, \delta)$-differentially private}
if for all neighboring databases $\bx, \bx'\in\calX^n$,
$d(\bx, \bx')=1$ and every $S\subseteq\calR$,
$$
\pr[\calM(\bx)\in S] \leq e^\eps\cdot\pr[\calM(\bx')\in S] + \delta.
$$
\end{definition}

If $\delta = 0$, then we say $\calM$ is $\eps$-DP, sometimes referred to as \textbf{pure}
differential privacy.
Typically, $\eps$ is a small constant (e.g., $\eps\in[0.1, 1]$) and $\delta\leq 1/\poly(n)$
is cryptographically small.

\begin{definition}[Zero-Concentrated Differential Privacy (zCDP)~\citep{BunS16}]
Let $\calM:\calX^n\rightarrow\calR$ be a (randomized) mechanism.
For any neighboring databases $\bx, \bx'\in\calX^n$, 
$d(\bx, \bx')=1$, we say
$\calM$ satisfies \textbf{$\rho$-zCDP} if for all $\alpha\in (1, \infty)$,
$$
D_\alpha(\calM(\bx)\Vert \calM(\bx')) \leq \rho\cdot\alpha,
$$
where $D_\alpha(\calM(\bx)\Vert \calM(\bx'))$ is the R\'enyi divergence of order $\alpha$ between
the distribution of $\calM(\bx)$ and the distribution of $\calM(\bx')$.
\footnote{A related differential privacy notion, in terms of the
R{\'{e}}nyi divergence, is given in~\citep{Mironov17}.}
\end{definition}

In this paper, we will primarily use $\rho$-zCDP as our definition of differential privacy, adding noise from a Gaussian distribution
to ensure zCDP.

\subsection{General Hypothesis Testing}

The goal of hypothesis testing is to infer, based on data, which of two hypothesis, $H_0$ (the null hypothesis) or $H_1$ (the alternative hypothesis), should be rejected.

Let $P_\theta$ be a family of probability distributions parameterized by $\theta\in\Omega$.
For some unknown parameter $\theta\in\Omega$, let $Z\sim P_\theta$ be the observed data.
Then the two competing hypothesis are:
$$
H_0 : \theta\in\Omega_0 \text { vs. } H_1 : \theta\in\Omega_1,
$$
where $(\Omega_0, \Omega_1)$ form a partition of $\Omega$.

A \textbf{test statistic} $T$ is random variable that is a function of the
observed data $Z\sim P_\theta$. $T$ can be used to decide whether to reject or fail
to reject the null hypothesis.
A \textbf{critical region} $S$ is the set of values for the test statistic 
$T$ (or correspondingly
for the observed data) for which the null hypothesis will be rejected. It
can be used to completely determine a test of $H_0$ 
versus $H_1$ as follows: We reject $H_0$ if $Y\in S$ and fail to reject
$H_0$ if $Y\notin S$.

Sometimes, 
\textit{external randomization}
might help with choosing 
between hypothesis $H_0$ and $H_1$
~\citep{Edgington2011, keener2010theoretical}. 
By external randomness, we mean randomness not 
inherent in the sample or data collection process.
In order to discriminate between hypothesis $H_0$ and $H_1$, we can
define a notion of a critical function that can indicate the
degree to which a test statistic is
within a critical region.
A \text{critical
function} $\phi$ with range in $[0, 1]$ characterize
randomized hypothesis tests. A nonrandomized test with
critical region $S$ can thus be specified as
$\phi = 1_S$. Conversely, if $\phi(y)$ is always 0 or 1
for all $y$ then the critical region is
$S=\{y\;:\;\phi(y) = 1\}$ for this nonrandomized test. 
An advantage of allowing randomization (even without DP constraints) is that
convex combinations of nonrandomized tests are not possible,
but convex combinations of randomized tests are 
possible.
i.e., if $\phi_1, \phi_2$ are critical functions and $t\in(0, 1)$,
then $t\phi_1 + (1-t)\phi_2$ is also a critical function so that
the set of all critical functions form a convex set.
Furthermore, nontrivial differentially private tests must be randomized.

For any $\theta\in\Omega$, the ideal test would tell us when $\theta\in\Omega_0$ and
when $\theta\in\Omega_1$. This can be described by a \textbf{power function} $R(\cdot)$, which gives the
chance of rejecting $H_0$ as a function of $\theta\in\Omega$:
$$
R(\theta) = \pr_\theta(Y\in S),
$$
for any critical region $S$.

A ``perfect'' hypothesis test would have
$R(\theta) = 0$ for every $\theta\in\Omega_0$ and
$R(\theta) = 1$ for every $\theta\in\Omega_1$. But this is generally impossible
given only the ``noisy'' observed data $Z\sim P_\theta$.

A \textbf{significance level $\alpha$} can be defined as
$$
\alpha = \sup_{\theta\in\Omega_0} \pr_\theta(Y\in S).
$$

In other words, the level $\alpha$ is the worst chance of incorrectly
rejecting $H_0$. Ideally, we want
tests that have a small chance of error when $H_0$
should not be rejected. The \textbf{$p$-value} is
the probability of finding, based on observed data, test statistics
at least as extreme as when the null hypothesis holds.
That is, if $T$ is the test statistic function and
$t$ is the observed test statistic, then the (one-sided)
$p$-value is $\pr[T \geq t\mid H_0]$.

\subsection{Convergence and Limits}

\begin{definition}[Limit of Sequence]
A sequence
$\{x_n\}$ \textbf{converges} toward the limit $x$,
denoted $x_n\rightarrow x$, if
$$
\forall \eps > 0,
\lim_{n\rightarrow\infty}\pr\left[|x_n - x| > \eps\right] = 0.
$$

\label{def:seqconverge}
\end{definition}

We will use Definition~\ref{def:seqconverge}
to show convergence of random variables
(in probability or distribution).

\begin{definition}[Convergence in Probability]
A sequence of random variables
$\{X_n\}$ \textbf{converges in probability}
toward random variable $X$,
denoted $X_n\xrightarrow{P} X$, if
$$
\forall \eps > 0,
\lim_{n\rightarrow\infty}\pr\left[|X_n - X| > \eps\right] = 0.
$$
\label{def:prob}
\end{definition}

Convergence in Probability (Definition~\ref{def:prob})
for a sequence of random variables $X_1, X_2, \ldots$
toward random variable $X$ can be shown
if for all $\eps > 0, \delta > 0$, there exists
$N(\eps, \delta) = N$ such that
for all $n\geq N$,
$\pr[|X_n - X| > \eps] < \delta$.

\begin{definition}[Convergence in Distribution]
A sequence of random variables
$\{X_n\}$ \textbf{converges in distribution}
toward random variable $X$,
denoted $X_n\xrightarrow{D} X$, if
$$
\lim_{n\rightarrow\infty}F_n(x) = F(x),
$$
for all $x\in\reals$ at which $F$ is continuous.
$F_n, F$ are the cumulative distribution functions
for $X_n, X$ respectively.

\label{def:dist}
\end{definition}

We can
generalize Definitions~\ref{def:prob} and~\ref{def:dist}
to random vectors and matrices
(beyond scalars) as follows:
if $A\in\reals^{K\times L}$ is a random vector,
then $A_n\xrightarrow{P}A$, 
$A_n\xrightarrow{D}A$ denotes entry-wise
convergence in probability and distribution,
respectively.
Also, for any distribution $\calD$ 
(e.g., $\calD = \chi^2_k$), we write
$A_n\xrightarrow{D} \calD$ as a shorthand to mean that
$A_n\xrightarrow{D} A$ for random variables
$A_n, A$ such that $A$ follows $\calD$ (i.e.,
$A\sim\calD$).
Similarly, $\calD_n\xrightarrow{D} \calD$ implies that
$A_n\xrightarrow{D} A$ for $A_n\sim\calD_n$ and
$A\sim\calD$.

\begin{lemma}
Let $\{X_n\}$ and $\{Y_n\}$ be a sequence of random
vectors and $X$ be a random vector. Then:
\begin{enumerate}
\item If $X_n\xrightarrow{D}X$ and $X_n-Y_n\xrightarrow{P}0$, then
$Y_n\xrightarrow{D}X$.
\item If $X_n\xrightarrow{P} X$, then
$X_n\xrightarrow{D} X$.
\item For a constant $c\in\reals$, if $X_n\xrightarrow{D}c$,
then $X_n\xrightarrow{P}c$.
\end{enumerate}
\label{lem:helperprob}
\end{lemma}

\begin{proof}
Follows from
Theorem 2.7 in~\citep{vaart_1998}.
\end{proof}

Lemma~\ref{lem:helperprob} is a helper lemma that is
useful for proving convergence results, especially on
DP estimates.

In later sections, we will show that the differentially private $F$-statistic
converges, in distribution, to a chi-squared distribution
(as does the non-DP $F$-statistic). This convergence result holds
under certain conditions.

\section{Hypothesis Testing for Linear Regression}

In this section, we review the theory of (non-private) hypothesis testing in the
general linear model.
We will consider hypothesis testing in the linear model
$$
Y = X\beta + \be,
$$
where $X\in\reals^{n\times p}$ is a matrix of known constants, 
$\beta\in\reals^p$ is the parameter vector that determines the 
linear relationship between $X$ and the dependent variable $Y$, and
$\be$ is a random vector such that for all $i\in[n]$, $\E[e_i] = 0$,
$\var[e_i] = \sigma_e^2$. Furthermore, for all $i\neq j\in[n]$, $\cov(e_i, e_j) = 0$.

Note that the simple linear regression model,
$y_i = \beta_2 + \beta_1\cdot x_i + e_i$ for scalars $x_i, y_i$ and $e_i$
$\forall i\in[n]$,
can be cast as a linear model as follows:
$X\in\reals^{n\times 2}$ where
\begin{equation}
X = \begin{pmatrix}
1 & x_1\\
1 & x_2\\
\cdots & \cdots\\
1 & x_{n-1}\\
1 & x_n\\
\end{pmatrix}.
\label{eq:design}
\end{equation}

We will consider the general linear model:
$Y\sim\calN(X\beta, \sigma_e^2I_{n\times n})$, where
$I_{n\times n}$ is the $n\times n$ identity matrix.
Let $\omega$ be an $r$-dimensional
linear subspace of $\reals^p$ and $\omega_0$ be a $q$-dimensional linear subspace of $\omega$ such that
$0\leq q < r$.
We will consider hypothesis tests of the form:
\begin{enumerate}
\item $H_0$: $\beta\in\omega_0$.
\item $H_1$: $\beta\in\omega - \omega_0$.
\end{enumerate}

Let $\hat\beta$ and $\hat\beta^N$ denote the least squares estimates under the
alternative and null hypothesis respectively. 
In other words,
$$
\hat\beta^N = \argmin_{z\in\omega_0}\norm{Xz - Y}^2,
\quad
\hat\beta = \argmin_{z\in\omega}\norm{Xz - Y}^2.
$$

The \textbf{test statistic} we will use is equivalent to the generalized likelihood ratio test statistic
\begin{align}
T &=  \left(\frac{n-r}{r-q}\right)\cdot\frac{\norm{Y - X\hat\beta^N}^2 - \norm{Y - X\hat\beta}^2}{\norm{Y - X\hat\beta}^2} \\
&= \left(\frac{n-r}{r-q}\right)\cdot\frac{\norm{X\hat\beta - X\hat\beta^N}^2}{\norm{Y - X\hat\beta}^2}\\
&= \frac{1}{r-q}\cdot\frac{\norm{X\hat\beta - X\hat\beta^N}^2}{S^2},
\label{eq:glrt}
\end{align}
where $S^2 = \norm{Y - X\hat\beta}^2/(n-r)$. The vectors $Y-X\hat\beta$ and $X\hat\beta - X\hat\beta^N$ can be shown to
be orthogonal, so that $\norm{Y-X\hat\beta^N}^2 = \norm{Y-X\hat\beta}^2 + \norm{X\hat\beta-X\hat\beta^N}^2$ by the Pythagorean theorem~\citep{keener2010theoretical}.

When $r-q = 1$, 
this test is \textit{uniformly most powerful unbiased}
and for $r-q > 1$, the test
is most powerful amongst all tests that satisfy certain symmetry restrictions
~\citep{keener2010theoretical}.

\begin{theorem}

For every $n\in\naturals$ with $n > r$, let
$X = X_n\in\reals^{n\times p}$ be the design matrix.
Under the general linear model 
$Y = Y_n\sim\calN(X_n\beta, \sigma_e^2I_{n\times n})$,
$$T = T_n\sim F_{r-q, n-r}(\eta_n^2),\quad \eta_n^2 = \frac{\norm{X_n\beta - X_n\beta^N}^2}{\sigma_e^2},
$$
where $F_{n, m}$ is the $F$-distribution with parameters $n$, $m$, 
$\beta^N = \E[\hat\beta^N]$,
$q$ is the dimension of $\omega_0$, and $r$ is the dimension of $\omega$ with $0\leq q < r$.

Furthermore,
\begin{enumerate}
\item
$$\norm{Y_n-X_n\hat\beta}^2\sim\calX^2_{n-r}\sigma_e^2,\quad\norm{X_n\hat\beta-X_n\hat\beta^N}^2\sim\calX^2_{r-q}(\eta_n^2)\sigma_e^2.$$
\item
If there exists $\eta\in\reals$ such that $\frac{\norm{X_n\beta - X_n\beta^N}^2}{\sigma_e^2}\rightarrow\eta^2$, then
$$
T = T_n\sim F_{r-q, n-r}(\eta_n^2)
\xrightarrow{D} \frac{\chi^2_{r-q}(\eta^2)}{r-q}.
$$
\item We have
$$\frac{\norm{Y_n-X_n\hat\beta}^2}{n-r}\xrightarrow{P}\sigma_e^2.$$
\end{enumerate}

The values $\beta = \E[\hat\beta], \beta^N = \E[\hat\beta^N]$ are
the expected values of our parameter estimates under
the alternative and null hypotheses respectively.

\label{thm:f}
\end{theorem}

The proof of Theorem~\ref{thm:f} appears
in the Appendix as Theorem~\ref{thm:fstat}. 
Above, the \textbf{noncentral $F$-distribution} $F_{n, m}(\lambda)$,
with parameters $n, m$ and noncentrality parameter $\lambda$ is the
distribution of $\frac{\chi_n^2(\lambda)/n}{\chi_m^2/m}$, the ratio of two scaled chi-squared random variables.
$\chi^2_K(\lambda)$ is a random variable distributed according to
a chi-squared distribution with $K$
degrees of freedom and noncentrality parameter $\lambda$. 
That is, $\chi^2_K(\lambda)$ is distributed as the
squared length of a $\calN(v, I_{K\times K})$ vector where
$v\in\reals^K$ has length $\lambda$.
Also, $\chi^2_K\sim\chi^2_K(0)$.

\subsection{Testing a Linear Relationship in Simple Linear Regression Models}
\label{sec:testlinear}

Consider the model:
$y_i = \beta_2 + \beta_1\cdot x_i + e_i$,
where $e_i\sim\calN(0, \sigma_e^2)$ are i.i.d. random variables and $x_1, \ldots, x_n$ are constants that form the following design matrix for our problem
$$
X = \begin{pmatrix}
1 & x_1\\
1 & x_2\\
\cdots & \cdots\\
1 & x_{n-1}\\
1 & x_n
\end{pmatrix}.
$$

In this case, $\omega = \reals^2$ and
$\omega_0 = \{\beta\in\reals^2\,:\,\beta_1 = 0\}$. As a result, our 
hypothesis is:
\begin{enumerate}
\item $H_0$: $\beta_1 = 0$.
\item $H_1$: $\beta_1 \neq 0$.
\end{enumerate}
Note that $r = p = 2$ and $q = 1$.

Furthermore, let
$$
\beta = \begin{pmatrix}
\beta_2\\
\beta_1\\
\end{pmatrix},\quad
\hat\beta = \begin{pmatrix}
\hat\beta_2\\
\hat\beta_1\\
\end{pmatrix}.
$$

We use $\hat\beta^N$ to refer to the estimate of $\beta$ when the
null hypothesis is true (i.e., $\beta_1 = 0$) and
$\hat\beta$ be the estimate of $\beta$ when the alternative hypothesis holds.

For the calculations below, let
\begin{enumerate}
\item $\bx \defeq (x_1, x_2, \ldots, x_n)^T$, $\by \defeq (y_1, y_2, \ldots, y_n)^T$,
\item $\bar{x} \defeq \frac{1}{n}\sum_{i=1}^nx_i$, $\bar{y} \defeq \frac{1}{n}\sum_{i=1}^ny_i$,
\item $\XXm \defeq \frac{1}{n}\sum_{i=1}^nx_i^2$, $\XYm \defeq \frac{1}{n}\sum_{i=1}^n x_i y_i$,
\item $\widehat{\sigma^2_{xy}} \defeq \XYm - \bar{x}\cdot\bar{y}$, and
$\widehat{\sigma^2_x} \defeq \XXm - \bar{x}^2$.
\end{enumerate}

We can then obtain the sufficient statistics
\begin{equation}
X^TX = \begin{pmatrix}
n & n\bar{x}\\
n\bar{x} & \XX\\
\end{pmatrix},\,\,
X^TY = \begin{pmatrix}
n\bar{y}\\
\XY\\
\end{pmatrix},
\label{eq:xtx}
\end{equation}
so that under the alternative hypothesis, the least squares estimate is
$$
\hat\beta = \argmin_{\beta\in\omega}\norm{Y - X\beta}^2 = (X^TX)^{-1}X^TY,
$$
assuming that $X^TX$ is invertible which happens iff
$\bx$ is not the constant vector (so that 
det($X^TX$) = $n^2\XXm - n^2\bar{x}^2 = n^2\cdot\widehat{\sigma^2_x} > 0$).
Assuming invertibility, we have
\begin{equation}
(X^TX)^{-1} = \frac{1}{n^2\XXm-n^2\bar{x}^2}\begin{pmatrix}
\XX & -n\bar{x}\\
-n\bar{x} & n\\
\end{pmatrix}.
\label{eq:inv}
\end{equation}

Thus, the least squares estimate under the alternative hypothesis
is
\begin{align*}
\hat\beta &= (X^TX)^{-1}X^TY \\
&=  \frac{1}{n^2\cdot\XXm-n^2\cdot\bar{x}^2}\begin{pmatrix}n^2\cdot\XXm\cdot\bar{y} - n^2\cdot\bar{x}\cdot\XYm\\-n^2\cdot\bar{x}\cdot\bar{y} + n^2\cdot\XYm\end{pmatrix}
= \begin{pmatrix}\hat\beta_2\\\hat\beta_1\end{pmatrix},
\end{align*}
and further simplification results in the following slope and intercept estimates:
$$
\hat\beta_1 = \frac{\XYm - \bar{x}\cdot\bar{y}}{\XXm - \bar{x}^2} = \frac{\widehat{\sigma^2_{xy}}}{\widehat{\sigma^2_x}},$$
$$
\hat\beta_2 = \bar{y}-\hat\beta_1\bar{x} = \frac{\bar{y}\cdot\XXm - \bar{x}\cdot\XYm}{\widehat{\sigma^2_x}}.
$$

The square of residuals is $\norm{Y - X\hat\beta}^2$ and
an (unbiased) estimate of $\sigma_e^2$ is $S^2 = \frac{\norm{Y - X\hat\beta}^2}{n-2}$.

Also, we can derive $\hat\beta^N$ as follows
$$\hat\beta^N = \argmin_{\beta\in\omega_0}\norm{Y - X\beta}^2 =
\begin{pmatrix}\bar{y}\\0\end{pmatrix} = \begin{pmatrix}\hat\beta_2^N\\0\end{pmatrix} $$
so that
\begin{align*}
&\norm{X\hat\beta - X\hat\beta^N}^2 = 
\sum_{i=1}^n(\hat\beta_2 + \hat\beta_1x_i - \hat\beta_2^N)^2 \\
&= \sum_{i=1}^n(\bar{y} - \hat\beta_1\bar{x} + \hat\beta_1x_i - \bar{y})^2 =
\hat\beta_1^2\sum_{i=1}^n(x_i - \bar{x})^2 \\
&= 
\hat\beta_1^2\cdot n\cdot\widehat{\sigma^2_x}.
\end{align*}

As a result, the test statistic $T$ is
$$
T = \left(\frac{n-r}{r-q}\right)\frac{\norm{X\hat\beta - X\hat\beta^N}^2}{\norm{Y - X\hat\beta}^2} = \frac{\hat\beta_1^2}{S^2}\cdot n\cdot\widehat{\sigma^2_x}.
$$

\noindent\textbf{Level-$\alpha$ Test}: Under $H_0$, $T\sim F_{1, n-2}$
since by Theorem~\ref{thm:f}, the centrality parameter is $\eta^2 = 0$.
The level-$\alpha$ test will then reject the null if
$T$ is greater than the upper $\alpha$th quantile of this $F$-distribution,
$F_{\alpha, 1, n-2}$.

In other words, we will reject the null if
$$
\frac{\hat\beta_1^2}{S^2}\cdot n\cdot\widehat{\sigma^2_x} > F_{\alpha, 1, n-2}.
$$

We see that the chance of rejecting the null increases as:
\begin{enumerate}
\item $\hat\beta_1^2$, the square of the slope estimate, increases.
\item $\widehat{\sigma^2_x}$ increases.
\item $n$ increases.
\item $S^2$ decreases.
\end{enumerate}

\noindent\textbf{Power}:
The power is the chance that a random variable distributed as
$F_{1, n-2}(\eta^2)$ exceeds $F_{\alpha, 1, n-2}$ where the centrality 
parameter is $\eta^2 = \frac{\beta_1^2}{\sigma_e^2}\cdot n\cdot\widehat{\sigma^2_x}$.

In other words, the power of the test is
$$
\pr\left[F_{1, n-2}\left(\frac{\beta_1^2}{\sigma_e^2}\cdot n\cdot\widehat{\sigma^2_x}\right) > F_{\alpha, 1, n-2}\right],
$$
where the probability is over the draws of the $F$-distribution.

\subsection{Testing for Mixtures in Simple Linear Regression Models}

The goal of testing mixtures is to detect the presence of
sub-populations.
Consider the model where $n = n_1 + n_2, n_1, n_2 > 0$, $\beta_1, \beta_2\in\reals$
with the following generation model:
\begin{itemize}
\item $y_i = \beta_1\cdot x_i + e_i$ for $i\in[n_1]$.
\item $y_i = \beta_2\cdot x_i + e_i$ for $i\in\{n_1 + 1, \ldots, n\}$.
\end{itemize}
where $e_i\sim\calN(0, \sigma_e^2)$ are i.i.d. random variables and $x_1, \ldots, x_n$ are constants that form the following design matrix for our problem
$$
X = \begin{pmatrix}
x_1 & 0\\
\cdots & \cdots\\
x_{n_1} & 0\\
0 & x_{n_1 + 1}\\
\cdots & \cdots\\
0 & x_{n}
\end{pmatrix}.
$$
Note that $X$ is of full rank (except if all the $x_i$'s are 0 either for all $i\in[n_1]$ or
for all $i\in\{n_1 + 1, \ldots, n\}$). Furthermore, $r = p = 2$.

In this case, $\omega = \reals^2$ and
$\omega_0 = \{\beta\in\reals^2\,:\,\beta_1 = \beta_2\}$. As a result, our 
hypothesis is:
\begin{enumerate}
\item $H_0$: $\beta_1 = \beta_2$.
\item $H_1$: $\beta_1 \neq \beta_2$.
\end{enumerate}

Furthermore, let
$$
\beta = \begin{pmatrix}
\beta_1\\
\beta_2\\
\end{pmatrix},\quad
\hat\beta = \begin{pmatrix}
\hat\beta_1\\
\hat\beta_2\\
\end{pmatrix}.
$$

We use $\hat\beta^N$ to refer to the estimate of $\beta$ when the
null hypothesis is true (i.e., $\beta_1 = \beta_2$) and
$\hat\beta$ be the estimate of $\beta$ when the alternative hypothesis holds.

For the calculations below, let $n_2 = n-n_1$ and
\begin{itemize}
\item 
$\XXAm = \frac{1}{n_1}\sum_{i=1}^{n_1}x_i^2, \XXBm = \frac{1}{n_2}\sum_{i=n_1 + 1}^{n}x_i^2, \XXm = \frac{1}{n}\sum_{i=1}^nx_i^2$.
\item 
$\XYAm = \frac{1}{n_1}\sum_{i=1}^{n_1}x_iy_i, \XYBm = \frac{1}{n_2}\sum_{i=n_1 + 1}^{n}x_iy_i, \XYm = \frac{1}{n}\sum_{i=1}^nx_iy_i$.
\end{itemize}
We can then obtain
$$
X^TX = \begin{pmatrix}
\XXA & 0\\
0 & \XXB\\
\end{pmatrix},
X^TY = \begin{pmatrix}
\XYA\\
\XYB\\
\end{pmatrix},
$$
so that, assuming $\XXAm, \XXBm > 0$, we have
$$
\hat\beta = (X^TX)^{-1}X^TY = \begin{pmatrix}
\XYAm/\XXAm\\
\XYBm/\XXBm\\
\end{pmatrix}.
$$

Furthermore, 
$$
\hat\beta^N = \begin{pmatrix}
\XYm/\XXm\\
\XYm/\XXm\\
\end{pmatrix},
$$
since under the null hypothesis ($\beta_1 = \beta_2$),
the design matrix ``collapses'' to
$$
X_0 = \begin{pmatrix}
x_1\\
\cdots\\
x_{n_1}\\
x_{n_1 + 1}\\
\cdots\\
x_{n}
\end{pmatrix},
$$
so that $X_0^TX_0 = \sum_{i=1}^nx_i^2 = \XX$ and
$X_0^TY = \sum_{i=1}^nx_iy_i = \XY$.

The squares of residuals is $\norm{Y - X\hat\beta}^2$ and
an (unbiased) estimate of $\sigma_e^2$ is $S^2 = \frac{\norm{Y - X\hat\beta}^2}{n-2}$.

\begin{lemma}
$$
\norm{X\hat\beta - X\hat\beta^N}^2 = \frac{\XXA\XXB}{\XX}(\hat\beta_1 - \hat\beta_2)^2,
$$
where $X$ is the design matrix, $\hat\beta, \hat\beta^N$ are the least squares estimates
under the alternative and null hypothesis, respectively.

\label{lem:avg}
\end{lemma}

\begin{proof}

First, from previous calculations, we obtained
$$
\hat\beta = \begin{pmatrix}
\XYAm/\XXAm\\
\XYBm/\XXBm\\
\end{pmatrix},\quad
\hat\beta^N = \begin{pmatrix}
\XYm/\XXm\\
\XYm/\XXm\\
\end{pmatrix}.
$$

Then $\hat\beta^N_1 = \hat\beta^N_2$ is a weighted average of $\hat\beta_1$ and
$\hat\beta_2$:
$$
\hat\beta^N_1 = \frac{\XXA}{\XX}\hat\beta_1 + \frac{\XXB}{\XX}\hat\beta_2.
$$

Using this calculation, we can obtain that
$$
\hat\beta_1 - \hat\beta_1^N = \frac{\XXB(\hat\beta_1 - \hat\beta_2)}{\XX},\quad
\hat\beta_2 - \hat\beta_1^N = \frac{\XXA(\hat\beta_2 - \hat\beta_1)}{\XX},
$$
so that
\begin{align*}
\norm{X\hat\beta - X\hat\beta^N}^2 
&= \sum_{i=1}^{n_1}(x_i\hat\beta_1 - x_i\hat\beta^N_1)^2 +  \sum_{i=n_1+1}^{n}(x_i\hat\beta_2 - x_i\hat\beta^N_1)^2\\
&= (\hat\beta_1 - \beta_1^N)^2\XXA + (\hat\beta_2 - \beta_1^N)^2\XXB\\
&= \frac{\XXA\XXB}{\XX}(\hat\beta_1 - \hat\beta_2)^2.
\end{align*}

This completes the proof.

\end{proof}

By Lemma~\ref{lem:avg}, our test statistic $T$ is
$$
T = \left(\frac{n-r}{r-q}\right)\frac{\norm{X\hat\beta - X\hat\beta^N}^2}{\norm{Y - X\hat\beta}^2} = \frac{\XXA\XXB}{S^2\XX}(\hat\beta_1 - \hat\beta_2)^2.
$$

\noindent\textbf{Level-$\alpha$ Test}: Under $H_0$, $T\sim F_{1, n-2}$
since by Theorem~\ref{thm:f}, the centrality parameter is $\eta^2 = 0$.
The level-$\alpha$ test will then reject the null if
$T$ is greater than the upper $\alpha$th quantile of this $F$-distribution,
$F_{\alpha, 1, n-2}$.

In other words, we will reject the null if
$$
\frac{\XXA\XXB}{S^2\XX}(\hat\beta_1 - \hat\beta_2)^2 > F_{\alpha, 1, n-2}.
$$
We see that the chance of rejecting the null increases as:
\begin{enumerate}
\item $|\hat\beta_1 - \hat\beta_2|$ increases.
\item $S^2$ decreases.
\item The ratio $\frac{\XXA\XXB}{\XX}$
increases,
which is more likely to occur when $\XXA$ 
is close to $\XXB$.
\end{enumerate}

\noindent\textbf{Power}:
The power is the chance that a random variable distributed as
$F_{1, n-2}(\eta^2)$ exceeds $F_{\alpha, 1, n-2}$ where the centrality parameter is
$\eta^2 = \frac{\XXA\XXB}{\sigma_e^2\XX}(\beta_1 - \beta_2)^2$.

In other words, the power of the test is
$$
\pr\left[F_{1, n-2}\left(\frac{\XXA\XXB}{\sigma_e^2\XX}(\beta_1 - \beta_2)^2\right) > F_{\alpha, 1, n-2}\right],
$$
where the probability is over the draws of the $F$-distribution.

\subsection{Generalization to Higher Dimensions}

Testing for a linear relationship and for mixtures using 
the $F$-statistic can be done in the multiple linear regression model
as well. The main changes that will need to be made are:
\begin{enumerate}
\item We can use any general design matrix $X\in\reals^{n\times p}$;
and
\item The parameter to be estimated lives in $\reals^p$
instead. i.e., $\beta\in\reals^p$, for any $p\geq 2$.
\end{enumerate}

\section{Differentially Private Monte Carlo Tests}

Since the private test statistic differs from the
non-private version, we have to create new statistics
to account for the level-$\alpha$ Monte Carlo differentially
private testing. The majority of our tests will be based on
DP sufficient statistics. In the statistics literature, a statistic
is considered \textbf{sufficient}, with respect to a particular
model, if it provides at least as much information for the value of an
unknown parameter as any other
statistic that can be calculated on a given sample~\citep{keener2010theoretical}.

Previous work~\citep{Sheffet17, Wang18, AMSSV20} perturb the
sufficient statistics for ordinary least squares 
and use the
result to compute a slope and intercept in a DP way.
To add noise to ensure privacy, we typically have to truncate certain random
variables. We use $Y|^A_B$ to mean that the random variable $Y$ will be truncated to have
an upper bound of $A$ and a lower bound of $B$.

For all our DP OLS Monte Carlo
tests that sample from a continuous Gaussian, we can
instead use discrete variants (e.g.,~\citep{Canonne0S20}).
The DP OLS Monte Carlo tests below use the zero-concentrated
differential privacy definition~\citep{BunS16}.

\subsection{Monte Carlo Hypothesis Testing}

We now proceed to discuss our general approach for designing a Level-$\alpha$ test for the task
of linear regression estimation based on sufficient statistic perturbation.
We rely on a sub-routine $\dpstats$ that can produce DP
statistics, given a dataset and
privacy parameters,
when testing. Algorithm~\ref{alg:mct} can then be specialized to test for the presence of a linear
relationship and for mixture models.

To design a Monte Carlo hypothesis test, we
follow a similar route to 
Gaboardi, Lim, Rogers, and Vadhan~\citep{RogersVLG16}.
In Algorithm~\ref{alg:mct}, we provide a framework to perform
DP Monte Carlo tests using a parametric bootstrap
based on a test statistic.
Let $\dpstats$ be a procedure that uses one or more
statistics of $X, Y$ to produce DP statistics that can be used to reject or
fail to reject the null hypothesis.
In this paper, $\dpstats$ will satisfy
$\rho$-zCDP (Zero-Concentrated Differential Privacy). 
\footnote{
Gaussian noise addition (for privacy) was chosen
because the noise in the dependent variable is also assumed to
be Gaussian. The use of the Gaussian (or truncated
Gaussian) distribution for both privacy and sampling error is a
convenient choice as it could result
in a clearer, more compatible, theoretical analysis.
}
$T$ is the test statistic computation procedure.
As done in~\citep{RogersVLG16}, for example, we will assume the dataset sizes are
public information.

Let $T = T(\hat\theta_1)$ be the non-private test statistic procedure given
$\hat\theta_1 = \hat\theta_1(X, Y)$.
The goal is to compute
$T(\tilde\theta_1)$ where $\tilde\theta_1$ is an approximation of 
$\hat\theta_1$.
$\dpstats$ returns $\tilde\theta_0$ and $\tilde\theta_1$. If 
$\tilde\theta_0$ and $\tilde\theta_1$ is not $\perp$ ($\perp$ is
returned whenever the perturbed statistics cannot be used to simulate the
null distributions), then we use $\tilde\theta_1$ to compute the DP
test statistic and $\tilde\theta_0$ to simulate the null.
$P_{\tilde\theta_0}$ represents the distribution from which we will sample from
to simulate the null distribution. When $(X, y)\sim P_{\tilde\theta_0}$ for
$\theta_0\in\Omega_0$ and we set $\tilde\theta_1 = \tilde\theta_1(X, Y)$ and sample
$(X', y')\sim P_{\tilde\theta_0}$, then 
$\hat\theta_1((X', y'))$ has approximately the same distribution as
$\hat\theta_1((X, y))$.

\begin{algorithm}
\KwData{$X\in\reals^{n\times p}; Y\in\reals^n$}
\KwIn{$n\text{ (dataset size)}; \rho\text{ (privacy-loss parameter)}; \alpha\text{ (target significance)}; T\text{ (test statistic)}$}

$(\tilde\theta_0, \tilde\theta_1) = \dpstats(X, Y, n, \rho)$

\If {$\tilde\theta_0 =  \tilde\theta_1 = \perp$} {
 \Return Fail to Reject the null
}

\

// non-DP test statistic applied to DP statistics

$\tilde{T} = \tilde{t} = T(\tilde\theta_1)$ 

Select $K > 1/\alpha$

\For {$k=1\ldots K$} {

   $\forall i\in[n],$ ${}^kX_i, {}^ky_i \sim P_{\tilde\theta_0}$
   
   ${}^k{}\tilde\theta_0, {}^k\tilde\theta_1 = \dpstats({}^kX_i, {}^ky_i, n, \rho)$
   
   \

   Obtain $t_k$ from $T({}^k\tilde\theta_1)$
}

\

Sort $t_{(1)} \leq\cdots\leq t_{(K)}$

\

Set $r = \lceil (K+1)(1-\alpha)\rceil$

\If { $\tilde{t} > t_{(r)}$ } {
 \Return Reject the null
} \Else {
 \Return Fail to Reject the null
}

\caption{Monte Carlo DP Test Framework.}
\label{alg:mct}
\end{algorithm}

\subsection{Testing a Linear Relationship}
\label{sec:dplinear}

We now discuss our $F$-statistic and Bernoulli testing
approaches.

\subsubsection{$F$-Statistic}

For testing a linear relationship in simple linear regression models, recall that in the
non-private case, we had
$$
T(X, Y, \hat\beta, \hat\beta^N, n, r, q) = \left(\frac{n-r}{r-q}\right)\frac{\norm{X\hat\beta - X\hat\beta^N}^2}{\norm{Y - X\hat\beta}^2}.
$$

Accordingly, we define and compute
$\tilde{T}_L(X, Y, \hat\beta, \hat\beta^N, n, r, q, \rho, \Delta) = \tilde{t}$,
a private estimate of $T(X, Y, \hat\beta, \hat\beta^N, n, r, q)$.
In Algorithm~\ref{alg:t1}, we give the full
$\rho$-zCDP procedure for computing all
necessary sufficient statistics to compute
$\tilde{T}_L(X, Y, \hat\beta, \hat\beta^N, n, r, q, \rho, \Delta)$.

The DP estimate of $S^2$, $\widetilde{S^2}$,
can be computed as
$\widetilde{S^2} = \frac{\sum_{i=1}^n(y_i - \tilde\beta_2 - \tilde\beta_1x_i)^2]^{\Delta^2}_{0} + \calN(0, \frac{\Delta^4}{2\rho})}{n-r}$. Another equivalent
way to compute $\widetilde{S^2}$ is to 
compute $\widetilde{\YYm}$ privately and then, together
with the other DP estimates, to compute
$\widetilde{S^2}$.
Note that under the null hypothesis, 
the DP estimate of $S^2$ is
$\widetilde{S_0^2} = \frac{\sum_{i=1}^n(y_i - \tilde\beta_2)^2]^{\Delta^2}_{0} + \calN(0, \frac{\Delta^4}{2\rho})}{n-r}$ which can also be
equivalently computed by using
$\tilde{\bar{y}}, \widetilde{\YYm}, \tilde{\beta}_2$.
Also, we return
$(\tilde\theta_0, \tilde\theta_1) =  (\perp, \perp)$ if
the computed DP sufficient statistics cannot be used to
simulate the null distribution.

\begin{algorithm}
\KwData{$X\in\reals^{n\times 2}; Y\in\reals^n$}
\KwIn{integer $n\geq 2$; $r, q\in\naturals$; $\rho > 0, \Delta > 0$}

\

Set $\rho = \rho/5$
and compute the following:

\

\begin{enumerate}

\item $\tilde{\bar{x}} = \frac{1}{n}\sum_{i=1}^nx_i]^\Delta_{-\Delta} + \calN(0, \frac{2\Delta^2}{\rho n^2})$.

\item $\tilde{\bar{y}} = \frac{1}{n}\sum_{i=1}^ny_i]_{-\Delta}^\Delta + \calN(0, \frac{2\Delta^2}{\rho n^2})$.

\item $\widetilde\XXm = \frac{1}{n}\sum_{i=1}^nx_i^2]_{0}^{\Delta^2} + \calN(0, \frac{\Delta^4}{2\rho n^2})$.

\item $\widetilde\XYm = \frac{1}{n}\sum_{i=1}^nx_iy_i]_{-\Delta^2}^{\Delta^2} + \calN(0, \frac{2\Delta^4}{\rho n^2})$.

\item
$\widetilde{\YYm} = \frac{1}{n}\sum_{i=1}^ny_i^2]_{0}^{\Delta^2} + \calN(0, \frac{\Delta^4}{2\rho n^2})$.

\item
$$
\tilde\beta_1 = \frac{\widetilde{\XYm}-\tilde{\bar{x}}\tilde{\bar{y}}}{\widetilde\XXm - \tilde{\bar{x}}^2}, \quad\tilde\beta_2 = \frac{\tilde{\bar{y}}\cdot\widetilde{\XXm} - \tilde{\bar{x}}\cdot\widetilde{\XYm}}{\widetilde{\XXm} - \tilde{\bar{x}}^2}.
$$

\item

$$
\widetilde{S_0^2} = \frac{n\widetilde{\YYm}-2\tilde\beta_2n\tilde{\bar{y}} + n\tilde\beta^2_2}{n-r}.
$$

$$
\widetilde{S^2} = \frac{n\widetilde{\YYm}-2\tilde\beta_2n\tilde{\bar{y}}-2\tilde\beta_1n\widetilde{\XYm}+n\tilde\beta^2_2+2\tilde\beta_1\tilde\beta_2n\tilde{\bar{x}} +\tilde\beta_1^2n\widetilde{\XYm}}{n-r}.
$$

\end{enumerate}

\

$(\tilde\theta_0, \tilde\theta_1) =  (\perp, \perp)$
 
\If {$\min(\widetilde{S_0^2}, (n\widetilde\XXm - n\tilde{\bar{x}}^2)/(n-1) ) > 0$} {
$\tilde\theta_0 = (\tilde{\beta}_2, \tilde{\bar{x}}, \widetilde\XXm, \widetilde{S_0^2}, n)$

$\tilde\theta_1 = (\tilde{\beta}_1, \tilde{\bar{x}}, \widetilde{\XXm}, \widetilde{S^2}, n)$
}

\

\Return ($\tilde\theta_0, \tilde\theta_1$)

\caption{$\rho$-zCDP procedure $\dpstats_L$}
\label{alg:t1}
\end{algorithm}

\begin{lemma}
For any $\rho, \Delta > 0$, Algorithm~\ref{alg:t1} 
satisfies $\rho$-zCDP.
\end{lemma}

\begin{proof}

This follows from Proposition 1.6 in~\citep{BunS16} (use of the
Gaussian Mechanism). Next,
we apply composition and 
post-processing (Lemmas 1.7 and 1.8 in~\citep{BunS16}).
The computation of the following statistics is each done to
satisfy $\rho/5$-zCDP:
$\tilde{\bar{x}},\tilde{\bar{y}}, \widetilde{\XXm}, \widetilde{\XYm},  \widetilde{\YYm}$. $\tilde{\beta}_1, \tilde{\beta}_2, \widetilde{S^2}, \widetilde{S_0^2}$ are
post-processing of the other DP releases.

As a result, the entire procedure satisfies $\rho$-zCDP.

\end{proof}

\paragraph{Instantiating Algorithm~\ref{alg:mct} for the Linear Tester:}
If the procedure $\dpstats_L$ returns $(\perp, \perp)$, then we fail to reject the
null. Otherwise, we use the returned statistics
$\tilde\theta_1 = (\tilde{\beta}_1, \tilde{\bar{x}}, \widetilde{\XXm}, \widetilde{S^2}, n)$ to create the
test statistic
$T(\tilde\theta_1) = \frac{\tilde\beta_1^2\cdot n\cdot (\widetilde{\XXm} - \tilde{\bar{x}}^2)}{\widetilde{S^2}}$
and use
$\tilde\theta_0 = (\tilde{\bar{y}}, \tilde{\bar{x}}, \widetilde{\XXm}, \widetilde{S_0^2}, n)$ to simulate the
null distributions (to decide to reject or fail to reject the null
hypothesis).
$P_{\tilde\theta_0}$ is instantiated as a normal distribution and used to
generate ${}^kx_i$ distributed as
$\calN(\tilde{\bar{x}}, (n\widetilde\XXm - n\tilde{\bar{x}}^2)/(n-1))$
and ${}^ky_i$ as
$\tilde{\beta}_2 + e_i$,
$e_i \sim\calN(0, \widetilde{S_0^2})$ for all $i\in[n]$.

\subsubsection{Bernoulli Tester}

Next, we define an approach for testing a linear relationship
via Bernoulli testing,
inspired by the DP regression estimators
of~\citep{DworkL09, AMSSV20}.

Under the null hypothesis (slope is $\beta_1=0$), observe that, under
the general linear model, if we pair datapoints
$(x_i, y_i), (x_{i+1}, y_{i+1})$ and
calculate the sign at the slope of the line between
the datapoints, we have
$$
\ind\left\{\frac{y_{i+1}-y_i}{x_{i+1}-x_i} > 0\right\}
= \ind\left\{\beta_1 + \frac{e_{i+1}-e_i}{x_{i+1}-x_i} > 0\right\} \sim
\Bern(1/2),
$$
provided that $x_{i+1}\neq x_i$ (if $x_{i+1}=x_i$, then 
we set the result to a random
$\Bern(1/2)$).
Note that this holds for any continuous distribution 
for the $e_i$'s, not necessarily normal.
There are simple DP tests to determine whether
$p=1/2$ given a dataset drawn from $\Bern(p)$ by computing
a noisy sum of the values in the dataset and comparing it
to a noisy threshold (determined by a normal approximation
to a binomial distribution).
The DP regression estimators
of~\citep{DworkL09, AMSSV20} also calculate the slopes
between pairs of points, but then outputs a 
DP~\textit{median} of the results.



The resulting algorithm for privately testing a linear relationship is
shown in Algorithm~\ref{alg:pbinomial}.
We first group the points into $\lfloor n/2\rfloor$
pairs. Then we calculate
$s$, the number of slopes that are positive. We add noise to this
estimate to satisfy $\rho$-zCDP and then use the noisy estimate
to decide to reject the null.
The noise to satisfy zCDP is $\calN(0, \frac{1}{2\rho})$
whereas a normal approximation to
$\Bin(n_s, 1/2)$ is $\calN(n_s, n_s/4)$. 
As a result,
we can reject the null hypothesis iff the DP observed number of
1s is not in the $(\alpha/2, 1-\alpha/2)$ quantiles of
$\calN(\frac{n_s}{2}, \frac{n_s}{4} + \frac{1}{2\rho})$ where
$\alpha$ is the target significance level.

\begin{algorithm}
\KwData{$X\in\reals^{n\times 2}; Y\in\reals^n$}
\KwIn{$n\in\mathbb{N}$; $\rho > 0$}

Let $x_1, \ldots, x_n$ be the observed 1-D independent variables from $X$

Let $\tau: [n]\rightarrow[n]$ be a randomly chosen permutation

$s = 0$

$n_s = \lfloor n/2\rfloor$

\For{$i=1 \ldots n_s$} {

\If {$x_{\tau(n_s + i)} - x_{\tau(i)} \neq 0$} {
   $s = s + \ind\left\{\frac{Y_{\tau(n_s + i)} - Y_{\tau(i)}}{x_{\tau(n_s + i)} - x_{\tau(i)}} > 0\right\}$
} \Else {
   $r \sim\Bern(1/2)$
   
   $s = s + r$
}

}

$s = s + \calN(0, \frac{1}{2\rho})$

Let $N_{\alpha/2}$ and $N_{1-\alpha/2}$ denote the
$\alpha/2$ and $1-\alpha/2$ quantiles of $\calN(n_s/2, \frac{n_s}{4} + \frac{1}{2\rho})$ respectively

\If { $s \notin (N_{\alpha/2}, N_{1-\alpha/2})$ } {
 \Return Reject the null
} \Else {
 \Return Fail to Reject the null
}

\caption{$\rho$-zCDP procedure $\dpbern$}
\label{alg:pbinomial}
\end{algorithm}

\begin{lemma}
For any $\rho> 0$ and $n\in\mathbb{N}$, 
Algorithm~\ref{alg:pbinomial} satisfies $\rho$-zCDP.
\end{lemma}

\begin{proof}

Algorithm~\ref{alg:pbinomial}
pairs the points into $n_s=\lfloor n/2\rfloor$ 
pairs. Note that this is a
1-Lipschitz transformation
(i.e., changing one datapoint will change a single slope 
estimate) so that the differential privacy
guarantees are preserved (by Definition 16 and Lemma 17
of~\citep{AMSSV20}).

A single datapoint $(x_i, y_i)$ can affect the
sum $s$ by at most 1 so the
global sensitivity is 1.
By post-processing and by
Proposition 1.6 in~\citep{BunS16}, 
via the use of the Gaussian
Mechanism, Algorithm~\ref{alg:pbinomial} satisfies $\rho$-zCDP.
\end{proof}

We note that a uniformly most powerful DP Bernoulli tester
has been designed in~\citep{Awan_Slavkovic_2020}. Using this
tester might yield better power than the tester
described in Algorithm~\ref{alg:pbinomial}. However, that work
is for pure DP (whereas we use zCDP) and the test is
computationally slower.

\subsection{Testing Mixture Models}
\label{sec:dpmixture}

As we will show experimentally, the best
framework for the mixture model test depends on
the properties of the dataset. This can be
seen as conditional inference~\citep{NBERw25456}.
We now discuss our $F$-statistic and Kruskal-Wallis approaches.

\subsubsection{$F$-Statistic}

In the non-private case, we can use the following test statistic
for testing mixtures in simple linear regression models:
\begin{align*}
T(X, Y, \hat\beta, \hat\beta^N, n, r, q) &= \left(\frac{n-r}{r-q}\right)\frac{\norm{X\hat\beta - X\hat\beta^N}^2}{\norm{Y - X\hat\beta}^2} \\
&= \frac{\XXA\XXB}{S^2\XX}(\hat\beta_1 - \hat\beta_2)^2.
\end{align*}

\begin{algorithm}
\KwData{$X\in\reals^{n\times 2}, Y\in\reals^n$}
\KwIn{integer $n_1, n\geq 2$; $r, q\in\naturals$; $\rho > 0, \Delta > 0$}

Set $\rho = \rho/8$ and $n_2 = n-n_1$. Then 
compute the following:

\begin{enumerate}

\item $\tilde{\bar{x}}_1 = \frac{1}{n_1}\sum_{i=1}^{n_1}x_i]^\Delta_{-\Delta} + \calN(0, \frac{2\Delta^2}{\rho n_1^2})$,
$\tilde{\bar{x}}_2 = \frac{1}{n_2}\sum_{i=n_1+1}^nx_i]^\Delta_{-\Delta} + \calN(0, \frac{2\Delta^2}{\rho n_2^2})$,
$\tilde{\bar{x}} = n_1/n\cdot\tilde{\bar{x}}_1+  n_2/n\cdot\tilde{\bar{x}}_2$.

\item $\widetilde\XXAm = \frac{1}{n_1}\sum_{i=1}^{n_1}x_i^2]^{\Delta^2}_0 + \calN(0, \frac{\Delta^4}{2\rho n_1^2})$, 
$\widetilde\XXBm = \frac{1}{n_2}\sum_{i=n_1+1}^{n}x_i^2]^{\Delta^2}_0 + \calN(0, \frac{\Delta^4}{2\rho n_2^2})$,
$\widetilde\XXm = n_1/n\cdot\widetilde\XXAm + n_2/n\cdot\widetilde\XXBm$.

\item $\widetilde\XYAm = \frac{1}{n_1}\sum_{i=1}^{n_1}x_iy_i]_{-\Delta^2}^{\Delta^2} + \calN(0, \frac{2\Delta^4}{\rho n_1^2})$,
$\widetilde\XYBm = \frac{1}{n_2}\sum_{i=n_1+1}^nx_iy_i]_{-\Delta^2}^{\Delta^2} + \calN(0, \frac{2\Delta^4}{\rho n_2^2})$,
$\widetilde\XYm = n_1/n\cdot\widetilde\XYAm + n_2/n\cdot\widetilde\XYBm$.

\item $\widetilde\YYAm = \frac{1}{n_1}\sum_{i=1}^{n_1}y_i^2]_{0}^{\Delta^2} + \calN(0, \frac{\Delta^4}{2\rho n_1^2})$,
$\widetilde\YYBm = \frac{1}{n_2}\sum_{i=n_1+1}^ny_i^2]_{0}^{\Delta^2} + \calN(0, \frac{\Delta^4}{2\rho n_2^2})$,
$\widetilde{\YYm} = n_1/n\cdot\widetilde\YYAm + n_2/n\cdot\widetilde\YYBm$.

\item $$\tilde\beta_1 = \frac{\widetilde\XYAm}{\widetilde\XXAm}, \quad\tilde\beta_2 = \frac{\widetilde\XYBm}{\widetilde\XXBm},
\quad\tilde\beta = n_1/n\cdot\tilde\beta_1 + n_2/n\cdot\tilde\beta_2.$$

\item 
$$
\widetilde{S_0^2} = \frac{n\widetilde{\YYm} + n\tilde\beta^2 - 2n\widetilde{\XYm}\tilde\beta}{n-r}.
$$

$$
\widetilde{S^2} =\frac{n_1\widetilde{\YYAm} + n_1\tilde\beta_1^2 - 2n_1\widetilde{\XYAm}\tilde\beta_1 + n_2\widetilde{\YYBm} + n_2\tilde\beta_2^2 - 2n_2\widetilde{\XYBm}\tilde\beta_2}{n-r}.
$$

\end{enumerate}

\

$(\tilde\theta_0, \tilde\theta_1) =  (\perp, \perp)$

\If {$\min(\widetilde{S_0^2}, (n\widetilde\XXm - n\tilde{\bar{x}}^2)/(n-1) ) > 0$} {
$\tilde\theta_0 = (\tilde\beta_1, \tilde{\bar{x}}, \widetilde\XXm, \widetilde{S_0^2}, n_1, n_2, n)$ 

$\tilde\theta_1 = (\tilde\beta_1, \tilde\beta_2, \widetilde\XXAm, \widetilde\XXBm, \widetilde\XXm, \widetilde{S^2}, n_1, n_2, n)$
}

\

\Return ($\tilde\theta_0, \tilde\theta_1$)

\caption{$\rho$-zCDP procedure $\dpstats_M$}
\label{alg:t2}
\end{algorithm}

In Algorithm~\ref{alg:t2}, we apply the Gaussian mechanism
to calculate the DP sufficient statistics.
$\widetilde{S_0^2}$ and
$\widetilde{S^2}$ are DP estimates of the sampling error
under the null and alternative hypothesis, respectively.
In particular, $\widetilde{S_0^2}$ corresponds to an estimate
of the sampling error when the groups have the same
distributional properties.

\begin{lemma}
For any $\rho, \Delta > 0$, Algorithm~\ref{alg:t2} satisfies $\rho$-zCDP.
\end{lemma}

\begin{proof}
This follows from Proposition 1.6 in~\citep{BunS16} via the
use of the
Gaussian Mechanism to ensure zCDP.

The composition and 
post-processing properties of zCDP
(Lemmas 1.7 and 1.8 in~\citep{BunS16}) can then be applied.
The computation of the following statistics is each done to
satisfy $\rho/8$-zCDP:
$\tilde{\bar{x}}_1, \tilde{\bar{x}}_2, \widetilde\XXAm, \widetilde\XXBm, \widetilde\XYAm, \widetilde\XYBm, \widetilde{\YYAm}, \widetilde{\YYBm}$.
The other statistics computed
are post-processed DP releases.

As a result, the entire procedure satisfies $\rho$-zCDP.

\end{proof}

\paragraph{Instantiating Algorithm~\ref{alg:mct} for the $F$-statistic Mixture Tester:}
If the procedure $\dpstats_M$ returns $(\perp, \perp)$, then we fail to reject the
null. Otherwise, we use the returned statistics\\
$\tilde\theta_1 = (\tilde\beta_1, \tilde\beta_2, \widetilde\XXAm, \widetilde\XXBm, \widetilde\XXm, \widetilde{S^2}, n_1, n_2, n)$ to create the
test statistic and use\\
$\tilde\theta_0 = (\tilde\beta_1, \tilde{\bar{x}}, \widetilde\XXm, \widetilde{S_0^2}, n_1, n_2, n)$
to simulate the
null distributions.
$P_{\tilde\theta_0}$ is instantiated as a normal distribution
and used to generate ${}^kx_i$ distributed 
as $\calN(\tilde{\bar{x}}, 
   (n\widetilde\XXm - n\tilde{\bar{x}}_1^2)/(n-1))$
and to generate ${}^ky_i$ distributed as
$\tilde\beta_1{}^kx_i + e_i$,
$e_i \sim\calN(0, \widetilde{S_0^2})$ for either group 1 with size $n_1$ or
group 2 with size $n-n_1$.

\subsubsection{Nonparametric Tests via Kruskal-Wallis}

Couch, Kazan, Shi, Bray, and Groce~\citep{CouchKSBG19} present
DP analogues of
nonparametric hypothesis testing 
methods (which require little or no distributional
assumptions). They find that the DP variant
of the Kruskal-Wallis test statistic
is more powerful than the DP version of the
traditional parametric statistics for testing if two groups
have the same medians. 
Here, we reduce our problem of testing mixture models to their problem
of testing if groups share the
same median. The reduction is as follows:
Given two datasets $(\bx^1, \by^1)$ and $(\bx^2, \by^2)$, each of size
$n_1$ and $n_2$ respectively, we wish to test if
the slopes are equal. We randomly match all pairs of points in $(\bx^1, \by^1)$
and obtain at most $n_1/2$ slopes in $s_1$. We do the same for the 
second group $(\bx^2, \by^2)$
to obtain $n_2/2$ slopes in $s_2$. Then we compute the mean of ranks of elements in $s_1$ and $s_2$ as
$\bar{r}_1$ and $\bar{r}_2$ respectively. Next, we compute the
Kruskal-Wallis absolute value test statistic $h$ from
~\citep{CouchKSBG19} and release a perturbed
version satisfying zCDP.
We can use the
Monte Carlo testing framework in Algorithm~\ref{alg:mct} and
use Algorithm~\ref{alg:t3} to compute the test statistics.
Under the null, the slopes in $s_1$ and $s_2$ would have 
similar ranks so we choose uniform random numbers in some interval.
We then decide to reject or fail to reject the null, based on the distribution
of test statistics obtained via this process and its relation to the statistic
computed on the observed data.

\begin{lemma}
For any $\rho> 0$ and even $n$, 
Algorithm~\ref{alg:t3} satisfies $\rho$-zCDP.
\end{lemma}

\begin{proof}

Algorithm~\ref{alg:t3} randomly pairs all $n_1$ pairs of
points in group 1 (to obtain slopes $s_1$ of size $n_1/2$)
and pairs all $n_2$ pairs in group 2 (to obtain 
slopes $s_2$ of size $n_2/2$). Note that this is a
1-Lipschitz transformation so that the differential privacy
guarantees are preserved (by Definition 16 and Lemma 17
of~\citep{AMSSV20}).

Then we proceed to use the DP Kruskal-Wallis absolute value
test statistic with sensitivity of 8 (as shown in Theorem 3.4 of~\citep{CouchKSBG19}).
By Proposition 1.6 in~\citep{BunS16}, via the use of the Gaussian
Mechanism, the procedure satisfies $\rho$-zCDP.
\end{proof}

\begin{algorithm}
\KwData{$X\in\reals^{n\times 2}; Y\in\reals^n$}
\KwIn{Even $n_1, n\in\mathbb{N}$; $\rho > 0$}

Let $x_1, \ldots, x_n$ be the observed 1-D independent variables from $X$

Let $\tau: [n]\rightarrow[n]$ be a randomly chosen permutation

$s_1 = \{\}$

\For{$i=1 \ldots n_1/2$} {

   $s_1 = s_1\,\,\bigcup\,\,\{\frac{Y_{\tau(n_1/2 + i)} - Y_{\tau(i)}}{x_{\tau(n_1/2 + i)} - x_{\tau(i)}}\}$

}

$n_2 = n - n_1$

$s_2 = \{\}$

$e = n_1 + n_2/2$

\For{$i=n_1+1 \ldots e$} {

   $s_2 = s_2\,\,\bigcup\,\,\{\frac{Y_{\tau(e + i)} - Y_{\tau(i)}}{x_{\tau(e + i)} - x_{\tau(i)}}\}$

}

Let $r:\reals^{m}\rightarrow[m]$ be rank-computing function on any $m$ elements

Compute $s$ by appending (in an order-preserving manner) $s_2$ to $s_1$

Compute $\bar{r}_1$, mean of ranks of $s_1$ in $r(s)$

Compute $\bar{r}_2$, mean of ranks of $s_2$ in $r(s)$

Compute $h = \frac{4(n-1)}{n^2}\left(n_1|\bar{r}_1-\frac{n+1}{2}|+n_2|\bar{r}_2-\frac{n+1}{2}|\right)$

\Return null, $h + \calN(0, 8^2/(2\rho))$

\caption{$\rho$-zCDP procedure $\dpkw$}
\label{alg:t3}
\end{algorithm}

\paragraph{Instantiating Algorithm~\ref{alg:mct} for the Kruskal-Wallis Mixture Tester:}
We use the returned statistic
$\tilde\theta_1 = (h)$ as the
sole statistic. In this case, $T$ is the identity function.
$\tilde\theta_0$ is taken to be null.
$P_{\tilde\theta_0}$ generates ${}^kx_i$ and ${}^ky_i$ (for the 2 groups) 
independently and uniformly at random in a fixed interval
(say $[-5, 5]$). Although this distribution may be very different from the actual
data distribution, the distribution of ranks of the slopes will be identical to that
under the null, ensuring that $T(({}^kx, {}^ky))$ has the right distribution.

\section{Differentially Private $F$-Statistic}
\label{sec:dpf}

In this section, we will
show that the DP $F$-statistic converges, 
in distribution,
to the asymptotic distribution of the $F$-statistic.
The focus will be on showing results for
Algorithm~\ref{alg:t1} but a similar route can be used to obtain
analogous results for Algorithm~\ref{alg:t2}.
Recall that Algorithm~\ref{alg:t1} is an instantiation of the
DP $F$-statistic for testing a linear relationship while
Algorithm~\ref{alg:t2} is for testing mixtures.

$T_n$ is the non-private $F$-statistic while
$\tilde{T}_n$ is the DP $F$-statistic
constructed from DP sufficient statistics
obtained via Algorithm~\ref{alg:t1}.
The main theorem in this section is Theorem~\ref{thm:dpf},
which shows the convergence, in distribution,
of $\tilde{T}_n$ to the asymptotic distribution of $T_n$,
the chi-squared distribution.
As a corollary, the statistical power of
$\tilde{T}_n$ converges to the statistical power of
$T_n$.
While Theorem~\ref{thm:dpf} is specialized to the simple
linear regression setting (i.e., $p = 2$), it can easily
be extended to multiple linear regression. 

\begin{theorem}
Let $\sigma_e > 0$, $r=p=2, q=1$, 
and $\beta\in\reals^p$. For every $n\in\naturals$ with
$n > r$, let
$X_n\in\reals^{n\times p}$ be the
design matrix where the first column is an
all-ones vector and the second column is
$(x_1, \ldots, x_n)^T$.
Let $\Delta = \Delta_n > 0$ be a sequence of
clipping bounds, $\rho = \rho_n > 0$ be a
sequence of privacy parameters, and
$\eta_n^2 = \frac{\norm{X_n\beta - X_n\beta^N}^2}{\sigma_e^2}$.
Under the general linear model (GLM),
$Y_n\sim\calN(X_n\beta, \sigma^2_eI_{n\times n})$.
Let $\tilde\beta$ and $\tilde\beta^N$ be the
DP least-squares
estimate of $\beta$, obtained in Algorithm~\ref{alg:t1}, under the alternative 
and null hypotheses, respectively.
Let $\tilde{T} = \tilde{T}_n$ be the DP $F$-statistic 
computed from DP sufficient statistics
via Algorithm~\ref{alg:t1} and Equation~(\ref{eq:glrt}). 
Suppose the following conditions hold:
\begin{enumerate}
\item $\exists c_x, c_{x^2}\in\reals$ such that
$\bar{x}\rightarrow c_x$,
$\XXm\rightarrow c_{x^2}$,
and
$c_{x^2} > c_x^2$,
\item $\exists \eta\in\reals$ such that
$\eta_n^2\rightarrow\eta^2$,
\item $\frac{\Delta_n^2}{\rho_n n}$,
$\frac{\Delta_n^4}{\rho_n n}\rightarrow 0$,
\item $\pr[\exists i\in[n], y_i\notin[-\Delta_n, \Delta_n]]\rightarrow 0$ and $\forall i\in[n], x_i\in[-\Delta_n, \Delta_n]$.
\end{enumerate}
Then we obtain the following results:
\begin{enumerate}
\item 
Under the null hypothesis:
$\tilde{\beta}^N = \tilde{\beta}_n^N \xrightarrow{P}\beta$,
\item
Under the alternative hypothesis:
$\tilde{\beta} = \tilde{\beta}_n \xrightarrow{P}\beta$,
\item
$\tilde{T} = \tilde{T}_n\xrightarrow{D}
\frac{\chi^2_{r-q}(\eta^2)}{r-q}$.
\end{enumerate}
\label{thm:dpf}
\end{theorem}

The condition that
$\pr[\exists i\in[n], y_i\notin[-\Delta_n, \Delta_n]]\rightarrow 0$ (Condition 4 in Theorem~\ref{thm:dpf}),
holds by a Gaussian tail bound (Claim~\ref{claim:gausstail}),
if for all $i\in[n]$,
$\Delta_n \geq |\beta_1x_i + \beta_2| + \sigma_e\sqrt{\log 2n^{O(1)}}$.

First, we will prove convergence results for sufficient
statistics used to construct the non-private
$F$-statistic $T_n$, in our setting. Then we will show
convergence
results for DP sufficient statistics
used to construct the private $F$-statistic $\tilde{T}_n$.
Finally, we will combine these previous results to show
Theorem~\ref{thm:dpf}.

\subsection{Convergence of Non-private Sufficient Statistics}

In Equation~(\ref{eq:glrt}), the non-private $F$-statistic
is given as
$$
T = T_n = \frac{n-r}{r-q}\cdot\frac{\norm{X\hat\beta - X\hat\beta^N}^2}{\norm{Y-X\hat\beta}^2} =
\frac{n-r}{r-q}\cdot\frac{\norm{X_n\hat\beta - X_n\hat\beta^N}^2}{\norm{Y_n-X_n\hat\beta}^2}.
$$

We start by writing this $F$-statistic, in an equivalent form,
in terms of quantities that we will show are convergent:

\begin{lemma}
Suppose that $\sigma_e > 0$, $p\in\naturals$, 
and $\beta\in\reals^p$.
Let $X = X_n\in\reals^{n\times p}$ be the
full-rank design matrix 
(as in Equation~(\ref{eq:design})) and
$Y = Y_n\sim\calN(X_n\beta, \sigma_e^2I_{n\times n})$.
Also, let $\hat\beta$ and $\hat\beta^N$ be the non-private
least-squares
estimate of $\beta$ under the alternative and null hypotheses,
respectively.

Define the following quantities:
$$
\hat{E}_n = \left(\frac{X_n^TX_n}{n}\right)^{1/2}\in\reals^{p\times p},\quad \hat{F}_n = \frac{X_n^TY_n}{n}\in\reals^p,\quad \hat{G}_n = \frac{Y_n^TY_n}{n}\in\reals.
$$

Then the test statistic
$T_n$ from
Equation~(\ref{eq:glrt}) can be 
re-written as
\begin{equation}
T_n = \frac{n-r}{r-q}\cdot\frac{\norm{\sqrt{n}\hat{E}_n(\hat\beta - \hat\beta^N)}^2}{n(\hat\beta^T\hat{E}_n^2\hat\beta - 2\hat\beta^T\hat{F}_n + \hat{G}_n)},
\label{eq:f}
\end{equation}
for any $n, r, q\in\naturals$ such that $q < r$.

\label{lem:fequiv}
\end{lemma}

\begin{proof}[Proof of Lemma~\ref{lem:fequiv}]

First, note that $\hat{E}_n\in\reals^{p\times p}$
(i)
exists because
$X_n^TX_n$
is positive definite,
(ii) is unique since its square is positive
definite~\citep{horn_johnson_2012}.

\begin{align*}
\norm{X_n\hat\beta-X_n\hat\beta^N}^2
&= (\hat\beta - \hat\beta^N)^TX_n^TX_n(\hat\beta - \hat\beta^N) \\
&= \sqrt{n}(\hat\beta - \hat\beta^N)^T\hat{E}_n^T\sqrt{n}\hat{E}_n(\hat\beta - \hat\beta^N) \\
&= \norm{\sqrt{n}\hat{E}_n(\hat\beta - \hat\beta^N)}^2.
\end{align*}

Next,
\begin{align*}
\norm{Y_n-X_n\hat\beta}^2
&= (Y_n-X_n\hat\beta)^T(Y_n-X_n\hat\beta) \\
&= Y_n^TY_n - Y_n^TX_n\hat\beta - \hat\beta^TX_n^TY_n + \hat\beta^TX_n^TX_n\hat\beta \\
&= Y_n^TY_n + \hat\beta^TX_n^TX_n\hat\beta - 2\hat\beta^TX_n^TY_n \\
&= n(\hat\beta^T\hat{E}_n^2\hat\beta - 2\hat\beta^T\hat{F}_n + \hat{G}_n).
\end{align*}

\end{proof}

It will be easier to use 
Equation~(\ref{eq:f}) as an equivalent form
of the $F$-statistic to prove convergence results.
An analogous representation will be used to
prove convergence results for the DP $F$-statistic.

In the case of testing a linear relationship
(as in Section~\ref{sec:testlinear}) in simple linear
regression (i.e., where $p = 2$ and the columns of $X$ are the all-ones vector and
$(x_1, \ldots, x_n)^T$),
$$
X^TX = \begin{pmatrix}
n & n\bar{x}\\
n\bar{x} & \XX\\
\end{pmatrix},\quad
X^TY = \begin{pmatrix}
n\bar{y}\\
\XY\\
\end{pmatrix},\quad
Y^TY = \sum_{i=1}^ny_i^2,
$$
$$
\hat\beta = \begin{pmatrix}
\hat\beta_2\\
\hat\beta_1\\
\end{pmatrix},\quad
\hat\beta^N =
\begin{pmatrix}\bar{y}\\0\end{pmatrix},
$$
$$
\widehat{\sigma^2_x} = \XXm-\bar{x}^2.
$$

In this case,
it
can be verified that
$\hat{E}_n, \hat{F}_n, \hat{G}_n$ is:
\begin{align}
\hat{E}_n &=
\frac{1}{\sqrt{\XXm+1+2\sqrt{\XXm-\bar{x}^2}}}
\begin{pmatrix}
1 + \sqrt{\XXm-\bar{x}^2} & \bar{x}\\
\bar{x} & \XXm+\sqrt{\XXm-\bar{x}^2}
\end{pmatrix},\nonumber\\
 &=
\frac{1}{\sqrt{\XXm+1+2\sqrt{\widehat{\sigma^2_x}}}}
\begin{pmatrix}
1 + \sqrt{\widehat{\sigma^2_x}} & \bar{x}\\
\bar{x} & \XXm+\sqrt{\widehat{\sigma^2_x}}
\end{pmatrix},\nonumber\\
\hat{F}_n &= \frac{X^TY}{n}
= \begin{pmatrix}
\bar{y}\\
\XYm
\end{pmatrix},
\quad \hat{G}_n = \frac{Y^TY}{n} \defeq \YYm.
\label{eq:efg}
\end{align}

Next, we proceed to
show non-private convergence results
that will be pivotal to our final result. We will
crucially rely on the Gaussian tail bound, the normality of
$\hat\beta, \hat\beta^N$, and
Corollary~\ref{cor:gausstail}.

\begin{lemma}
For every sequence of clipping bounds
$\Delta = \Delta_n > 0$ and
sequence of privacy parameters
$\rho = \rho_n > 0$, under the conditions of
Theorem~\ref{thm:dpf}:
\begin{enumerate}
\item[(1)] $\exists c_y\in\reals$ such that
$\bar{y}\xrightarrow{P} c_y$,
\item[(2)]$\exists c_a, c_{xy}\in\reals$
such that $\XYm\xrightarrow{P}c_{xy}$,
$\XYm-\bar{x}\cdot\bar{y}\xrightarrow{P} c_a$,
\item[(3)] $\exists c_b\neq 0$ such that
$\widehat{\sigma^2_x}\xrightarrow{P} c_b$,
\item[(4)] $\exists$ unique positive-definite $C^{1/2}\in\reals^{2\times 2}$ 
such that
$\hat{E}_n\rightarrow C^{1/2}$,
\item[(5)] $\hat{F}_n \xrightarrow{P} (c_y\quad c_{xy})^T$,
\item[(6)] $\exists c_{y^2}\in\reals$ such that
$\hat{G}_n = \frac{Y^TY}{n} \xrightarrow{P} c_{y^2}$,
\item[(7)] Normality of $\hat\beta$: $\hat\beta \sim \calN\left(\beta, \sigma_e^2(X_n^TX_n)^{-1}\right)$; Consistency of $\hat\beta$: $\hat\beta\xrightarrow{P}\beta$.
\end{enumerate}

\label{lem:nonprivateconv}
\end{lemma}

To prove Lemma~\ref{lem:nonprivateconv}, we will make use of the following tools:
the Gaussian tail bound and Slutsky's Theorem, which we state below:

\begin{claim}[Gaussian Tail Bound]
Let $Z$ be a standard normal random variable with mean 0 and variance 1. i.e., $Z\sim\calN(0, 1)$.
Then
$$
\pr[|Z| > t] \leq 2\exp(-t^2/2),
$$
for every $t > 0$.
\label{claim:gausstail}
\end{claim}

By the Gaussian tail bound, 
any Gaussian
random variable (such as the DP estimates) converges,
in probability, to the asymptotic distributions
of the estimates without Gaussian noise added as long as the
variance goes to 0 (Corollary~\ref{cor:gausstail}):

\begin{corollary}
Let $N_n\sim\calN(0, \sigma_n^2)$ where
$\sigma_n\rightarrow 0$, then
$N_n\xrightarrow{P} 0$.
\label{cor:gausstail}
\end{corollary}

Corollary~\ref{cor:gausstail} follows from the definition
of convergence in probability and the Gaussian tail bound
(Claim~\ref{claim:gausstail}).

Next, we introduce Slutsky's Theorem which
will be crucial to
combining individual convergence results to show more general
results:

\begin{theorem}[Slutsky's Theorem, see~\citep{gut2013probability}]
Let $\{W_n\}, \{Z_n\}$ be a sequence of random vectors
and $W$ be a random vector.
If $W_n\xrightarrow{D}W$ and
$Z_n\xrightarrow{P}c$ for a constant $c\in\reals$,
then as $n\rightarrow\infty$:
\begin{enumerate}
\item $W_n\cdot Z_n\xrightarrow{D} Wc$,
\item $W_n + Z_n\xrightarrow{D}W + c$,
\item $W_n/Z_n \xrightarrow{D} W/c$ as long as
$c\neq 0$.
\end{enumerate}
\label{thm:slutsky}
\end{theorem}

\begin{proof}[Proof of Lemma~\ref{lem:nonprivateconv}]

(1):
By definition,
for all $i\in[n]$,
$y_i\sim\beta_2 + \beta_1x_i + \calN(0, \sigma_e^2)$.
Then,
$\bar{y}\sim\beta_2 + \beta_1\bar{x} + \calN(0, \frac{\sigma_e^2}{n})$.
By Slutsky's 
Theorem and Corollary~\ref{cor:gausstail},
$\bar{y}\xrightarrow{P}\beta_2 + \beta_1c_x \defeq c_y$.
As a result, $\bar{y}\xrightarrow{P} c_y\in\reals$.

(2):
Also,
\begin{align*}
&\XYm = \frac{1}{n}\sum_{i=1}^nx_iy_i \\
&\sim \frac{1}{n}\sum_{i=1}^nx_i\calN\left(\beta_2 + \beta_1x_i, \sigma_e^2\right)\\
&= \frac{1}{n}\sum_{i=1}^n\left(\beta_2x_i + \beta_1x_i^2\right) + \calN\left(0, \frac{1}{n^2}\sum_{i=1}^n\sigma_e^2x_i^2\right).
\end{align*}

From the assumptions of Theorem~\ref{thm:dpf},
$\frac{1}{n}\sum_{i=1}^n(\beta_2x_i + \beta_1x_i^2)
\rightarrow \beta_2c_x+\beta_1c_{x^2}$ and
$\frac{\sigma_e^2}{n^2}\sum_{i=1}^nx_i^2 = \frac{\sigma_e^2}{n}\XXm\rightarrow 0$.
Then by Slutsky's Theorem and Corollary~\ref{cor:gausstail},
$\XYm\xrightarrow{P} \beta_2c_x+\beta_1c_{x^2} \defeq c_{xy}$ and
$\XYm-\bar{x}\cdot\bar{y}\xrightarrow{P} c_{xy}-c_xc_y \defeq c_a$.

(3): By Slutsky's Theorem and the assumptions in
Theorem~\ref{thm:dpf} that
$\XXm\rightarrow{c_{x^2}}$,
$\bar{x}\rightarrow c_x$, $c_{x^2}\neq c_x^2$, we have
that $\exists c_b\neq 0$ such that
$\widehat{\sigma^2_x}\rightarrow c_{x^2}-c_x^2 \defeq c_b$.

(4):
By Equation~(\ref{eq:efg}), we have
\begin{align}
\hat{E}_n &= \frac{1}{\sqrt{\XXm+1+2\sqrt{\widehat{\sigma^2_x}}}}
\begin{pmatrix}
1 + \sqrt{\widehat{\sigma^2_x}} & \bar{x}\\
\bar{x} & \XXm+\sqrt{\widehat{\sigma^2_x}}
\end{pmatrix}  \\
&\rightarrow
\frac{1}{\sqrt{c_{x^2}+1+2\sqrt{c_{x^2}-c_x^2}}}
\begin{pmatrix}
1 + \sqrt{c_{x^2}-c_x^2} & c_x\\
c_x & c_{x^2}+\sqrt{c_{x^2}-c_x^2}
\end{pmatrix} \defeq C^{1/2}.
\end{align}

(5):
Next,
$$
\hat{F}_n = \frac{X^TY}{n}
= \begin{pmatrix}
\bar{y}\\
\XYm
\end{pmatrix}
\xrightarrow{P}
\begin{pmatrix}
c_y\\
c_{xy}
\end{pmatrix}.
$$

(6): By the weak law of large numbers (Lemma~\ref{lem:wlln}),
$\frac{\chi^2_{n}}{n}\xrightarrow{P} 1$. Then,
\begin{align*}
\hat{G}_n &= \frac{\sum_{i=1}^ny_i^2}{n} \\
&\sim \frac{1}{n}\sum_{i=1}^n\left(\beta_2 + \beta_1x_i + \calN(0, \sigma_e^2)\right)^2 \\
&= \beta_2^2 + 2\bar{x}\beta_1\beta_2 + 2\beta_2\calN\left(0, \frac{\sigma_e^2}{n}\right) + \beta_1^2\XXm + 
2\beta_1\bar{x}\calN\left(0, \frac{\sigma_e^2}{n}\right) + \sigma_e^2\frac{\chi^2_{n}}{n} \\
&\xrightarrow{P} \beta_2^2 + 2c_x\beta_1\beta_2 + \beta_1^2c_{x^2} +  \sigma_e^2 \\
&\defeq c_{y^2},
\end{align*}
via the use of Slutsky's Theorem,
Corollary~\ref{cor:gausstail}, and assumptions that
$\bar{x}\rightarrow c_x, \XXm\rightarrow c_{x^2}$.

(7):
First, we recall that $\hat\beta$, the non-private OLS estimate,
is Gaussian and centered at $\beta$:

$$
\hat\beta \sim \calN\left(\beta, \sigma_e^2(X^TX)^{-1}\right) = \calN\left(\beta, \sigma_e^2\hat{E}_n^{-2}/n\right),
$$

This follows from 
Equations (3.9) and (3.10) in~\citep{HastieTF09} for any design matrix 
$X\in\reals^{n\times 2}$.

Since $\hat{E}_n\rightarrow C^{1/2}$, it follows that
$\hat{E}_n^{-2}/n\rightarrow 0$ so that
by Corollary~\ref{cor:gausstail},
$\hat\beta\xrightarrow{P}\beta$.

\end{proof}

Now,
we will show that the DP statistics converge, either in probability
or distribution, to the distributions of their corresponding
non-DP statistics.

\subsection{Convergence of Differentially Private Sufficient Statistics}

The DP $F$-statistic is constructed via
Algorithm~\ref{alg:t1} and Equation~(\ref{eq:glrt}).
We start by rewriting the DP $F$-statistic analogously to
Lemma~\ref{lem:fequiv}:

\begin{lemma}
Suppose that $\sigma_e > 0$, $p\in\naturals$, 
and $\beta\in\reals^p$.
Let $X = X_n\in\reals^{n\times p}$ be the
full-rank design matrix 
(as in Equation~(\ref{eq:design})) and
$Y = Y_n\sim\calN(X_n\beta, \sigma_e^2I_{n\times n})$.
Also, let 
$\tilde{\bar{x}}, \tilde{\bar{y}}$,
$\widetilde{\XXm}, \widetilde{\XYm}$,
$\widetilde{\YYm}$,
$\tilde\beta_1$, and $\tilde\beta_2$
be as computed in Algorithm~\ref{alg:t1}.

Define the following quantities:

\begin{align}
&\widetilde{\sigma^2_x} \defeq \widetilde{\XXm}-\tilde{\bar{x}}^2,\\
&\tilde{E}_n = \left(\frac{\widetilde{X^TX}}{n}\right)^{1/2} =
\frac{1}{\sqrt{\widetilde{\XXm}+1+2\sqrt{\widetilde{\sigma^2_x}}}}
\begin{pmatrix}
1 + \sqrt{\widetilde{\sigma^2_x}} & \tilde{\bar{x}}\\
\tilde{\bar{x}} & \widetilde{\XXm}+\sqrt{\widetilde{\sigma^2_x}}
\end{pmatrix},\nonumber\\
& \tilde{F}_n = \frac{\widetilde{X^TY}}{n}
= \begin{pmatrix}
\tilde{\bar{y}}\\
\widetilde{\XYm}
\end{pmatrix},
\quad \tilde{G}_n = \widetilde{\YYm},
\label{eq:dpefg}\\
&\tilde\beta = \begin{pmatrix}
\tilde\beta_2\\
\tilde\beta_1\\
\end{pmatrix},\quad
\tilde\beta^N =
\begin{pmatrix}\tilde{\bar{y}}\\0\end{pmatrix}.
\end{align}

where we take $\sqrt{\widetilde{\sigma^2_x}}$ to be the square root of 
$\widetilde{\sigma^2_x}$ with
non-negative real and imaginary parts.

Furthermore,
if
$\tilde{T}_n = T(\tilde\theta_1)$ is the test statistic obtained via statistics
computed in Algorithm~\ref{alg:t1} and via
Equation~(\ref{eq:glrt}), then $\tilde{T}_n$ can be 
re-written as
\begin{equation}
\tilde{T} = \tilde{T}_n = \frac{n-r}{r-q}\cdot\frac{\norm{\sqrt{n}\tilde{E}_n(\tilde\beta - \tilde\beta^N)}^2}{n(\tilde\beta^T\tilde{E}_n^2\tilde\beta - 2\tilde\beta^T\tilde{F}_n + \tilde{G}_n)},
\label{eq:dpf}
\end{equation}
for any $n, r, q\in\naturals$ such that $q < r$.

\label{lem:dpfequiv}
\end{lemma}

\begin{proof}[Proof of Lemma~\ref{lem:dpfequiv}]

\begin{align*}
(\tilde\beta-\tilde\beta^N)^T\widetilde{X_n^TX_n}(\tilde\beta-\tilde\beta^N)
&= \sqrt{n}(\tilde\beta - \tilde\beta^N)^T\tilde{E}_n^T\sqrt{n}\tilde{E}_n(\tilde\beta - \tilde\beta^N) \\
&= \norm{\sqrt{n}\tilde{E}_n(\tilde\beta - \tilde\beta^N)}^2.
\end{align*}

Next,
\begin{align*}
\widetilde{Y_n^TY_n}-2\tilde{\beta}^T\widetilde{X_n^TY_n}+\tilde{\beta}^T\widetilde{X_n^TX_n}\tilde\beta
&= n\tilde{G}_n - 2n\tilde\beta^T\tilde{F}_n + n\tilde\beta^T\hat{E}_n^2\tilde\beta \\
&= n(\tilde\beta^T\tilde{E}_n^2\tilde\beta - 2\tilde\beta^T\tilde{F}_n + \tilde{G}_n).
\end{align*}

\end{proof}

We now introduce two helper lemmas 
that are useful for showing later results.
The first uses a hybrid-type argument to show
a $1/f(n)$ rate of convergence of a ratio
of random variables.
The second can be used to show that if the difference of
two random variables converge to 0, then
as long as they converge to a non-zero constant,
the difference of their square root converge to 0.

\begin{lemma}

Let $A_n, B_n, \widetilde{A}_n, \widetilde{B}_n$ be random variables such that:
\begin{enumerate}
\item For constants $c_1, c_2\in\reals$, 
$c_2\neq 0$:\quad $A_n\xrightarrow{P} c_1$, $B_n\xrightarrow{P} c_2$,
\item For function $f(n)$:\quad
$f(n)(\widetilde{A}_n - A_n)\xrightarrow{P} 0$ and
$f(n)(\widetilde{B}_n - B_n)\xrightarrow{P} 0$.
\end{enumerate}

Then:
$$
f(n)\left(\frac{\widetilde{A}_n}{\widetilde{B}_n} - \frac{A_n}{B_n}\right)\xrightarrow{P}0.
$$

\label{lem:hybrid}
\end{lemma}

\begin{proof}[Proof of Lemma~\ref{lem:hybrid}]
We use a hybrid-type argument. We write
$$
f(n)\left(\frac{\widetilde{A}_n}{\widetilde{B}_n} - \frac{A_n}{B_n}\right)
= f(n)\left(\frac{\widetilde{A}_n}{\widetilde{B}_n} - \frac{\widetilde{A}_n}{B_n} + \frac{\widetilde{A}_n}{B_n} - \frac{A_n}{B_n}\right).
$$

Then,
\begin{align*}
f(n)\left(\frac{\widetilde{A}_n}{\widetilde{B}_n} - \frac{\widetilde{A}_n}{B_n}\right)
&= f(n)\left(\frac{\widetilde{A}_n B_n - \widetilde{A}_n\widetilde{B}_n}{\widetilde{B}_n B_n}\right) \\
&= \widetilde{A}_n f(n)\left(\frac{B_n - \widetilde{B}_n}{\widetilde{B}_nB_n}\right) \\
&\xrightarrow{P} 0,
\end{align*}
since $\widetilde{A}_n\xrightarrow{P} c_1$, $f(n)(B_n - \widetilde{B}_n)\xrightarrow{P}0$,
and by Slutsky's Theorem 
$B_n\widetilde{B}_n\xrightarrow{P}c_2^2 \neq 0$.

Also,
$$
f(n)\left(\frac{\widetilde{A}_n}{B_n} - \frac{A_n}{B_n}\right) 
= f(n)\left(\frac{\widetilde{A}_n - A_n}{B_n}\right)
\xrightarrow{P} 0,
$$
since $f(n)(\widetilde{A}_n - A_n)\xrightarrow{P}0$,
$B_n\xrightarrow{P}c_2\neq 0$ so that the result follows by a routine
application of Slutsky's Theorem.

As a result,
$f(n)\left(\frac{\widetilde{A}_n}{\widetilde{B}_n} - \frac{A_n}{B_n}\right)\xrightarrow{P}0$.

\end{proof}

\begin{lemma}

Let $A_n,\widetilde{A}_n$ 
be random variables such that:
\begin{enumerate}
\item For constant $c\in\reals$, 
$c\neq 0$:\quad $A_n\xrightarrow{P} c$,
\item For function $f(n)$:\quad
$f(n)(\widetilde{A}_n - A_n)\xrightarrow{P} 0$.
\end{enumerate}

Then:
$$
f(n)(\widetilde{A}_n^{1/2} - A_n^{1/2})\xrightarrow{P} 0.
$$

\label{lem:squareroot}
\end{lemma}

\begin{proof}[Proof of Lemma~\ref{lem:squareroot}]

Throughout,
we take square roots in which both the real and imaginary
parts are non-negative.

Recall that by difference of two squares:
$$
a^{1/2}-b^{1/2} = \frac{a-b}{a^{1/2}+b^{1/2}},
$$
for any $a, b\in\complex$.

Then by Slutsky's Theorem:
\begin{align*}
f(n)(\widetilde{A}_n^{1/2} - A_n^{1/2})
&= \frac{f(n)(\widetilde{A}_n - A_n)}{\widetilde{A}_n^{1/2} + A_n^{1/2}} \\
&\xrightarrow{P} 0,
\end{align*}
where 
$\widetilde{A}_n^{1/2}, A_n^{1/2}\xrightarrow{P} c^{1/2}$.

\end{proof}

We will show that
the DP regression coefficients converge to the true coefficients.
i.e., $\tilde\beta\xrightarrow{P}\beta$.
But we begin with
showing convergence of the constituent
DP sufficient statistics.

\begin{lemma}

For every sequence of clipping bounds
$\Delta = \Delta_n > 0$ and
sequence of privacy parameters
$\rho = \rho_n > 0$,
in Algorithm~\ref{alg:t1}, under the conditions of
Theorem~\ref{thm:dpf}:
\begin{enumerate}
\item[(1)] $\sqrt{n}|\tilde{\bar{x}}-\bar{x}|\xrightarrow{P}0$,
$\tilde{\bar{x}}\xrightarrow{P} c_x$,
\item[(2)] $\sqrt{n}|\tilde{\bar{y}}-\bar{y}|\xrightarrow{P}0$,
$\tilde{\bar{y}}\xrightarrow{P} c_y$,
\item[(3)] $\sqrt{n}|\widetilde{\XXm}-\XXm|\xrightarrow{P}0$,
$\widetilde{\XXm}\xrightarrow{P} c_{x^2}$,
\item[(4)] $\sqrt{n}|\widetilde{\XYm}-\XYm|\xrightarrow{P}0$,
$\widetilde{\XYm}\xrightarrow{P} c_{xy}$,
\item[(5)] $\sqrt{n}|\tilde{\bar{x}}^2-\bar{x}^2|\xrightarrow{P}0$,
\item[(6)] $\sqrt{n}|\tilde{\bar{x}}\tilde{\bar{y}}-\bar{x}\cdot\bar{y}|\xrightarrow{P}0$,
\item[(7)] $\exists c_a\in\reals$, $\widetilde{\XYm}-\tilde{\bar{x}}\tilde{\bar{y}}\xrightarrow{P} c_a$,
\item[(8)] $\exists c_b\neq 0$, $\widetilde{\sigma^2_x}\xrightarrow{P} c_b$,
\item[(9)] $\sqrt{n}(\widetilde{\sigma^2_x} - \widehat{\sigma^2_x})\xrightarrow{P} 0$,
\item[(10)] $\exists C^{1/2}\in\reals^{2\times 2}$ such that
$\sqrt{n}(\tilde{E}_n-\hat{E}_n)\xrightarrow{P}0$, $\tilde{E}_n\xrightarrow{P} C^{1/2}$,
\item[(11)] $\tilde{F}_n\xrightarrow{P} (c_y\quad c_{xy})^T$,
\item[(12)] $\exists c_{y^2}\in\reals$ such that
$\tilde{G}_n\xrightarrow{P} c_{y^2}$,
\end{enumerate}
where the constant scalars and
matrix $c_x, c_y, c_{x^2}, c_{xy}, c_{y^2}, c_a, c_b, C^{1/2}$ are the same
as the ones defined in Lemma~\ref{lem:nonprivateconv}.

\label{lem:t1conv1}
\end{lemma}

\begin{proof}[Proof of Lemma~\ref{lem:t1conv1}]

Define 
\begin{enumerate}
\item $\breve{\bar{x}} = \frac{1}{n}\sum_{i=1}^nx_i|^{\Delta_n}_{-\Delta_n}$,
\item $\breve{\bar{y}} = \frac{1}{n}\sum_{i=1}^ny_i|^{\Delta_n}_{-\Delta_n}$,
\item $\widebreve{\XXm} = \frac{1}{n}\sum_{i=1}^nx_i^2|^{\Delta_n^2}_{0}$,
\item $\widebreve{\XYm} = \frac{1}{n}\sum_{i=1}^nx_iy_i|^{\Delta_n^2}_{-\Delta_n^2}$,
and
\item $\widebreve{\YYm} = \sum_{i=1}^ny_i^2|^{\Delta_n^2}_{0}$.
\end{enumerate}

Then
$\tilde{\bar{x}} = \breve{\bar{x}} + N_1$,
$\tilde{\bar{y}} = \breve{\bar{y}} + N_2$,
$\widetilde{\XXm} = \widebreve{\XXm} + N_3$,
$\widetilde{\XYm} = \widebreve{\XYm} + N_4$,
$\widetilde{\YYm} = \widebreve{\YYm} + N_5$
where $N_1, N_2\sim\calN(0,  \frac{2\Delta^2}{\rho n^2})$,
$N_3\sim\calN(0, \frac{\Delta^4}{2\rho n^2})$,
$N_4\sim\calN(0, \frac{2\Delta^4}{\rho n^2})$,
and $N_5\sim\calN(0, \frac{\Delta^4}{2\rho})$.

By conditions of Theorem~\ref{thm:dpf},
$\frac{\Delta_n^2}{\rho_n n}\rightarrow0$, 
$\frac{\Delta_n^4}{\rho_n n}\rightarrow0$ so that
by Corollary~\ref{cor:gausstail},
$\sqrt{n}|N_1|\xrightarrow{P}0$,
$\sqrt{n}|N_2|\xrightarrow{P}0$,
$\sqrt{n}|N_3|\xrightarrow{P}0$,
$\sqrt{n}|N_4|\xrightarrow{P}0$, and
$\frac{\sqrt{n}}{n}|N_5|\xrightarrow{P}0$
since
$\sqrt{n}\calN(0, \frac{2\Delta^2}{\rho n^2}) \sim
\calN(0, \frac{2\Delta^2}{\rho n})$,
$\sqrt{n}\calN(0, \frac{\Delta^4}{2\rho n^2}) \sim
\calN(0, \frac{\Delta^4}{2\rho n})$, and
$\frac{\sqrt{n}}{n}\calN(0, \frac{2\Delta^4}{\rho}) \sim
\calN(0, \frac{2\Delta^4}{\rho n})$.

(1): By assumption, for all $i\in[n]$,
$x_i\in[-\Delta_n, \Delta_n]$. Thus, $\bar{x} = \breve{\bar{x}}$ so that
$\tilde{\bar{x}} = \bar{x} + N_1$. Then,
$$\sqrt{n}|\tilde{\bar{x}}-\bar{x}| \leq \sqrt{n}|N_1| + \sqrt{n}|\breve{\bar{x}}-\bar{x}| \xrightarrow{P} 0$$ by Slutsky's Theorem.
As a corollary, $\tilde{\bar{x}}\xrightarrow{P} c_x$
by Lemma~\ref{lem:helperprob} and the assumption in
Theorem~\ref{thm:dpf} that $\bar{x}\rightarrow c_x$.

(2):
The proof that $\sqrt{n}|\bar{y} - \tilde{\bar{y}}|\xrightarrow{P} 0$ is very similar:
observe that by the assumptions of Theorem~\ref{thm:dpf}:
\begin{align}
\pr[|\bar{y} - \breve{\bar{y}}| > 0]
&\leq \pr[\exists i\in[n], y_i\notin[-\Delta_n, \Delta_n]] \\
&\rightarrow 0.
\end{align}
Thus, $\sqrt{n}|\breve{\bar{y}}-\bar{y}|\xrightarrow{P}0$.
Combining with $\sqrt{n}|N_2|\xrightarrow{P}0$, by the
triangle inequality,
$\sqrt{n}|\bar{y} - \tilde{\bar{y}}|\xrightarrow{P} 0$.
As a corollary, $\tilde{\bar{y}}\xrightarrow{P} c_y$
by Lemma~\ref{lem:helperprob} and the assumption in
Theorem~\ref{thm:dpf} that $\bar{x}\rightarrow c_x$.

(3):
To show
$\sqrt{n}|\widetilde{\XXm}-\XXm|\xrightarrow{P}0$, we
proceed in an analogous way:
using the assumption that
$\pr[\exists i\in[n], x_i\notin[-\Delta_n, \Delta_n]] = 0$,
we obtain that
$\pr[\exists i\in[n], x_i^2\notin[0, \Delta_n^2]] = 0$
so that
$\sqrt{n}|\XXm-\widebreve{\XXm}|\xrightarrow{P}0$.
Combining with $\sqrt{n}|N_3|\xrightarrow{P}0$, by the
triangle inequality,
$\sqrt{n}|\widetilde{\XXm}-\XXm|\xrightarrow{P}0$.
As a corollary, $\widetilde{\XXm}\xrightarrow{P} c_{x^2}$
by Lemma~\ref{lem:helperprob} and the assumption in
Theorem~\ref{thm:dpf} that $\XXm \rightarrow c_{x^2}$.
 
(4):
In a similar fashion, 
$\sqrt{n}|\widetilde{\XYm}-\XYm|\xrightarrow{P}0$:
using the assumptions 
$$
\pr[\exists i\in[n], x_i\notin[-\Delta_n, \Delta_n]] = 0,\quad\pr[\exists i\in[n], y_i\notin[-\Delta_n, \Delta_n]] \rightarrow 0,
$$
we have that
\begin{align*}
&\pr[\exists i\in[n], x_iy_i\notin[-\Delta_n^2, \Delta_n^2]]\\
&\leq \pr[\exists i\in[n], x_i\notin[-\Delta_n, \Delta_n]] + \pr[\exists i\in [n], y_i\notin[-\Delta_n, \Delta_n]]\\
&\rightarrow 0,
\end{align*}
so that
$\sqrt{n}|\XYm-\widebreve{\XYm}|\xrightarrow{P}0$.
Combining with $\sqrt{n}|N_4|\xrightarrow{P}0$, by the
triangle inequality,
$\sqrt{n}|\widetilde{\XYm}-\XYm|\xrightarrow{P}0$.
By Lemma~\ref{lem:nonprivateconv},
$\XYm\xrightarrow{P} c_{xy}$. Then by
Lemma~\ref{lem:helperprob},
$\widetilde{\XYm}\xrightarrow{P} c_{xy}$.

(5):
Next we show
$\sqrt{n}|\tilde{\bar{x}}^2-\bar{x}^2|\xrightarrow{P}0$:
$\sqrt{n}(\tilde{\bar{x}}^2-\bar{x}^2) = \sqrt{n}(\tilde{\bar{x}}-\bar{x})(\tilde{\bar{x}}+\bar{x})$.
Since, $\tilde{\bar{x}}, \bar{x}\xrightarrow{P} c_x$, we have
$(\tilde{\bar{x}}+\bar{x})\xrightarrow{P} 2c_x$,
$(\tilde{\bar{x}}-\bar{x})\xrightarrow{P} 0$ so that
by Slutsky's Theorem,
$\sqrt{n}|\tilde{\bar{x}}^2-\bar{x}^2|\xrightarrow{P}0$

(6):
In a similar fashion,
$\sqrt{n}|\tilde{\bar{x}}\tilde{\bar{y}}-\bar{x}\cdot\bar{y}|\xrightarrow{P}0$:
$\tilde{\bar{x}} = \breve{\bar{x}} + N_1 = 
\bar{x} + N_1$, since
$\forall i\in[n], x_i\in[-\Delta_n, \Delta_n]$.
Then,
\begin{align*}
\sqrt{n}(\tilde{\bar{x}}\tilde{\bar{y}}-\bar{x}\cdot\bar{y})
&= \sqrt{n}\left[(\bar{x} + N_1)\tilde{\bar{y}}-\bar{x}\cdot\bar{y}\right] \\
&= \sqrt{n}N_1\tilde{\bar{y}} + 
\bar{x}\sqrt{n}(\tilde{\bar{y}} - \bar{y}) \\
&\xrightarrow{P} 0,
\end{align*}
since $\sqrt{n}(\tilde{\bar{y}} - \bar{y})\xrightarrow{P} 0$,
$\tilde{\bar{y}}\xrightarrow{P} c_y$,
$\sqrt{n}N_1\xrightarrow{P}0$.

(7): Next, we show that
$\widetilde{\XYm}-\tilde{\bar{x}}\tilde{\bar{y}}\xrightarrow{P} c_a\in\reals$:
Follows by Slutsky's Theorem since
$\widetilde{\XYm}\xrightarrow{P} c_{xy}$,
$\tilde{\bar{x}}\xrightarrow{P}c_x$, and
$\tilde{\bar{y}}\xrightarrow{P}c_y$.

(8): In a similar fashion,
$\widetilde{\sigma^2_x}\xrightarrow{P} c_b$:
This follows by Slutsky's Theorem since
$\widetilde{\XXm}\xrightarrow{P} c_{x^2}$,
$\tilde{\bar{x}}\xrightarrow{P}c_x$.

(9) $\sqrt{n}(\widetilde{\sigma^2_x} - \widehat{\sigma^2_x})\xrightarrow{P} 0$ follows from
parts (3) and (5).

(10):
By Lemma~\ref{lem:squareroot},
$$
\sqrt{n}\left(\sqrt{\widetilde{\sigma^2_x}} - \sqrt{\widehat{\sigma^2_x}}\right)
\xrightarrow{P} 0,
$$
since 
$\sqrt{n}\left(\widetilde{\sigma^2_x} - \widehat{\sigma^2_x}\right)\xrightarrow{P} 0$
and 
$\widehat{\sigma^2_x}\xrightarrow{P} c_b\neq 0$ by
Lemma~\ref{lem:nonprivateconv}.

We have already established that
$\sqrt{n}(\widetilde{\XXm}-\XXm)\xrightarrow{P}0$
and 
$\sqrt{n}\left(\sqrt{\widetilde{\sigma^2_x}} - \sqrt{\widehat{\sigma^2_x}}\right)
\xrightarrow{P}0$.

As a result,
by Lemma~\ref{lem:squareroot},
$$
\sqrt{n}\left(\sqrt{\widetilde{\XXm}+1+2\sqrt{\widetilde{\sigma^2_x}}} - 
\sqrt{\XXm+1+2\sqrt{\widehat{\sigma^2_x}}}
\right) \xrightarrow{P} 0.
$$

Then since $\hat{E}_n, \tilde{E}_n$ converge to constant matrices and
$\sqrt{n}(\tilde{\bar{x}}-\bar{x})\xrightarrow{P}0$,
$\sqrt{n}(\widetilde{\XXm}-\XXm)\xrightarrow{P}0$,
$\sqrt{n}\left(\sqrt{\widetilde{\sigma^2_x}} - \sqrt{\widehat{\sigma^2_x}}\right)\xrightarrow{P}0$, it follows by Lemma~\ref{lem:hybrid}
that $\sqrt{n}(\tilde{E}_n-\hat{E}_n)\xrightarrow{P}0$.

As a corollary,
$\tilde{E}_n\xrightarrow{P} C^{1/2}$ since $\hat{E}_n\xrightarrow{P} C^{1/2}$
by Lemma~\ref{lem:nonprivateconv}.

(11): Also,
$\tilde{F}_n\xrightarrow{P} (c_y\quad c_{xy})^T$
since $\tilde{\bar{y}}\xrightarrow{P} c_y$ and
$\widetilde{\XYm}\xrightarrow{P} c_{xy}$.

(12): Finally, we show that
$\tilde{G}_n\xrightarrow{P} c_{y^2}\in\reals$:
using the assumption that
$\pr[\exists i\in[n], y_i\notin[-\Delta_n, \Delta_n]]\rightarrow 0$, we can obtain that
$\pr[\exists i\in[n], y_i^2\notin[0, \Delta_n^2]]\rightarrow 0$
so that
$\sqrt{n}|\YYm-\widebreve{\YYm}|\xrightarrow{P}0$.
Combining with $\frac{\sqrt{n}}{n}|N_5|\xrightarrow{P}0$, by the
triangle inequality,
$\sqrt{n}|\YYm - \widetilde{\YYm}|\xrightarrow{P} 0$ which
implies that
$\sqrt{n}|\tilde{G}_n-\hat{G}_n|\xrightarrow{P}0$.
Then by Lemma~\ref{lem:nonprivateconv} and
Lemma~\ref{lem:helperprob},
$\tilde{G}_n\xrightarrow{P} c_{y^2}$.

\end{proof}

Lemma~\ref{lem:t1conv1} shows that the noise added
to the non-DP estimates converges, in probability,
to 0 and that the DP estimates of the regression parameters
converge, in 
probability, to the true parameters.
Next, we will show the $1/\sqrt{n}$ convergence rates of
$\tilde{\beta}^N, \tilde{\beta}$. As a corollary, this implies
the consistency of $\tilde{\beta}^N, \tilde{\beta}$.

\begin{lemma}
For every sequence of clipping bounds
$\Delta = \Delta_n > 0$ and
sequence of privacy parameters
$\rho = \rho_n > 0$,
in Algorithm~\ref{alg:t1}, under the conditions of Theorem~\ref{thm:dpf}:
\begin{enumerate}
\item $\sqrt{n}(\tilde\beta^N - \hat\beta^N)\xrightarrow{P}0$,
\item $\sqrt{n}(\tilde\beta - \hat\beta)\xrightarrow{P}0$.
\end{enumerate}
\label{lem:t1conv2}
\end{lemma}

\begin{proof}[Proof of Lemma~\ref{lem:t1conv2}]

As previously defined,
$$
\tilde\beta^N =
\begin{pmatrix}\tilde{\bar{y}}\\0\end{pmatrix},\quad
\hat\beta^N =
\begin{pmatrix}\bar{y}\\0\end{pmatrix}.
$$
Then,
$\sqrt{n}(\tilde\beta^N - \hat\beta^N)\xrightarrow{P}0$
by Lemma~\ref{lem:t1conv1} since
$\sqrt{n}|\tilde{\bar{y}}-\breve{\bar{y}}|\xrightarrow{P}0$.

We will
show that $\sqrt{n}(\tilde\beta - \hat\beta)\xrightarrow{P}0$.
First, to show that
$\sqrt{n}(\tilde\beta_1 - \hat\beta_1)\xrightarrow{P}0$, we apply
Lemma~\ref{lem:hybrid} with
$\widetilde{A} = \widetilde{\XYm} - \tilde{\bar{x}}\tilde{\bar{y}}$,
$A = \XYm - \bar{x}\cdot\bar{y}$,
$\widetilde{B} = \widetilde{\XXm} - \tilde{\bar{x}}^2 = \widetilde{\sigma^2_x}$,
$B = \XXm - \bar{x}^2 = \widehat{\sigma^2_x}$ and $f(n) = \sqrt{n}$.
Then $\tilde\beta_1 = \frac{\widetilde{A}}{\widetilde{B}}$ and
$\hat\beta_1 = \frac{A}{B}$.
By Lemma~\ref{lem:nonprivateconv},
$B = \widehat{\sigma^2_x}$ converges, in probability, to a non-zero constant
and
$\sqrt{n}(\widetilde{\XYm} - \XYm)$,
$\sqrt{n}\left(\bar{x}\cdot\bar{y}-\tilde{\bar{x}}\tilde{\bar{y}}\right)\xrightarrow{P}0$ by
Lemma~\ref{lem:t1conv1}.
Also,
by Lemma~\ref{lem:t1conv1} and Lemma~\ref{lem:nonprivateconv},
if we define
$\widetilde{B} = \widetilde{\sigma^2_x}, B = \widehat{\sigma^2_x}, \widetilde{A} = \widetilde{\XYm} - \tilde{\bar{x}}\tilde{\bar{y}}$, then
$\sqrt{n}(\widetilde{B}-B)\xrightarrow{P}0$ and
$\sqrt{n}(\widetilde{A}-A)\xrightarrow{P}0$.
Then by Lemma~\ref{lem:hybrid},
$\sqrt{n}(\tilde\beta_1 - \hat\beta_1)\xrightarrow{P} 0$.
By similar arguments,
$\sqrt{n}(\hat\beta_2 - \tilde\beta_2)\xrightarrow{P}0$
so that
$\sqrt{n}(\hat\beta - \tilde\beta)\xrightarrow{P}0$.

\end{proof}

Lemma~\ref{lem:t1conv2} leads to the following corollary, showing consistency of the
DP estimates of $\beta$ under the null or alternative hypothesis.

\begin{corollary}
For every sequence of clipping bounds
$\Delta = \Delta_n > 0$ and
sequence of privacy parameters
$\rho = \rho_n > 0$,
in Algorithm~\ref{alg:t1}, under the conditions of Theorem~\ref{thm:dpf}:
\begin{enumerate}

\item Under the null hypothesis:
$\tilde\beta^N \xrightarrow{P}\beta$,

\item Under the alternative hypothesis:
$\tilde\beta\xrightarrow{P}\beta$.

\end{enumerate}
\label{cor:dpbetas}
\end{corollary}

\begin{proof}[Proof of Corollary~\ref{cor:dpbetas}]

By Lemma~\ref{lem:nonprivateconv},
$\hat\beta_1\xrightarrow{P}\beta_1$ and
$\hat\beta_2\xrightarrow{P}\beta_2$.
Then, using Lemma~\ref{lem:t1conv2} and
Lemma~\ref{lem:helperprob}:
$\tilde\beta_1\xrightarrow{P}\beta_1$,
$\tilde\beta_2\xrightarrow{P}\beta_2$.

Also, by Lemma~\ref{lem:nonprivateconv},
$\bar{y}\xrightarrow{P} c_y$ so that under the null hypothesis,
$\tilde\beta^N\xrightarrow{P}\beta$.

\end{proof}

\subsection{Convergence of Differentially Private $F$-Statistic}

We now introduce the continuous mapping theorem, which is especially useful for combining
individual convergence results to show, under certain conditions,
more complex convergence results.
The continuous mapping theorem can be used to map
convergent sequences into another convergent sequence via
a continuous function.

\begin{theorem}[Continuous Mapping Theorem,  see~\citep{gut2013probability}]
Let $\{W_n\}$
be a sequence of random vectors and $W$ be
a random vector taking values in the same
metric space $\calX$. Let $\calY$ be a metric space and 
$g:\calX\rightarrow\calY$ be a measurable function.

Define $D_g = \{x\,:\,g\text{ is discontinuous at }x\}$.
Suppose that 
$\pr[W\in D_g] = 0$. Then:
\begin{enumerate}
\item $W_n\xrightarrow{P}W\Rightarrow g(W_n)\xrightarrow{P} g(W)$,
\item $W_n\xrightarrow{D}W\Rightarrow g(W_n)\xrightarrow{D} g(W)$,
\item $W_n\xrightarrow{a.s.}W\Rightarrow g(W_n)\xrightarrow{a.s.} g(W)$.
\end{enumerate}

\label{thm:cmt}
\end{theorem}

We now state and prove a helper lemma that will be useful for showing that the
numerators and denominators of the DP $F$-statistic converge to the right
distribution.

\begin{lemma}

Let $A_n, B_n$ be random vectors such that
there exists distribution $L$ where:
\begin{enumerate}
\item $A_n - B_n \xrightarrow{P} 0$,
\item $\norm{B_n} \xrightarrow{D} L$ such that $\pr[\norm{B_n} = 0] = 0$.
\end{enumerate}

Then,
$$
\norm{A_n}^2 \xrightarrow{D} L^2.
$$

\label{lem:squareddist}
\end{lemma}

\begin{proof}[Proof of Lemma~\ref{lem:squareddist}]

Consider the unit vector
$\frac{B_n}{\norm{B_n}}$. 
Since
$A_n - B_n \xrightarrow{P} 0$, we have that by definition of convergence in
probability:
$$
\frac{B_n}{\norm{B_n}}(A_n - B_n) \xrightarrow{P} 0,
$$
where $\norm{B_n}$ is almost surely never 0. 

First, let $W_n = (\norm{B_n}, \norm{B_n})$.
Then since $\norm{B_n} \xrightarrow{D} L$, by the
continuous mapping theorem (Theorem~\ref{thm:cmt}), if
$g((x, y)) = x\cdot y$, then
$g(W_n) = \norm{B_n}^2 \xrightarrow{D} L^2$.

Then, let $W_n = (\norm{B_n}, \frac{B_n}{\norm{B_n}}(A_n - B_n))$.
Then since $\pr[\norm{B_n} = 0] = 0$, by the
continuous mapping theorem (Theorem~\ref{thm:cmt}), if $g((x, y)) = x\cdot y$, then
$g(W_n) = B_n\cdot(A_n - B_n) \xrightarrow{D} 0$ which implies that
$$
\brackets{A_n, B_n} - \norm{B_n}^2 = B_n\cdot(A_n - B_n)\xrightarrow{P} 0,
$$
so that
\begin{equation}
2(\brackets{A_n, B_n} - \norm{B_n}^2) \xrightarrow{P} 0.
\label{eq:aminusb}
\end{equation}

Also, by the continuous mapping theorem,
\begin{align}
&\norm{A_n}^2 - 2\brackets{A_n, B_n} + \norm{B_n}^2\label{eq:absq} \\
&= \norm{A_n - B_n}^2 \\
&\xrightarrow{P} 0.
\end{align}

Adding Equations~(\ref{eq:aminusb}) and~(\ref{eq:absq}) results in the following:
$\norm{A_n}^2 - \norm{B_n}^2\xrightarrow{P} 0$. Then by
Lemma~\ref{lem:helperprob}, since $\norm{B_n}^2 \xrightarrow{D} L^2$, we have that
$\norm{A_n}^2 \xrightarrow{D} L^2$.

\end{proof}

We now show that the main terms in the numerators and denominators of
the DP $F$-statistic converge to the asymptotic distribution of their
non-private counterparts.

\begin{lemma}
Let $\sigma_e > 0$, $r=p=2, q=1$, 
and $\beta\in\reals^p$. For every $n\in\naturals$ with
$n > r$, let
$X_n\in\reals^{n\times p}$ be the
design matrix.
For every sequence of clipping bounds
$\Delta = \Delta_n > 0$ and
sequence of privacy parameters
$\rho = \rho_n > 0$, in Algorithm~\ref{alg:t1}, under the conditions of
Theorem~\ref{thm:dpf}:
$$
\frac{n(\tilde{\beta}^T\tilde{E}_n^2\tilde{\beta} - 2\tilde{\beta}^T\tilde{F}_n + \tilde{G}_n)}{n-r}\xrightarrow{P}\sigma_e^2,\quad
\norm{\sqrt{n}\tilde{E}_n(\tilde\beta - \tilde{\beta}^N)}^2\xrightarrow{D}\calX^2_{r-q}(\eta^2)\sigma_e^2
.
$$
\label{lem:convbeta}
\end{lemma}

\begin{proof}[Proof of Lemma~\ref{lem:convbeta}]

By Lemma~\ref{lem:nonprivateconv} and Lemma~\ref{lem:t1conv1}:
\begin{enumerate}
\item $\tilde\beta, \hat\beta \xrightarrow{P} \beta$,
\item $\tilde{E}_n, \hat{E}_n\rightarrow C^{1/2}$,
\item $\tilde{F}_n, \hat{F}_n\xrightarrow{P} (c_y\quad c_{xy})^T$,
\item $\tilde{G}_n, \hat{G}_n\xrightarrow{P} c_{y^2}$.
\end{enumerate}

Furthermore,
$$
\frac{n}{n-r} = \frac{1}{1 - r/n} \rightarrow 1.
$$

As a result, by Slutsky's Theorem,
$$
\left(\frac{n(\tilde{\beta}^T\tilde{E}_n^2\tilde{\beta} - 2\tilde{\beta}^T\tilde{F}_n + \tilde{G}_n) - n(\hat{\beta}^T\hat{E}_n^2\hat{\beta} - 2\hat{\beta}^T\hat{F}_n + \hat{G}_n)}{n-r}\right)\xrightarrow{P}0.
$$

By Theorem~\ref{thm:f},
$$
\frac{\norm{Y_n-X_n\hat\beta}^2}{n-r}\xrightarrow{P}\sigma_e^2.
$$
By Lemma~\ref{lem:fequiv},
$\norm{Y-X\hat\beta}^2 = n(\hat\beta^T\hat{E}_n^2\hat\beta - 2\hat\beta^T\hat{F}_n + \hat{G}_n)$.
As a result, by Lemma~\ref{lem:helperprob},
$$\frac{n(\tilde{\beta}^T\tilde{E}_n^2\tilde{\beta} - 2\tilde{\beta}^T\tilde{F}_n + \tilde{G}_n)}{n-r}\xrightarrow{P}\sigma_e^2.$$

Next, by Theorem~\ref{thm:f},
$\norm{X_n\hat\beta - X_n\hat\beta^N}^2\sim\chi^2_{r-q}(\eta^2)\sigma_e^2$.
Then by Lemma~\ref{lem:fequiv},
$\norm{\sqrt{n}\hat{E}_n(\hat\beta - \hat{\beta}^N)}^2\sim\chi^2_{r-q}(\eta^2)\sigma_e^2$.
We will show that
$$
\norm{\sqrt{n}\tilde{E}_n(\tilde\beta - \tilde{\beta}^N)}^2
\xrightarrow{D}\chi^2_{r-q}(\eta^2)\sigma_e^2.
$$

First, we define the random vectors
$$
\tilde{H}_n = \tilde{E}_n\sqrt{n}(\tilde\beta - \tilde{\beta}^N),\quad
\hat{H}_n = \hat{E}_n\sqrt{n}(\hat\beta - \hat{\beta}^N).
$$

By Lemma~\ref{lem:t1conv2},
$\sqrt{n}(\tilde\beta - \hat\beta)\xrightarrow{P}0$ and
$\sqrt{n}(\tilde\beta^N - \hat\beta^N)\xrightarrow{P}0$.
And since by Lemma~\ref{lem:nonprivateconv},
$\hat{E}_n\xrightarrow{P} C^{1/2}$,
we have that by Slutsky's Theorem,
$\sqrt{n}\hat{E}_n(\tilde\beta - \hat\beta)\xrightarrow{P}0$, and
$\sqrt{n}\hat{E}_n(\tilde\beta^N - \hat\beta^N)\xrightarrow{P}0$.
Also, by Lemma~\ref{lem:t1conv1},
$\sqrt{n}(\tilde{E}_n-\hat{E}_n)\xrightarrow{P} 0$ which implies
that, by Slutsky's Theorem, and Lemma~\ref{lem:t1conv2},
$\sqrt{n}(\tilde{E}_n-\hat{E}_n)\tilde\beta\xrightarrow{P}0$
and 
$\sqrt{n}(\tilde{E}_n-\hat{E}_n)\tilde\beta^N\xrightarrow{P}0$ so that
$\sqrt{n}(\tilde{E}_n-\hat{E}_n)(\tilde\beta - \tilde\beta^N)\xrightarrow{P}0$.

As a result,
\begin{align*}
&\tilde{H}_n - \hat{H}_n \\
&= \tilde{E}_n\sqrt{n}(\tilde{\beta}-\tilde{\beta}^N) - \hat{E}_n\sqrt{n}(\hat{\beta}-\hat{\beta}^N) \\
&= \tilde{E}_n\sqrt{n}\tilde\beta - \hat{E}_n\sqrt{n}\hat\beta
- \left[\tilde{E}_n\sqrt{n}\tilde{\beta}^N - \hat{E}_n\sqrt{n}\hat\beta^N\right] \\
&= (\tilde{E}_n-\hat{E}_n)\sqrt{n}\tilde\beta + \hat{E}_n\sqrt{n}(\tilde\beta - \hat\beta) - \left[(\tilde{E}_n-\hat{E}_n)\sqrt{n}\tilde\beta^N + \hat{E}_n\sqrt{n}(\tilde\beta^N - \hat\beta^N)\right]\\
&= (\tilde{E}_n-\hat{E}_n)\sqrt{n}(\tilde\beta - \tilde\beta^N) + \hat{E}_n\sqrt{n}(\tilde\beta - \hat\beta) - \left[\hat{E}_n\sqrt{n}(\tilde\beta^N - \hat\beta^N)\right]\\
&\xrightarrow{P} 0.
\end{align*}

By Lemma~\ref{lem:squareddist}, since
$\norm{\hat{H}_n}\xrightarrow{D}\chi_{r-q}(\eta^2)\sigma_e$ and
$\tilde{H}_n - \hat{H}_n\xrightarrow{P} 0$, we have that
$$\norm{\tilde{H}_n}^2\xrightarrow{D}\chi^2_{r-q}(\eta^2)\sigma_e^2.$$
As a result, 
$\norm{\sqrt{n}\tilde{E}_n(\tilde\beta - \tilde\beta^N)}^2\xrightarrow{D}\chi^2_{r-q}(\eta^2)\sigma_e^2$.
This completes the proof.

\end{proof}

Lemma~\ref{lem:convbeta} shows the convergence
of individual quantities that can now be
combined to show the convergence of the
DP test statistic
$\tilde{T}$:

\begin{proof}[Proof of Theorem~\ref{thm:dpf}]

First, by Corollary~\ref{cor:dpbetas},
under the null hypothesis:
$\tilde\beta^N = \tilde{\beta}_n^N
\xrightarrow{P}\beta$. And under the alternative
hypothesis:
$\tilde{\beta} = \tilde{\beta}_n\xrightarrow{P} \beta$.

By Lemma~\ref{lem:fequiv},
$$
T = T_n = \frac{n-r}{r-q}\cdot\frac{\norm{X\hat\beta - X\hat\beta^N}^2}{\norm{Y-X\hat\beta}^2}
= \frac{n-r}{r-q}\cdot\frac{\norm{\sqrt{n}\hat{E}_n(\hat\beta - \hat\beta^N)}^2}{n(\hat\beta^T\hat{E}_n^2\hat\beta - 2\hat\beta^T\hat{F}_n + \hat{G}_n)}.
$$

And by Equation~(\ref{eq:dpf}),
$$
\tilde{T} = \tilde{T}_n = \frac{n-r}{r-q}\cdot\frac{\norm{\sqrt{n}\tilde{E}_n(\tilde\beta - \tilde\beta^N)}^2}{n(\tilde\beta^T\tilde{E}_n^2\tilde\beta - 2\tilde\beta^T\tilde{F}_n + \tilde{G}_n)}.
$$

From Theorem~\ref{thm:f}, in the non-private case where
$Y_n\sim\calN(X_n\beta, \sigma^2_eI_{n\times n})$,
if $T = T_n$ is the test statistic from Equation~(\ref{eq:glrt}), then
$$
T_n\sim F_{r-q, n-r}(\eta_n^2),\quad \eta_n^2 = \frac{\norm{X_n\beta - X_n\beta^N}^2}{\sigma_e^2}.
$$
Also, by Theorem~\ref{thm:f},
the asymptotic distribution of $T$ is a chi-squared
distribution. i.e., 
$T = T_n\xrightarrow{D}\frac{\chi^2_{r-q}(\eta^2)}{r-q}$.
Next, we show that the DP $F$-statistic also
has asymptotic distribution of chi-squared.

By Lemma~\ref{lem:convbeta},
$$
\norm{\sqrt{n}\tilde{E}_n(\tilde\beta - \tilde{\beta}^N)}^2\xrightarrow{D}\calX^2_{r-q}(\eta^2)\sigma_e^2,
$$
and
$$
\frac{n(\tilde{\beta}^T\tilde{E}_n^2\tilde{\beta} - 2\tilde{\beta}^T\tilde{F}_n + \tilde{G}_n)}{n-r}\xrightarrow{P}\sigma_e^2.
$$

Let
$$
W_n = \left(\frac{\norm{\sqrt{n}\tilde{E}_n(\tilde\beta - \tilde{\beta}^N)}^2}{r-q}, 
\frac{n(\tilde{\beta}^T\tilde{E}_n^2\tilde{\beta} - 2\tilde{\beta}^T\tilde{F}_n + \tilde{G}_n)}{n-r}\right) = (A_n, B_n),
$$
and
$$
W = \left(\frac{\calX^2_{r-q}(\eta^2)\sigma_e^2}{r-q}, \sigma_e^2\right).
$$

By the condition that
$\sigma_e > 0$\,$(\pr[\sigma_e = 0] = 0)$ and
the continuous mapping theorem (Theorem~\ref{thm:cmt}),
$\tilde{T}_n = g(W_n) = g(A_n, B_n) = A_n/B_n$
converges,
in distribution, to $\frac{\chi^2_{r-q}(\eta^2)}{r-q}$ as $n\rightarrow\infty$.

\end{proof}

\section{Experimental Evaluation of Power and Significance}
\label{sec:experiments}

We will measure the effectiveness of our hypothesis tests via 
significance and power.
In Section~\ref{sec:expframework}, we describe our meta-procedures
for collating the significance and power of our implementations of
non-private and private statistical tests.

The power and significance of our differentially private
tests are estimated on both semi-synthetic datasets based on
the Opportunity Atlas~\citep{NBERw25147, 10.1093/qje/qjz042}
and on synthetic datasets.
The OI semi-synthetic datasets consists of simulated microdata
for each census tract in some states in the U.S.
The dependent variable $Y$ is the child national income percentile and the
independent variable $X$ is the corresponding parent national income percentile.
See~\citep{AMSSV20} for more details on the
properties of simulated data from the OI team.
In the OI data, $X$ is lognormally distributed and the distribution of counts of
individuals across tracts in a state follows an exponential distribution.

\paragraph{General Parameter Setup for Synthetic Data:}
For experimental evaluation on synthetic datasets, we generated datasets
with sizes between
$n=100$ and $n=10,000$. 

For both the linear relationship and mixture model tests on synthetic data below, we consider a subset of
the following values of the privacy budget
$\rho$: $\{0.1^2/2, 0.5^2/2, 1^2/2, 2^2/2, 3^2/2, 5^2/2, 10^2/2\}$.

We draw the independent variables $x_1, \ldots, x_n$ according to a few
different distributions:
Normal, Uniform, Exponential. We will
detail the parameters used to generated
variables from these distributions in the corresponding subsections.

For all tests below, the clipping
parameter is either set to $\Delta = 2$ or $\Delta = 3$.
For the synthetic data,
the dependent variable $Y$ is generated using the linear or mixture
model
specification described in previous sections and by fixing or
varying parameters (such as $\sigma_e$).
For estimating the power and significance,
we fix the target significance
level to 0.05 and run Monte Carlo tests 2000 times.
We estimate the power and significance as the fraction of times the null
is rejected, given various settings of parameters that satisfy the
alternative and null hypothesis, respectively.

\subsection{Testing a Linear Relationship on Synthetic Data}

\subsubsection{$F$-statistic}

We evaluate our DP linear relationship test
on synthetically generated data from three
different distributions: normal, uniform, and exponential. We
also vary parameters such as: the slope of the linear model and
the noise distribution of the dependent variable.

\textbf{Evaluating the Significance for Normally Distributed Independent Variables}:
Generally, we see that the significance
remains below the target significance level, on average, for all values of
$\rho$. For the linear relationship tester, when the
standard deviation
of the dependent variable ($\sigma_e$) is small (Figure~\ref{fig:lin1a}), we see that
the true significance level is well below the target significance of 0.05, which is
fine (but conservative). We conjecture that this happens because when $\sigma_e$ is small:
(i)  we fail to reject when the noisy estimate of $\sigma_e$ is $\leq 0$; or
(ii)  the test statistic under the null distribution will be almost always 0 since under the null (even without privacy), 
the standard deviation of the test statistic is proportional to $\sigma_e$. 
In Figures~\ref{fig:lin1a},~\ref{fig:lin1b} and~\ref{fig:lin1c}, 
we see the significance of the linear tester attains the target (of 0.05) 
as we vary the noise in the dependent variable $\sigma_e$.

\begin{figure}
  \begin{subfigure}[b]{0.3\textwidth}
    \includegraphics[width=\textwidth]{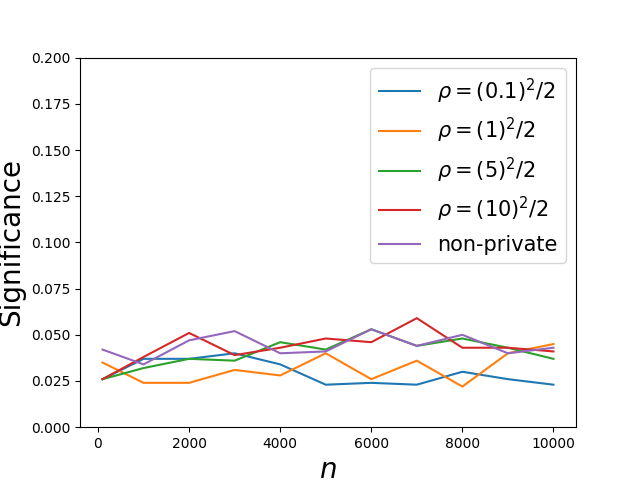}
    \caption{Significance for testing a linear relationship. $x_i\sim\calN(0.5, 1)$, $y_i\sim 0\cdot x_i + \calN(0, 0.001^2)$. $\Delta = 2$.}
    \label{fig:lin1a}
  \end{subfigure}
  \qquad
   \begin{subfigure}[b]{0.3\textwidth}
    \includegraphics[width=\textwidth]{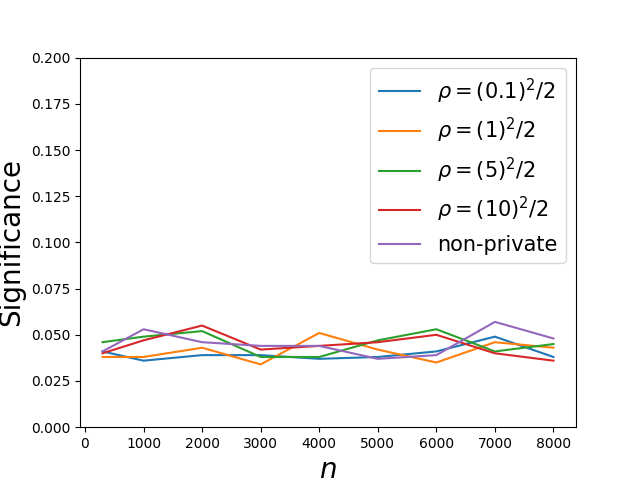}
    \caption{Significance for testing a linear relationship.  $x_i\sim\calN(0.5, 1)$, $y_i\sim 0\cdot x_i + \calN(0, 0.35^2)$. $\Delta = 2$.}
    \label{fig:lin1b}
    \end{subfigure}
    \qquad
    \begin{subfigure}[b]{0.3\textwidth}
    \includegraphics[width=\textwidth]{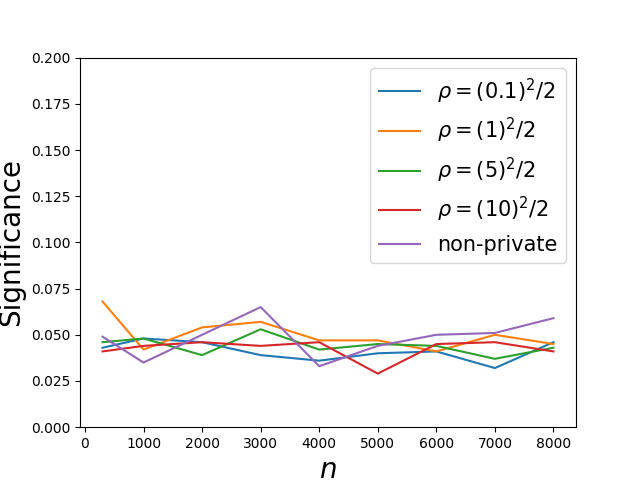}
    \caption{Significance for testing a linear relationship.  $x_i\sim\calN(0.5, 1)$, $y_i\sim 0\cdot x_i + \calN(0, 1)$. $\Delta = 2$.}
    \label{fig:lin1c}
    \end{subfigure}
  \caption{}
\end{figure}

\textbf{Evaluating Power for Varying the Noise in the Dependent Variable}:
For Figures~\ref{fig:lin3a},~\ref{fig:lin3b}, and~\ref{fig:lin3c}, 
we set the true slope to 1.
We then vary the
noise in the dependent variable. That is, 
for the general linear model,
$Y \sim\calN(X\beta, \sigma_e^2I_{n\times n})$, we vary
$\sigma_e$. The following values of $\sigma_e$ are considered:
$\{0.001, 0.35, 1\}$.

In Figure~\ref{fig:lin3a},
we generally see that compared to higher values of $\sigma_e$
(Figures~\ref{fig:lin3b} and~\ref{fig:lin3c}),
the power is relatively low.
We believe this occurs because
when $\sigma_e$ is small, its DP estimate is more likely to be 
less than 0, in which case we fail to reject the null (even when the alternative is
true). This generally leads to a reduction in the power.

\begin{figure}
  \begin{subfigure}[b]{0.3\textwidth}
    \includegraphics[width=\textwidth]{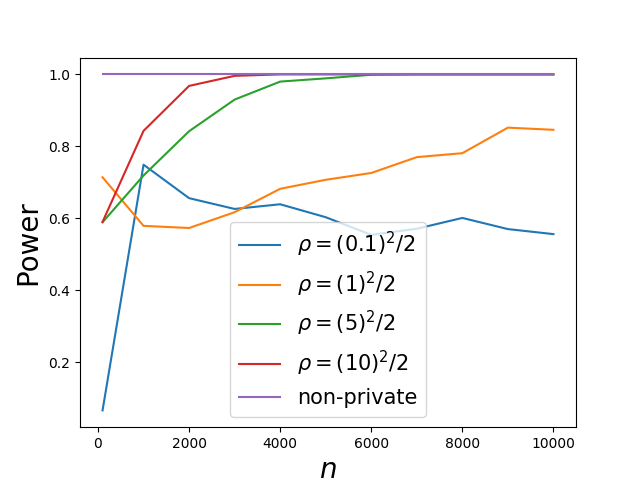}
    \caption{Power for testing a linear relationship.
    $x_i\sim\calN(0.5, 1)$, $y_i\sim 1\cdot x_i + \calN(0, 0.001^2)$. $\Delta = 2$.}
    \label{fig:lin3a}
  \end{subfigure}
  \qquad
  \begin{subfigure}[b]{0.3\textwidth}
    \includegraphics[width=\textwidth]{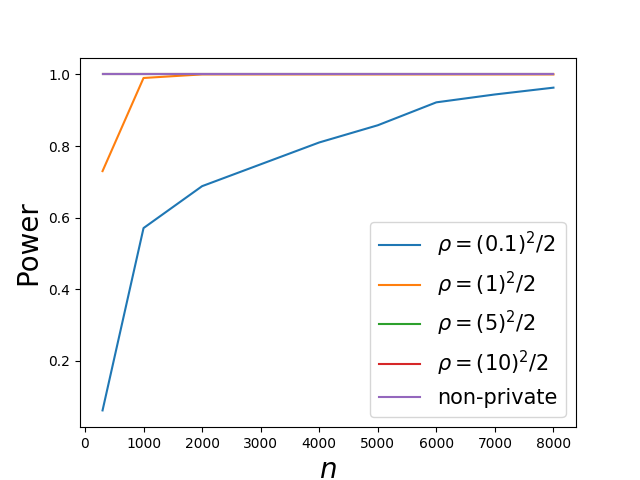}
    \caption{Power for testing a linear relationship. $x_i\sim\calN(0.5, 1)$, $y_i\sim 1\cdot x_i + \calN(0, 0.35^2)$. $\Delta = 2$.}
    \label{fig:lin3b}
  \end{subfigure}
  \qquad
  \begin{subfigure}[b]{0.3\textwidth}
    \includegraphics[width=\textwidth]{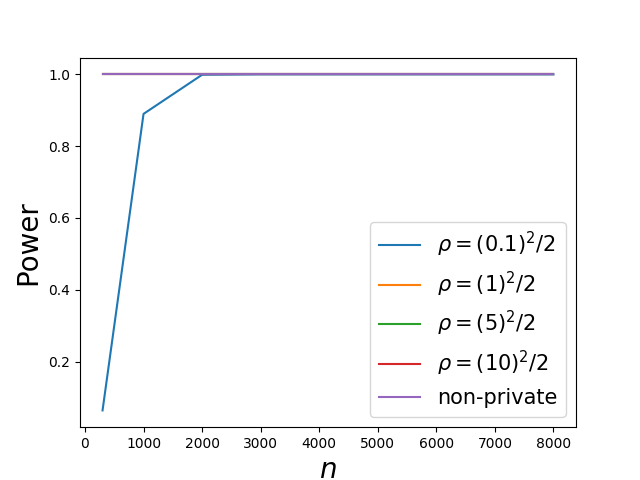}
    \caption{Power for testing a linear relationship. $x_i\sim\calN(0.5, 1)$, $y_i\sim 1\cdot x_i + \calN(0, 1)$. $\Delta = 2$.}
    \label{fig:lin3c}
  \end{subfigure}
  \caption{}
\end{figure}

\textbf{Evaluating Power for Varying Slopes}:
Figures~\ref{fig:lin2a},~\ref{fig:lin2b}
show the power
of the linear test for slopes of 0.1, 1.
We generally see that the larger the slope, the higher the power of
the DP tests.

\begin{figure}
  \begin{subfigure}[b]{0.5\textwidth}
    \includegraphics[width=\textwidth]{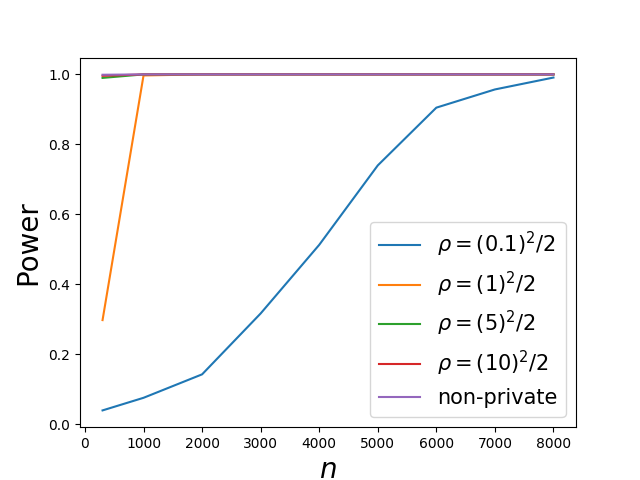}
    \caption{Power for testing a linear relationship.\\ $x_i\sim\calN(0.5, 1)$, $y_i\sim 0.1\cdot x_i + \calN(0, 0.35^2)$. $\Delta = 2$.}
    \label{fig:lin2a}
  \end{subfigure}
  \qquad
  \begin{subfigure}[b]{0.5\textwidth}
    \includegraphics[width=\textwidth]{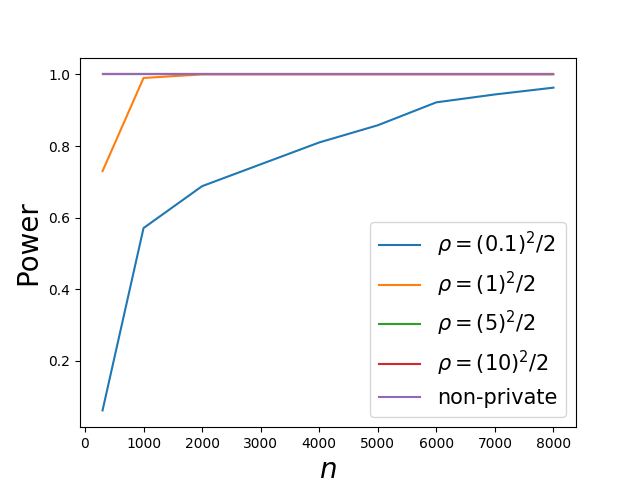}
    \caption{Power for testing a linear relationship.\\
    $x_i\sim\calN(0.5, 1)$, $y_i\sim 1\cdot x_i + \calN(0, 0.35^2)$. $\Delta = 2$.}
    \label{fig:lin2b}
  \end{subfigure}
  \caption{}
\end{figure}

\textbf{Evaluating the Significance while Varying the Distribution of the Independent Variable}:
For Figures~\ref{fig:lin5a},~\ref{fig:lin5b},
and~\ref{fig:lin5c}, 
we set
the standard deviation of the noise dependent variable to 0.35.
We then vary the
distribution of the independent variable 
--- while maintaining the variance ---
to take on one of the following:
\begin{enumerate}
\item\textbf{Normal}: with mean 0.5 and variance 1/12.
\item\textbf{Uniform}: between 0 and 1 (variance of 1/12).
\item\textbf{Exponential}: with scale of $1/\sqrt{12}$.
\end{enumerate}

We observe that the significance is still preserved even though,
in our DP testers, the null distribution is simulated via a normal
distribution.

\begin{figure}
  \begin{subfigure}[b]{0.3\textwidth}
    \includegraphics[width=\textwidth]{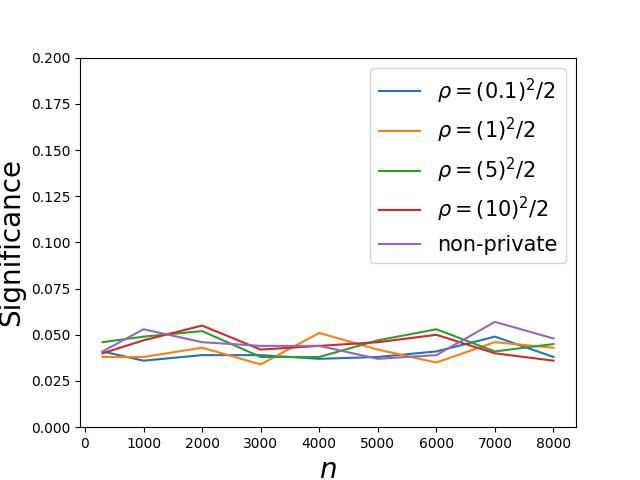}
    \caption{Significance for testing a linear relationship. Normal Distribution on $X$.}
    \label{fig:lin5a}
  \end{subfigure}
  \qquad
  \begin{subfigure}[b]{0.3\textwidth}
    \includegraphics[width=\textwidth]{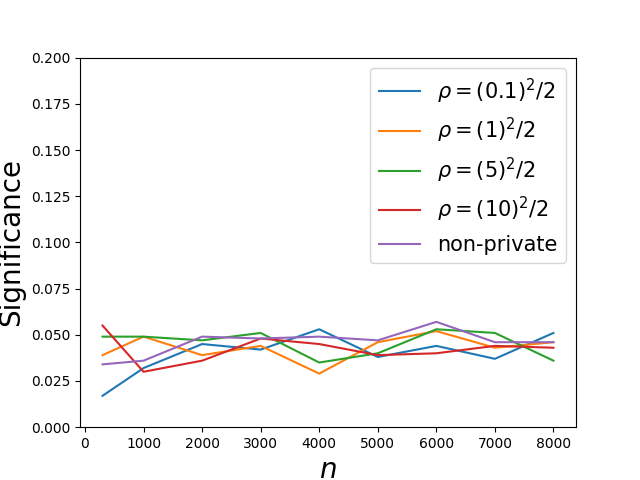}
    \caption{Significance for testing a linear relationship. Uniform Distribution on $X$.}
    \label{fig:lin5b}
  \end{subfigure}
  \begin{subfigure}[b]{0.3\textwidth}
    \includegraphics[width=\textwidth]{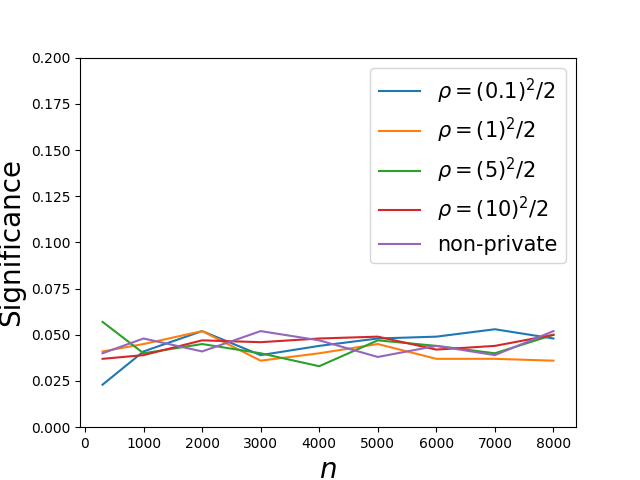}
    \caption{Significance for testing a linear relationship. Exponential Distribution on $X$.}
    \label{fig:lin5c}
  \end{subfigure}
  \caption{}
\end{figure}

\subsubsection{Bootstrap Confidence Intervals}
\label{sec:ci}

Using the duality between confidence interval estimation and
hypothesis testing, we can construct hypothesis tests
for testing a linear relationship based on DP confidence
interval procedures.
Specifically, 
we compare the $F$-statistic linear relationship tester
to the tester that uses DP confidence intervals.
See Section~\ref{sec:dual} for more details on the
experimental framework of the DP bootstrap confidence intervals.
Algorithm~\ref{alg:mctci} summarizes the approach for testing
that builds on DP confidence intervals.

In Figure~\ref{fig:cisig}, we present
experimental results for the significance level of
Algorithm~\ref{alg:mctci} compared to 
Algorithm~\ref{alg:mct} instantiated with
the DP $F$-statistic.
As we see, Algorithm~\ref{alg:mctci}
achieves the target significance level. 
In Figure~\ref{fig:cipower}, we also present
experimental results for the power of
Algorithm~\ref{alg:mctci} compared to 
Algorithm~\ref{alg:mct}. We see that Algorithm~\ref{alg:mctci}
has less power than Algorithm~\ref{alg:mct}.
This observation is more pronounced for
less concentrated distributions (i.e., uniform) on the
independent variable. See
Figure~\ref{fig:cipowerunif}.
This might be due to the, sometimes excessive,
width of the confidence interval produced by the
bootstrap interval (in order to ensure coverage under the null
hypothesis).

Figures~\ref{fig:cisig},~\ref{fig:cipower},~\ref{fig:cisigunif},
and~\ref{fig:cipowerunif} show
results averaged out over 2000 trials.
The dashed lines
correspond to the bootstrap confidence interval approach
(denoted CI) while the solid lines
are for the $F$-statistic (denoted $F$-stat).

\begin{figure}
   \begin{subfigure}[b]{0.5\textwidth}
    \includegraphics[width=\textwidth]{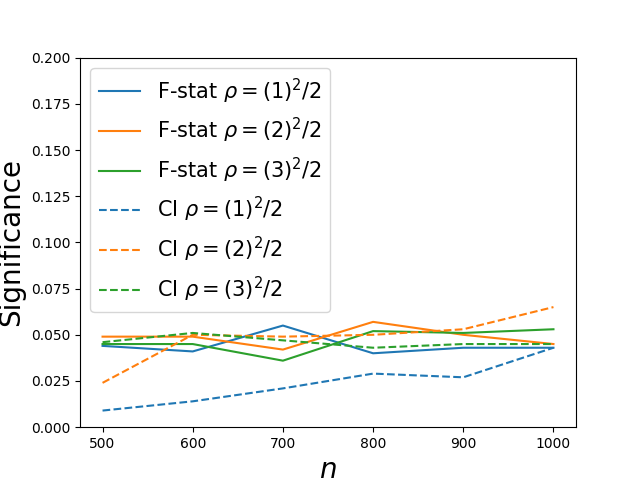}
    \caption{Significance for $F$-statistic versus confidence
    interval approach. 
    $x_i\sim\calN(0.5, 1)$, $y_i\sim 0\cdot x_i + \calN(0, 0.35^2)$.
    $\Delta = 2$.}
    \label{fig:cisig}
  \end{subfigure}
  \qquad
   \begin{subfigure}[b]{0.5\textwidth}
    \includegraphics[width=\textwidth]{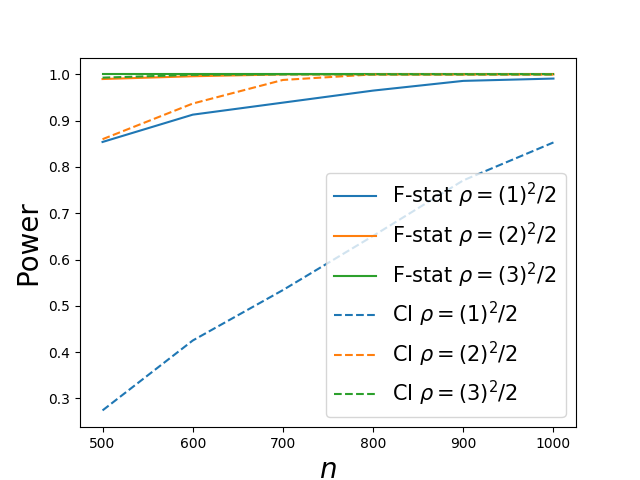}
    \caption{Power for $F$-statistic versus confidence
    interval approach. $x_i\sim\calN(0.5, 1)$, $y_i\sim 1\cdot x_i + \calN(0, 0.35^2)$.
    $\Delta = 2$.}
    \label{fig:cipower}
  \end{subfigure}
  \caption{}
\end{figure}

\begin{figure}
   \begin{subfigure}[b]{0.5\textwidth}
    \includegraphics[width=\textwidth]{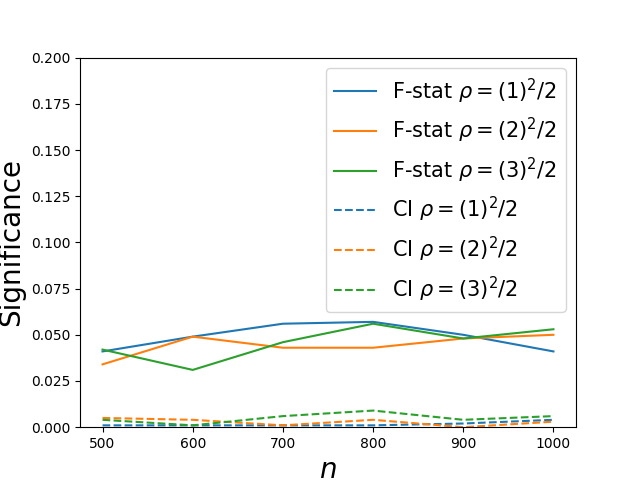}
    \caption{Significance for $F$-statistic versus confidence
    interval approach. 
    $x_i\sim\Unif[0, 1]$, $y_i\sim 0\cdot x_i + \calN(0, 0.35^2)$.
    $\Delta = 2$.}
    \label{fig:cisigunif}
  \end{subfigure}
  \qquad
   \begin{subfigure}[b]{0.5\textwidth}
    \includegraphics[width=\textwidth]{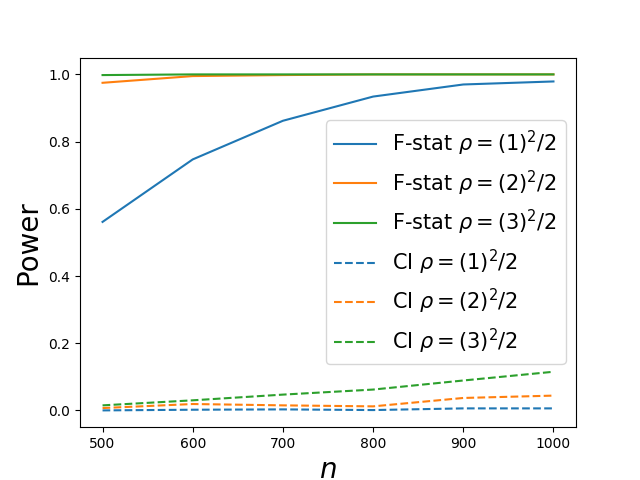}
    \caption{Power for $F$-statistic versus confidence
    interval approach. $x_i\sim\Unif[0, 1]$, $y_i\sim 1\cdot x_i + \calN(0, 0.35^2)$.\\
    $\Delta = 2$.}
    \label{fig:cipowerunif}
  \end{subfigure}
  \caption{}
\end{figure}

\subsubsection{Bernoulli Tester}
\label{sec:bin}

The DP Bernoulli tester has a higher power than the
DP $F$-statistic when the slope is large or the privacy-loss
parameter is small; otherwise the DP $F$-statistic performs
better.
We vary the slope from 0.01 up to 1. 
As the slope gets smaller, the performance gap between the $F$-statistic and the Bernoulli tester gets larger. 
Figures~\ref{fig:bin01} and~\ref{fig:bin05} show the
power of the DP $F$-statistic compared to the DP Bernoulli
tester for the problem of testing a 
linear relationship.
The significance levels of the Bernoulli
tester are shown in Figures~\ref{fig:binsig1} 
and~\ref{fig:binsig2}.

The results are averaged out over 1000 trials.
For the figures illustrating the power, dashed lines
correspond to the Bernoulli testing approach
(denoted Bern) while the solid lines
are for the $F$-statistic (denoted $F$-stat).

\begin{figure}
   \begin{subfigure}[b]{0.5\textwidth}
    \includegraphics[width=\textwidth]{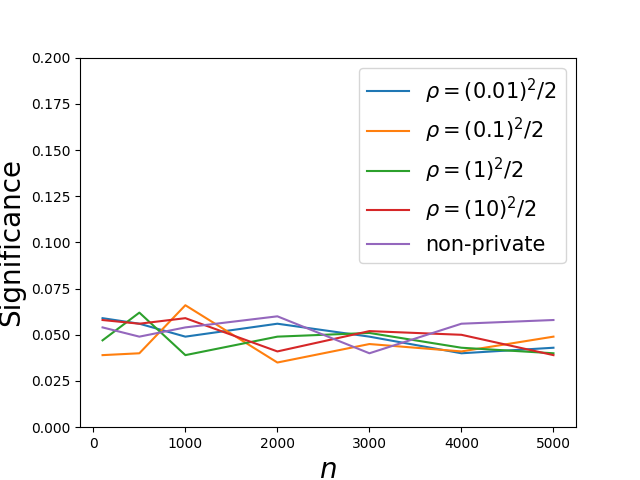}
    \caption{Significance for testing for a fair coin. \\
    $n$ observations are generated from $\Bern(1/2)$.\\
    }
    \label{fig:binsig1}
  \end{subfigure}
  \qquad
   \begin{subfigure}[b]{0.5\textwidth}
    \includegraphics[width=\textwidth]{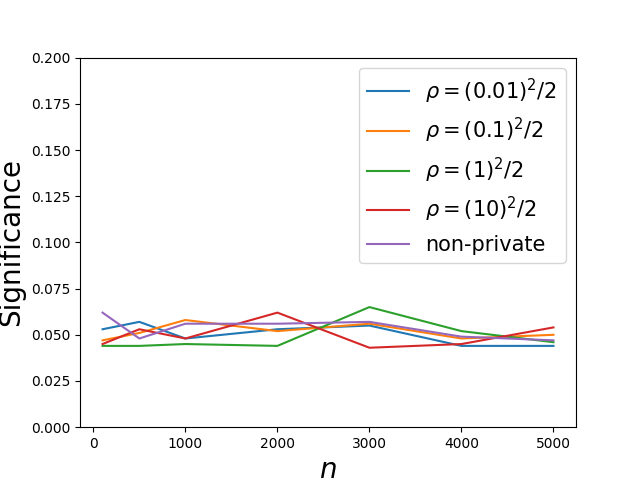}
    \caption{Significance for linear relationship testing via Bernoulli testing approach. $x_i\sim\calN(0.5, 1)$, $y_i\sim 0\cdot x_i + \calN(0, 1)$.
    $\Delta = 2$.}
    \label{fig:binsig2}
  \end{subfigure}
  \caption{}
\end{figure}

\begin{figure}
   \begin{subfigure}[b]{0.5\textwidth}
    \includegraphics[width=\textwidth]{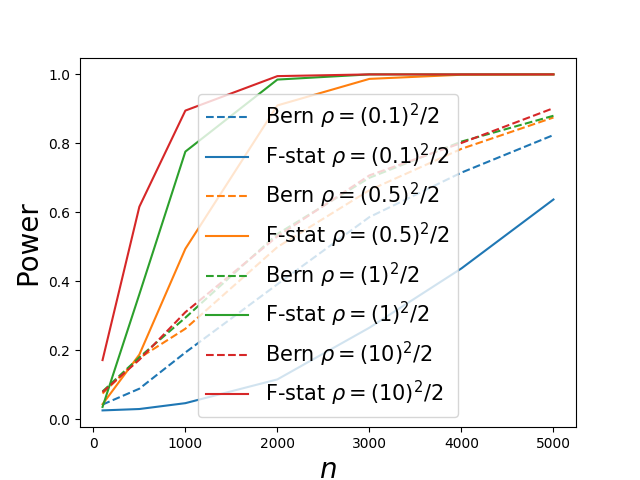}
    \caption{Power for $F$-statistic versus Bernoulli testing approach. $x_i\sim\calN(0.5, 1)$, $y_i\sim 0.1\cdot x_i + \calN(0, 1)$.
    $\Delta = 2$.}
    \label{fig:bin01}
  \end{subfigure}
  \qquad
   \begin{subfigure}[b]{0.5\textwidth}
    \includegraphics[width=\textwidth]{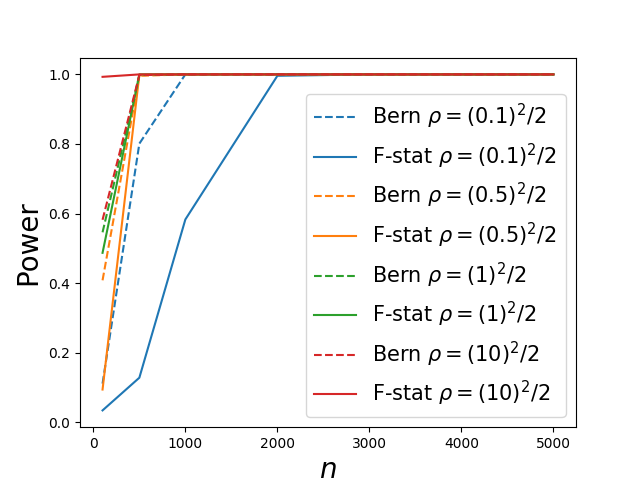}
    \caption{Power for $F$-statistic versus Bernoulli testing approach. $x_i\sim\calN(0.5, 1)$, $y_i\sim 0.5\cdot x_i + \calN(0, 1)$.
    $\Delta = 2$.}
    \label{fig:bin05}
  \end{subfigure}
  \caption{}
\end{figure}

\subsection{Testing Mixture Models on Synthetic Data}

\subsubsection{$F$-Statistic}

We evaluate the $F$-statistic DP mixture model test on
synthetically generated data. We vary
parameters such as: the fraction of data in each group and
the slopes used to generate data for each group.
Let $\beta_1, \beta_2$ denote the slopes of the two groups.

\textbf{Evaluating the Significance}:
Like in the DP linear model tester, we
also see that we achieve
the target significance levels, on average, for all values of
$\rho$. In Figures~\ref{fig:mix1a},~\ref{fig:mix1b},
and~\ref{fig:mix1c}, we vary the noise in the dependent variable.
In Figures~\ref{fig:mix2a},~\ref{fig:mix2b},
and~\ref{fig:mix2c}, we vary the fraction of group sizes,
using either a 1/8, 1/4, or 1/2 fraction for the first group.
We see that the more unbalanced splits tend to have lower significance
levels.

\begin{figure}
  \begin{subfigure}[b]{0.3\textwidth}
    \includegraphics[width=\textwidth]{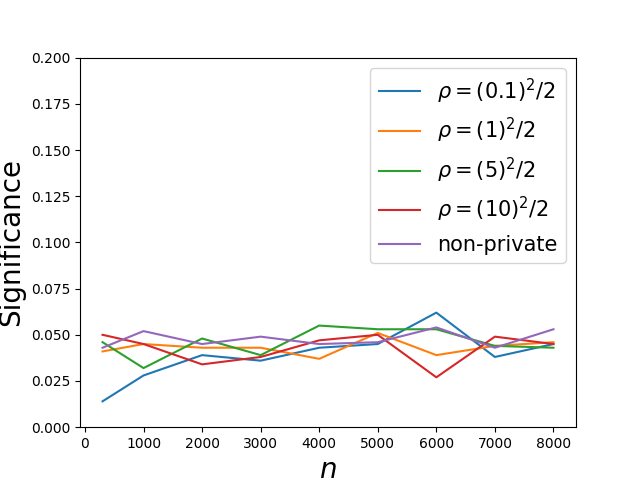}
    \caption{Significance for testing mixtures. Equal-sized groups. $x_i\sim\calN(0.5, 1)$, $y_i\sim 1\cdot x_i + \calN(0, 0.01^2)$ for Group 1. $y_i\sim 1\cdot x_i + \calN(0, 0.01^2)$ for Group 2.}
    \label{fig:mix1a}
  \end{subfigure}
  \qquad
   \begin{subfigure}[b]{0.3\textwidth}
    \includegraphics[width=\textwidth]{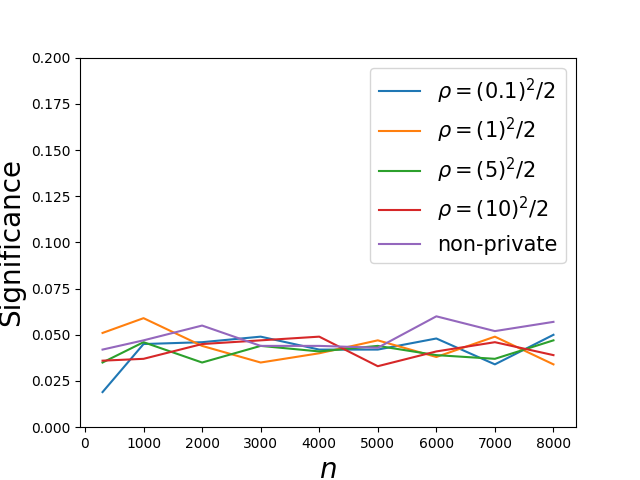}
    \caption{Significance for testing mixtures. Equal-sized groups. $x_i\sim\calN(0.5, 1)$, $y_i\sim 1\cdot x_i + \calN(0, 0.35^2)$ for Group 1. $y_i\sim 1\cdot x_i + \calN(0, 0.35^2)$ for Group 2.}
    \label{fig:mix1b}
    \end{subfigure}
    \qquad
    \begin{subfigure}[b]{0.3\textwidth}
    \includegraphics[width=\textwidth]{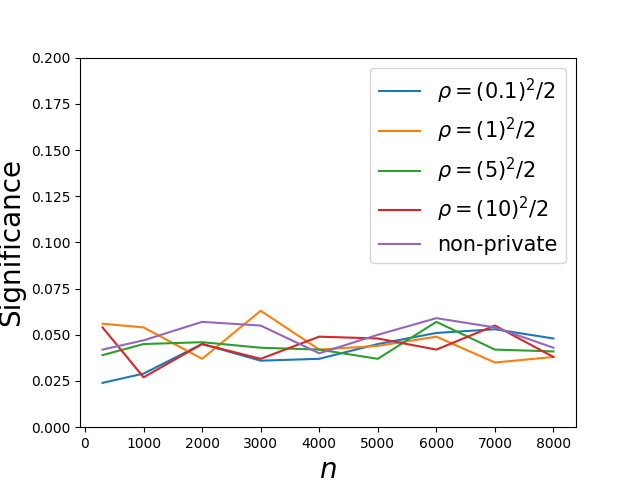}
    \caption{Significance for testing mixtures. Equal-sized groups. $x_i\sim\calN(0.5, 1)$, $y_i\sim 1\cdot x_i + \calN(0, 1)$ for Group 1. $y_i\sim 1\cdot x_i + \calN(0, 1)$ for Group 2.}
    \label{fig:mix1c}
    \end{subfigure}
  \caption{}
\end{figure}

\begin{figure}
  \begin{subfigure}[b]{0.3\textwidth}
    \includegraphics[width=\textwidth]{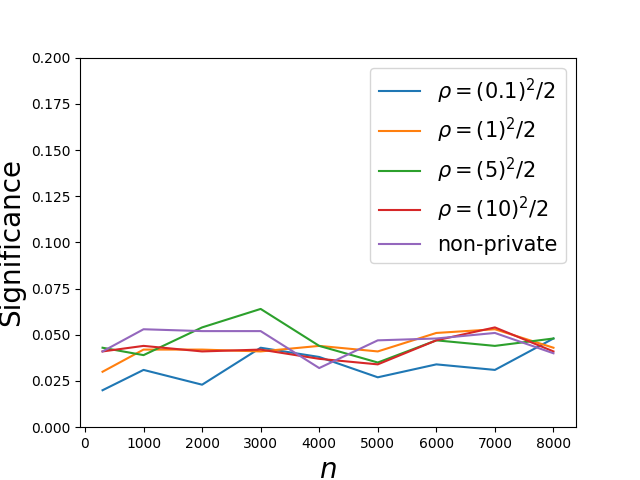}
    \caption{Significance for testing mixtures. 1/8th vs. 7/8th splits. $x_i\sim\calN(0.5, 1)$, $y_i\sim 1\cdot x_i + \calN(0, 0.35^2)$ for Group 1. $y_i\sim 1\cdot x_i + \calN(0, 0.35^2)$ for Group 2.}
    \label{fig:mix2a}
  \end{subfigure}
  \qquad
   \begin{subfigure}[b]{0.3\textwidth}
    \includegraphics[width=\textwidth]{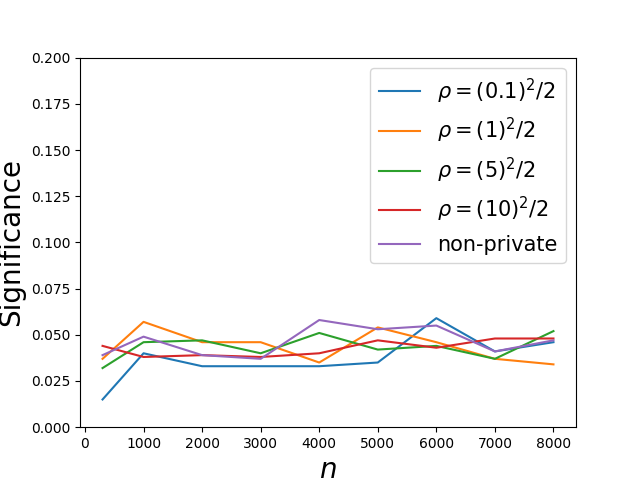}
    \caption{Significance for testing mixtures. 1/4th vs. 3/4th splits. $x_i\sim\calN(0.5, 1)$, $y_i\sim 1\cdot x_i + \calN(0, 0.35^2)$ for Group 1. $y_i\sim 1\cdot x_i + \calN(0, 0.35^2)$ for Group 2.}
    \label{fig:mix2b}
    \end{subfigure}
    \qquad
    \begin{subfigure}[b]{0.3\textwidth}
    \includegraphics[width=\textwidth]{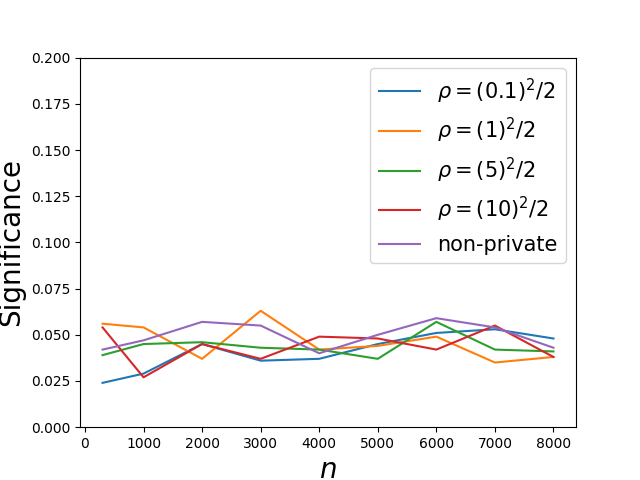}
    \caption{Significance for testing mixtures. 
     Equal-sized groups. $x_i\sim\calN(0.5, 1)$, $y_i\sim 1\cdot x_i + \calN(0, 0.35^2)$ for Group 1. $y_i\sim 1\cdot x_i + \calN(0, 0.35^2)$ for Group 2.}
    \label{fig:mix2c}
    \end{subfigure}
  \caption{}
\end{figure}

\textbf{Power while Varying the Group Size Fraction}: Let $n$ be the total
number of datapoints and $n_1, n_2$ be the number of points in
groups 1 and 2 respectively. We vary the fraction of points in
group 1: $n_1/n$. Setting the slopes of each group to $\beta_1=-1$
and $\beta_2=1$,
we vary this fraction so that $n_1/n\in\{1/8, 1/4, 1/2\}$.
For Figure~\ref{fig:mix3a},
we set the group sizes to be equal.
For Figure~\ref{fig:mix3b},
we set the group sizes to be $n/4, 3n/4$.
And last, for
Figure~\ref{fig:mix3c}, the group sizes are
$n/8, 7n/8$.

Generally, the more even the group size fractions are, the higher
the power of the DP test for testing mixtures in the general
linear model.

\begin{figure}
  \begin{subfigure}[b]{0.3\textwidth}
    \includegraphics[width=\textwidth]{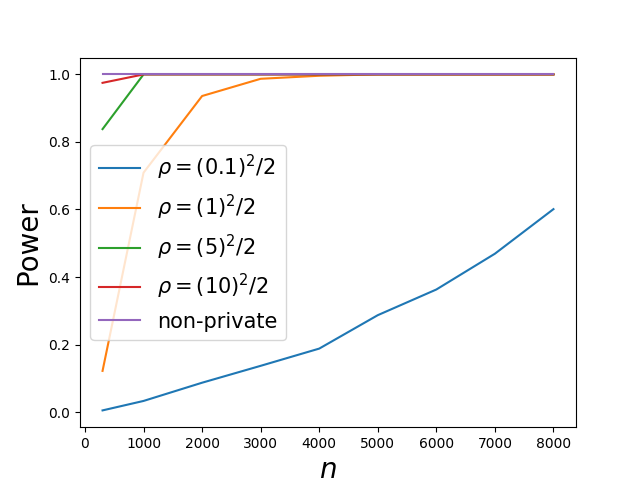}
    \caption{Power for testing mixtures. 1/8th vs. 7/8th splits. $x_i\sim\calN(0.5, 1)$, $y_i\sim -1\cdot x_i + \calN(0, 0.35^2)$ for Group 1. $y_i\sim 1\cdot x_i + \calN(0, 0.35^2)$ for Group 2.}
    \label{fig:mix3a}
  \end{subfigure}
  \qquad
   \begin{subfigure}[b]{0.3\textwidth}
    \includegraphics[width=\textwidth]{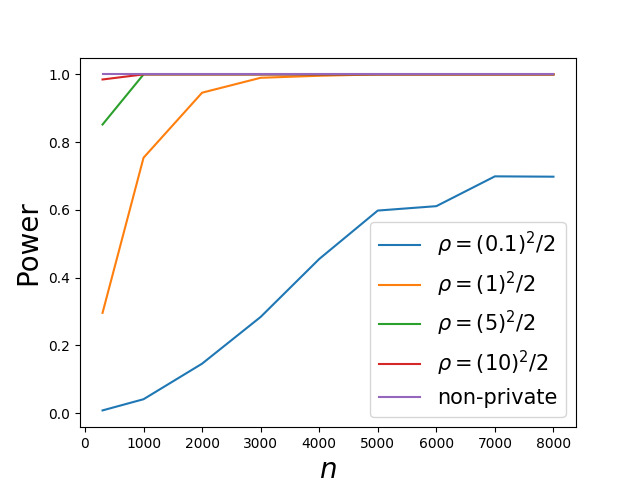}
    \caption{Power for testing mixtures. 1/4th vs. 3/4th splits. $x_i\sim\calN(0.5, 1)$, $y_i\sim -1\cdot x_i + \calN(0, 0.35^2)$ for Group 1. $y_i\sim 1\cdot x_i + \calN(0, 0.35^2)$ for Group 2.}
    \label{fig:mix3b}
    \end{subfigure}
    \qquad
    \begin{subfigure}[b]{0.3\textwidth}
    \includegraphics[width=\textwidth]{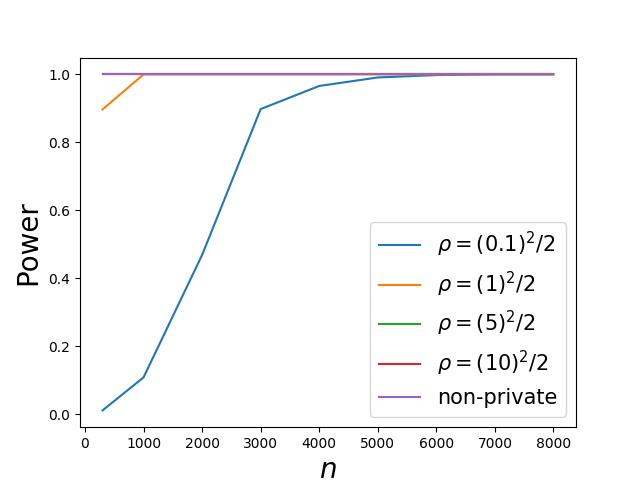}
    \caption{Power for testing mixtures. 
     Equal-sized groups. $x_i\sim\calN(0.5, 1)$, $y_i\sim -1\cdot x_i + \calN(0, 1)$ for Group 1. $y_i\sim 1\cdot x_i + \calN(0, 1)$ for Group 2.}
    \label{fig:mix3c}
    \end{subfigure}
  \caption{}
\end{figure}

\textbf{Power while Varying the Difference Between Slopes in Each Group}: 
Let $\beta_1, \beta_2$ correspond to the slopes for groups 1 and 2.
We vary $|\beta_1 - \beta_2|$.
Generally, we see that the larger $|\beta_1 - \beta_2|$ is, the
higher the power of the test.
In Figures~\ref{fig:mix4a},~\ref{fig:mix4b}
we vary the slope in the two groups and observe
the aforementioned phenomena.

\begin{figure}
  \begin{subfigure}[b]{0.5\textwidth}
    \includegraphics[width=\textwidth]{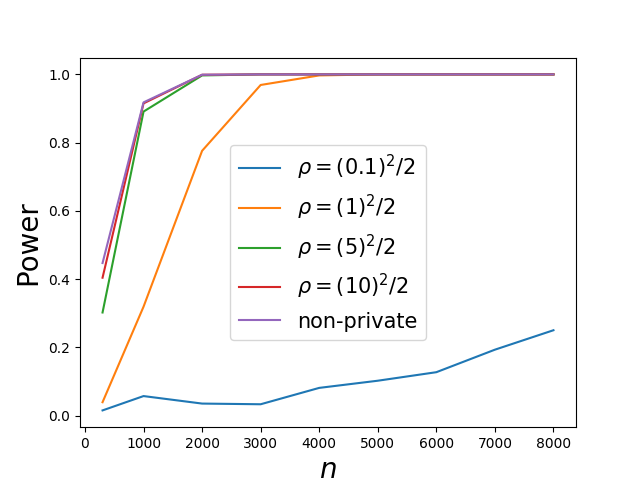}
    \caption{Power for testing mixtures. Equal-sized groups. $x_i\sim\calN(0.5, 1)$, $y_i\sim -0.1\cdot x_i + \calN(0, 0.35^2)$ for Group 1. $y_i\sim 0.1\cdot x_i + \calN(0, 0.35^2)$ for Group 2.}
    \label{fig:mix4a}
  \end{subfigure}
  \qquad
   \begin{subfigure}[b]{0.5\textwidth}
    \includegraphics[width=\textwidth]{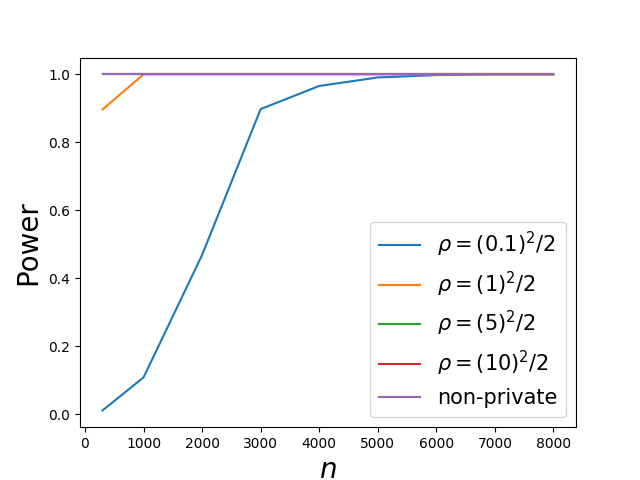}
    \caption{Power for testing mixtures. Equal-sized groups. $x_i\sim\calN(0.5, 1)$, $y_i\sim -1\cdot x_i + \calN(0, 0.35^2)$ for Group 1. $y_i\sim 1\cdot x_i + \calN(0, 0.35^2)$ for Group 2.}
    \label{fig:mix4b}
    \end{subfigure}
    \caption{}
\end{figure}

\textbf{Power while Varying the Noise in the Dependent Variable}: 
We also vary $\sigma_e$. We generally see that the smaller it is,
the smaller the power. We conjecture that this happens because
we err on the side of failing to reject the null if 
the DP estimate of $\sigma_e$ becomes $\leq 0$, which is more
likely to happen if $\sigma_e$ is small.
In Figures~\ref{fig:mix5a},~\ref{fig:mix5b},and~\ref{fig:mix5c} we see this phenomenon.

\begin{figure}
  \begin{subfigure}[b]{0.3\textwidth}
    \includegraphics[width=\textwidth]{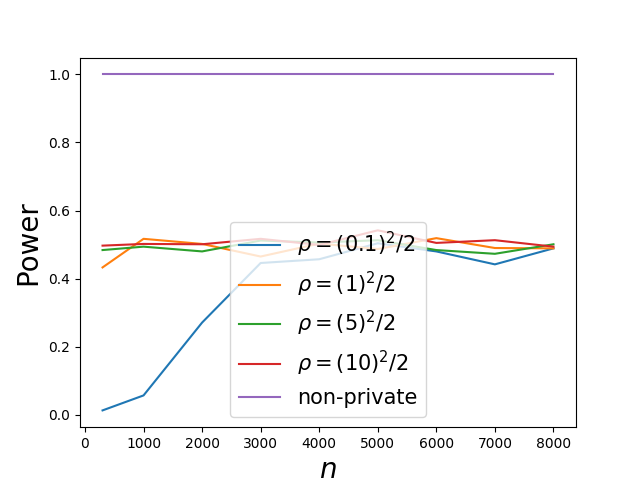}
    \caption{Power for testing mixtures. Equal-sized groups. $x_i\sim\calN(0.5, 1)$, $y_i\sim -1\cdot x_i + \calN(0, 0.01^2)$ for Group 1. $y_i\sim 1\cdot x_i + \calN(0, 0.01^2)$ for Group 2.}
    \label{fig:mix5a}
  \end{subfigure}
  \qquad
   \begin{subfigure}[b]{0.3\textwidth}
    \includegraphics[width=\textwidth]{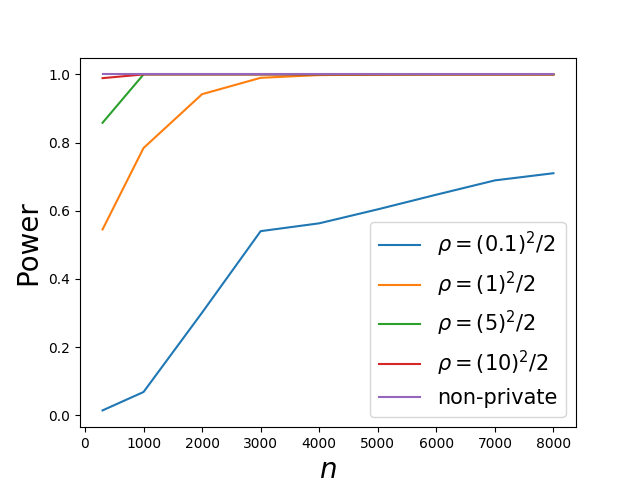}
    \caption{Power for testing mixtures. Equal-sized groups. $x_i\sim\calN(0.5, 1)$, $y_i\sim -1\cdot x_i + \calN(0, 0.35^2)$ for Group 1. $y_i\sim 1\cdot x_i + \calN(0, 0.35^2)$ for Group 2.}
    \label{fig:mix5b}
    \end{subfigure}
    \qquad
    \begin{subfigure}[b]{0.3\textwidth}
    \includegraphics[width=\textwidth]{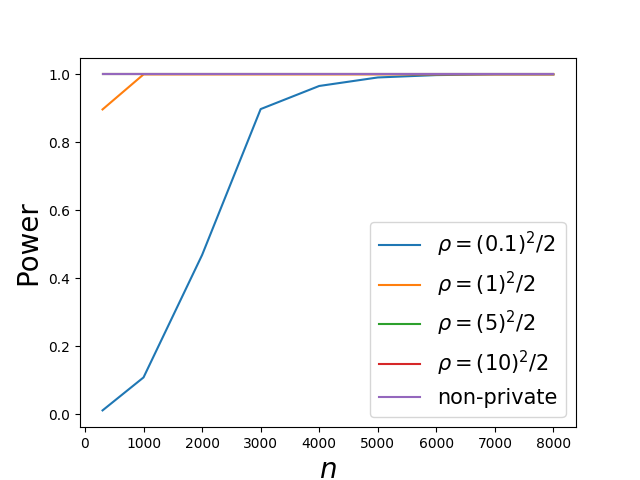}
    \caption{Power for testing mixtures. Equal-sized groups. $x_i\sim\calN(0.5, 1)$, $y_i\sim -1\cdot x_i + \calN(0, 1)$ for Group 1. $y_i\sim 1\cdot x_i + \calN(0, 1)$ for Group 2.}
    \label{fig:mix5c}
    \end{subfigure}
  \caption{}
\end{figure}

\subsubsection{Nonparametric Tests via Kruskal-Wallis}

We now proceed to show results for comparing the mixture models
based on Kruskal-Wallis (KW) to the parametric $F$-statistic method.

\textbf{Evaluating the Significance}:
The KW methods, on average, achieve
the target significance levels for all values of
$\rho$ as illustrated
in Figures~\ref{fig:mixkw1a} and~\ref{fig:mixkw1b},
where we vary the noise in the dependent variable.

\begin{figure}
  \begin{subfigure}[b]{0.5\textwidth}
    \includegraphics[width=\textwidth]{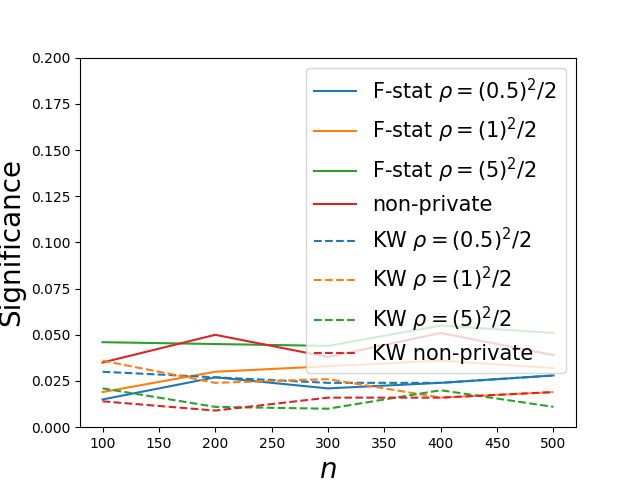}
    \caption{Significance for testing mixtures. Equal-sized groups. $x_i\sim\calN(0.5, 0.1)$, $y_i\sim 1\cdot x_i + \calN(0, 0.35^2)$ for Group 1. $y_i\sim 1\cdot x_i + \calN(0, 0.35^2)$ for Group 2.}
    \label{fig:mixkw1a}
  \end{subfigure}
    \qquad
    \begin{subfigure}[b]{0.5\textwidth}
    \includegraphics[width=\textwidth]{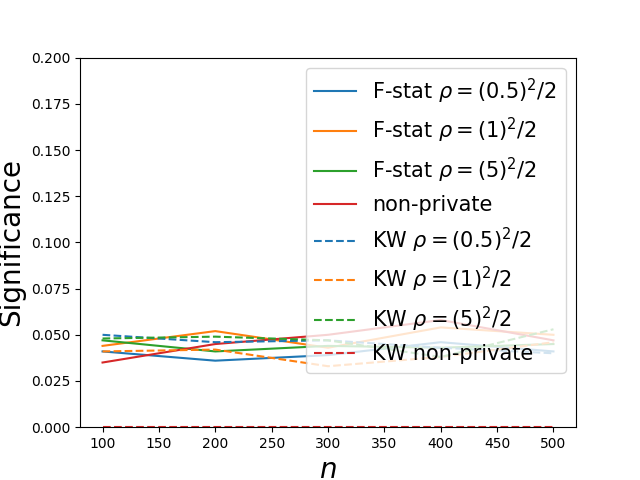}
    \caption{Significance for testing mixtures. Equal-sized groups. $x_i\sim\calN(0.5, 1)$, $y_i\sim 1\cdot x_i + \calN(0, 1)$ for Group 1. $y_i\sim 1\cdot x_i + \calN(0, 1)$ for Group 2.}
    \label{fig:mixkw1b}
    \end{subfigure}
  \caption{}
\end{figure}

\textbf{Evaluating the Power as we Increase the Difference in Slopes}:
We see that the the KW method outperforms the $F$-statistic method on
small datasets. But as the difference in slopes between the two groups
increases, the $F$-statistic method does better and begins to outperform
the KW method. See Figures~\ref{fig:mixkw2a} and
~\ref{fig:mixkw2b}.

\begin{figure}
  \begin{subfigure}[b]{0.5\textwidth}
    \includegraphics[width=\textwidth]{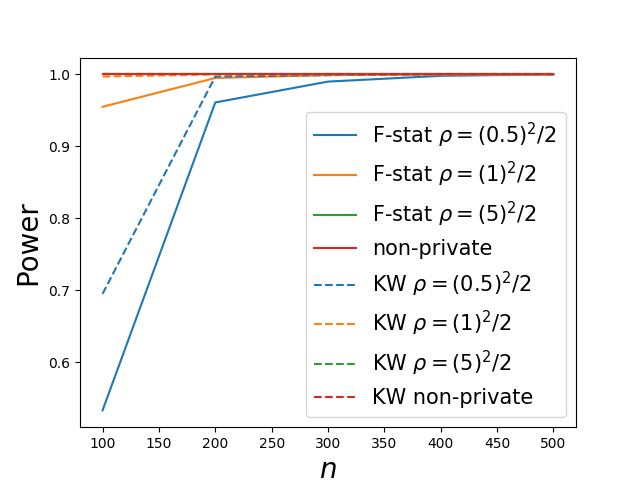}
    \caption{Power for testing mixtures. Equal-sized groups. $x_i\sim\calN(0.5, 1)$, $y_i\sim -1\cdot x_i + \calN(0, 1)$ for Group 1. $y_i\sim 1\cdot x_i + \calN(0, 1)$ for Group 2. }
    \label{fig:mixkw2a}
  \end{subfigure}
  \qquad
  \begin{subfigure}[b]{0.5\textwidth}
    \includegraphics[width=\textwidth]{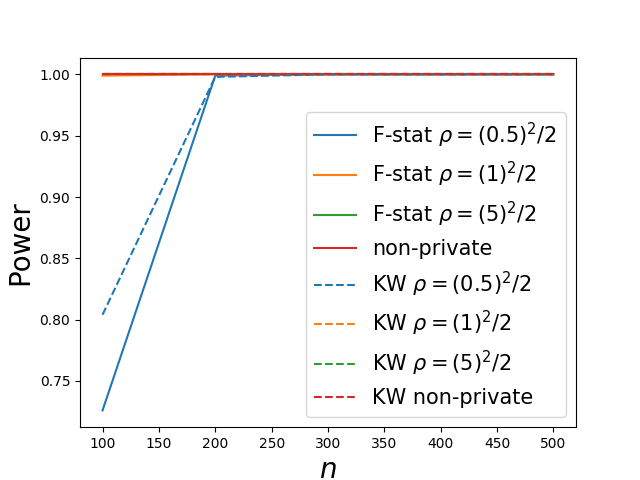}
    \caption{Power for testing mixtures. Equal-sized groups. $x_i\sim\calN(0.5, 1)$, $y_i\sim -1\cdot x_i + \calN(0, 1)$ for Group 1. $y_i\sim 5\cdot x_i + \calN(0, 1)$ for Group 2.}
    \label{fig:mixkw2b}
    \end{subfigure}
  \caption{}
\end{figure}

\textbf{Evaluating the Power as we Increase the Variance of the Independent Variable}:
In Figures~\ref{fig:mixkw3a} and~\ref{fig:mixkw3b},
we see that the $F$-statistic method outperforms the KW method when
the variance of the independent
variable is much larger (10x) than previously.

\begin{figure}
   \begin{subfigure}[b]{0.5\textwidth}
    \includegraphics[width=\textwidth]{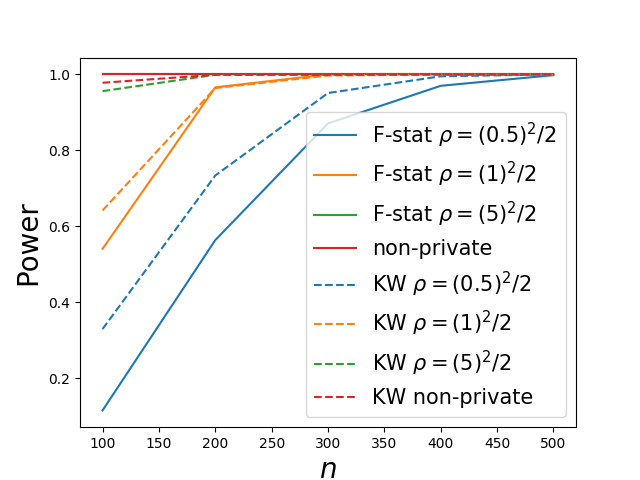}
    \caption{Power for Kruskal-Wallis versus the $F$-statistic. $x_i\sim\calN(0.5, 1)$, $y_i\sim -1\cdot x_i + \calN(0, 1)$ for Group 1. $y_i\sim 1\cdot x_i + \calN(0, 1)$ for Group 2.}
    \label{fig:mixkw3a}
  \end{subfigure}
  \qquad
   \begin{subfigure}[b]{0.5\textwidth}
    \includegraphics[width=\textwidth]{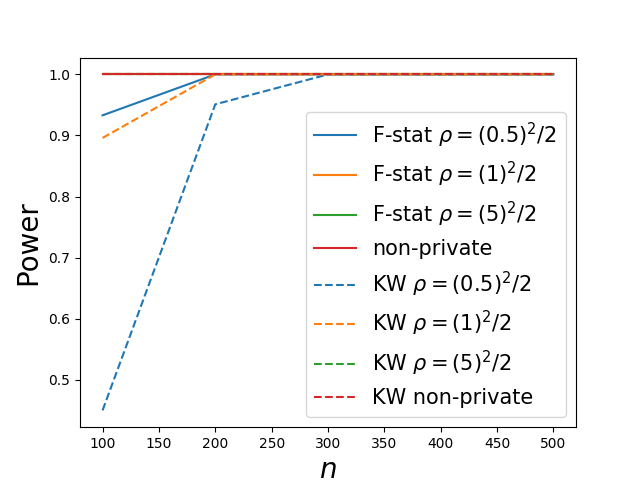}
    \caption{Power for Kruskal-Wallis versus the $F$-statistic. $x_i\sim\calN(0.5, 10)$, $y_i\sim -1\cdot x_i + \calN(0, 1)$ for Group 1. $y_i\sim 1\cdot x_i + \calN(0, 1)$ for Group 2.}
    \label{fig:mixkw3b}
  \end{subfigure}
  \caption{}
\end{figure}

\subsection{Testing on Opportunity Insights Data}

The Opportunity Insights (OI) team gave us simulated data
for census tracts from the following states in the United States:
Idaho, Illinois, New York, North Carolina, Texas, and Tennessee.
The dependent and independent variables are the
child and parent national income percentiles, respectively.
For the linear tester, a rejection of the null hypothesis implies that
there is a relationship between the parent and child income percentiles.
For the mixture model tester, it implies that there is more than one
linear relationship in the data which suggests that more granular data
is needed for analysis on the data. The groups of data
fed to the mixture model tester are conglomeration of one or more tracts.

Some of these states have a small number of datapoints.
For example, within Illinois,
there are tracts with just $n=39$ datapoints. 
For the Illinois dataset,
there are $n = 219, 594$ datapoints that are subdivided into
$3, 108$ census tracts. 
The North Carolina and Texas datasets consists of datapoints subdivided into
$2, 156$ and $5, 187$ census tracts respectively.
We will focus on data from North Carolina (NC), and Texas (TX)
and experimentally evaluate 
$\pr[\text{reject null}]$, the probability of rejecting the
null hypothesis over the randomness of the DP algorithms.
We run our tests
on some census tracts in these states showing how these measures
fair as the privacy parameter is relaxed.
For the experiments below, from each state, we randomly and uniformly
select: 
(i) a single tract;
(ii) 10 randomly selected tracts and concatenate; and
(iii) 50 randomly selected tracts and concatenate.
Then we test for the presence of a
(non-zero) linear relationship.
The concatenation could result in hundreds or thousands of points.

Our tests are evaluated on the OI data.
We have not included the test based on  Kruskal-Wallis as our current implementation
is, at the moment, 
relatively computationally inefficient to evaluate on such large datasets.
See above synthetic data experiments for comparison of Kruskal-Wallis to the 
$F$-statistic method.
Figures~\ref{fig:oilin2a}, and~\ref{fig:oilin3a} show the 
probability of rejecting the null as we increase the parameter $\rho$ when using
the DP linear tester. 
Figures~\ref{fig:oimix2a}, and~\ref{fig:oimix3a} show the
corresponding results for the $F$-stat based DP mixture model tester.
We see that for the small-sized datasets tend to have a small chance of
rejecting the null while larger ones have a higher chance.

\begin{figure}
  \begin{subfigure}[b]{0.5\textwidth}
    \includegraphics[width=\textwidth]{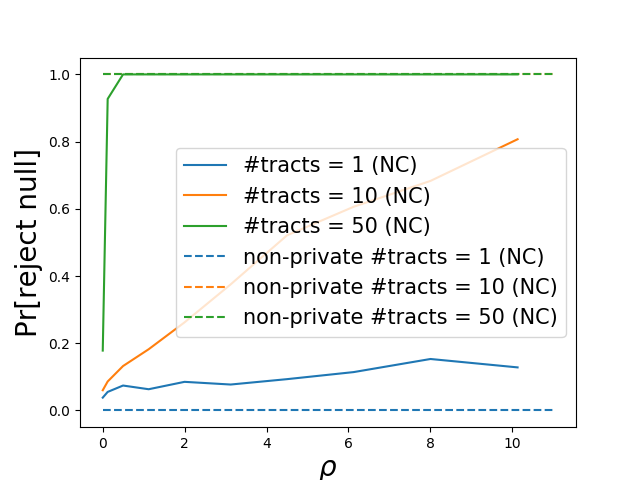}
    \caption{$\pr[\text{reject null}]$ for testing a linear relationship in NC.
    $\Delta = 2$.}
    \label{fig:oilin2a}
  \end{subfigure}
  \qquad
  \begin{subfigure}[b]{0.5\textwidth}
    \includegraphics[width=\textwidth]{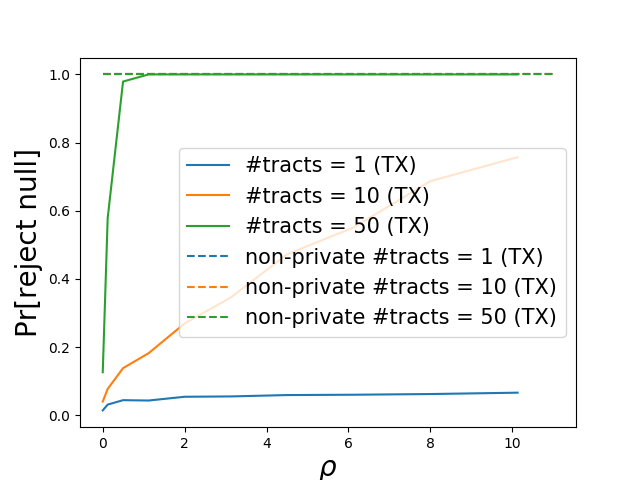}
    \caption{$\pr[\text{reject null}]$ for testing a linear relationship in TX.
    $\Delta = 2$.}
    \label{fig:oilin3a}
  \end{subfigure}
  \caption{}
\end{figure}

\begin{figure}
  \begin{subfigure}[b]{0.5\textwidth}
    \includegraphics[width=\textwidth]{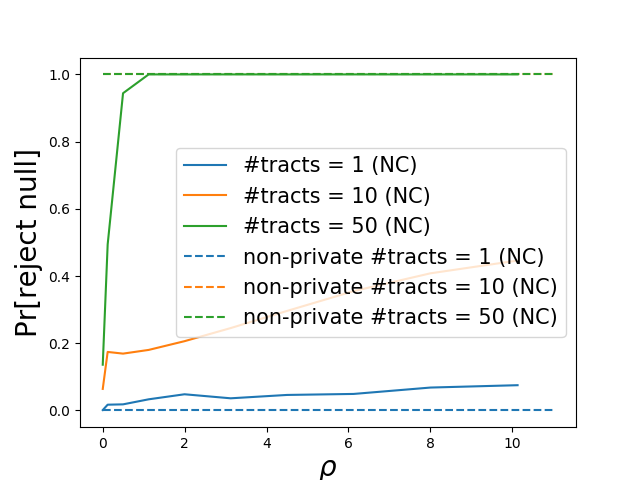}
    \caption{$\pr[\text{reject null}]$ for testing for mixtures in NC.
    $\Delta = 2$.}
    \label{fig:oimix2a}
  \end{subfigure}
  \qquad
  \begin{subfigure}[b]{0.5\textwidth}
    \includegraphics[width=\textwidth]{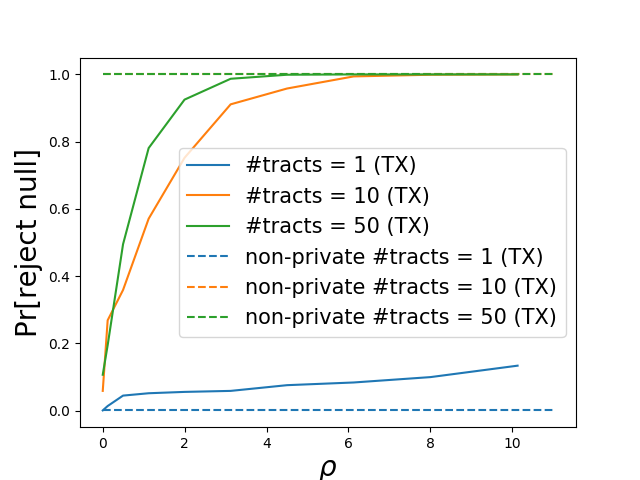}
    \caption{$\pr[\text{reject null}]$ for testing for mixtures in TX.
    $\Delta = 2$.}
    \label{fig:oimix3a}
  \end{subfigure}
  \caption{}
\end{figure}

\subsection{Testing on UCI Bike Dataset}

We use the UCI bike dataset
~\citep{Fanaee-TG14} with 17,389 instances.
For this dataset, we test for a linear relationship
between the ``temp'' (normalized temperature in Celsius)
and ``hr'' (hour between 0 and 23) attributes.
The null hypothesis is that there is no linear
relationship between the ``temp'' and ``hr'' attributes.
Without privacy, the linear relationship tester
based on the $F$-statistic rejects the null.
In Table~\ref{tab:bike}, we show the
probability of Algorithm~\ref{alg:mct}
rejecting the null as we vary the
privacy parameter.
We can observe that for almost all---except for the
smallest setting of $\rho$---privacy parameters,
$\pr[\text{reject null}\mid p\%\text{ data}]$ (probability of
rejecting the null, given $p\%$ of the dataset)
for the private test
matches that of the non-private test.

\begin{table}[]
\begin{tabular}{| l | l | l | l | l | l | l | l | l | l | l | l |}\hline
$\rho$
&  0.005            &  0.125       &  0.5 & 1.125 & 2.0 & 3.125 & 4.5 & 6.125 & 8.0 & 10.125 & non-DP \\\hline
$\pr[\text{reject null}\mid 100\%\text{ data}]$ 
&  1.0          &  1.0       &  1.0 & 1.0 & 1.0 & 1.0 & 1.0 & 1.0 & 1.0 & 1.0 & 1.0  \\\hline
$\pr[\text{reject null}\mid 10\%\text{ data}]$ 
&  0.85          &  1.0       &  1.0 & 1.0 & 1.0 & 1.0 & 1.0 & 1.0 & 1.0 & 1.0 & 1.0  \\\hline
\end{tabular}
\caption{$\pr[\text{reject null}]$ for testing for a
linear relationship between temperature and time
(in hours).}
\label{tab:bike}
\end{table}

While we show that our methods can run on
real-world datasets, the synthetically generated
datasets give a lot more 
information on the behavior of
the tests.

\section{Conclusion}

We have developed differentially private hypothesis tests
for testing a linear relationship in data and for testing for mixtures
in linear regression models.
We also show that the DP $F$-statistic converges to the 
asymptotic distribution of the non-private $F$-statistic.
Through experiments, we show that our Monte Carlo tests
achieve significance that is
less than the target significance level across a wide variety of experiments.
Furthermore, our
tests generally have a high power, getting higher as we increase the
dataset size and/or relax the privacy parameter. Even on small datasets
(in the hundreds) with small slopes, our tests retain the small
significance while having a good power.
We have provided formal statements for the DP $F$-statistic
in the asymptotic regime. We leave to future work the task
of theoretically analyzing the procedures in the
non-asymptotic regime.

Experimental evaluation is done on 
simulated data for the Opportunity Atlas tool, UCI datasets,
and on synthetic datasets of varying distributions on the
independent variable (normal, exponential, and uniform).

\section{Other Acknowledgements}

We are grateful to 
Mark Fleischer (of the U.S. Census Bureau) for many helpful
comments and suggestions that improved the presentation of
this work.
We thank Isaiah Andrews, 
Jordan Awan,
Cynthia Dwork,
Cecilia Ferrando,
Kosuke Imai, Gary King, Po-Ling Loh,
Santiago Olivella, Neil Shephard, 
Adam Smith, Thomas Steinke,
Soichiro Yamauchi, and
seminar participants of the
Harvard econometrics workshop for helpful discussions
related to this work. Finally,
we thank anonymous reviewers for helpful comments.

\clearpage

\bibliographystyle{alpha}
\bibliography{main}

\newcommand{\etalchar}[1]{$^{#1}$}
\begin{thebibliography}{DKM{\etalchar{+}}06}

\bibitem[ACFT19]{AcharyaCFT19}
Jayadev Acharya, Cl{\'{e}}ment~L. Canonne, Cody Freitag, and Himanshu Tyagi.
\newblock Test without trust: Optimal locally private distribution testing.
\newblock In {\em AISTATS}, pages 2067--2076, 2019.

\bibitem[ADKR19]{AliakbarpourDKR19}
Maryam Aliakbarpour, Ilias Diakonikolas, Daniel Kane, and Ronitt Rubinfeld.
\newblock Private testing of distributions via sample permutations.
\newblock In {\em NeurIPS}, pages 10877--10888, 2019.

\bibitem[ADR18]{AliakbarpourDR18}
Maryam Aliakbarpour, Ilias Diakonikolas, and Ronitt Rubinfeld.
\newblock Differentially private identity and equivalence testing of discrete
  distributions.
\newblock In {\em ICML}, pages 169--178, 2018.

\bibitem[AKM19]{NBERw25456}
Isaiah Andrews, Toru Kitagawa, and Adam McCloskey.
\newblock Inference on winners.
\newblock Working Paper 25456, National Bureau of Economic Research, January
  2019.

\bibitem[AM20]{Medina20}
Marco Avella-Medina.
\newblock Privacy-preserving parametric inference: A case for robust
  statistics.
\newblock {\em Journal of the American Statistical Association}, 0(0):1--15,
  2020.

\bibitem[AMS{\etalchar{+}}22]{AMSSV20}
Daniel Alabi, Audra McMillan, Jayshree Sarathy, Adam Smith, and Salil Vadhan.
\newblock Differentially private simple linear regression.
\newblock {\em Proc. Priv. Enhancing Technol.}, 2022.

\bibitem[AS20]{Awan_Slavkovic_2020}
Jordan~Alexander Awan and Aleksandra Slavkovic.
\newblock Differentially private inference for binomial data.
\newblock {\em Journal of Privacy and Confidentiality}, 10(1), Jan. 2020.

\bibitem[ASZ18]{AcharyaSZ18}
Jayadev Acharya, Ziteng Sun, and Huanyu Zhang.
\newblock Differentially private testing of identity and closeness of discrete
  distributions.
\newblock In {\em NeurIPS}, pages 6879--6891, 2018.

\bibitem[BDR10]{BleningerDR10}
Philipp Bleninger, J{\"{o}}rg Drechsler, and Gerd Ronning.
\newblock Remote data access and the risk of disclosure from linear regression:
  An empirical study.
\newblock In {\em PSD}, volume 6344 of {\em Lecture Notes in Computer Science},
  pages 220--233. Springer, 2010.

\bibitem[BRMC17]{Barrientos2017DifferentiallyPS}
Andr'es~F. Barrientos, J.~Reiter, Ashwin Machanavajjhala, and Yan Chen.
\newblock Differentially private significance tests for regression
  coefficients.
\newblock {\em Journal of Computational and Graphical Statistics}, 28:440 --
  453, 2017.

\bibitem[BS16]{BunS16}
Mark Bun and Thomas Steinke.
\newblock Concentrated differential privacy: Simplifications, extensions, and
  lower bounds.
\newblock In {\em TCC}, pages 635--658, 2016.

\bibitem[BS19]{BernsteinS19}
Garrett Bernstein and Daniel~R. Sheldon.
\newblock Differentially private bayesian linear regression.
\newblock In {\em NeurIPS}, pages 523--533, 2019.

\bibitem[CBRG18]{CampbellBRG18}
Zachary Campbell, Andrew Bray, Anna~M. Ritz, and Adam Groce.
\newblock Differentially private {ANOVA} testing.
\newblock In {\em ICDIS}, pages 281--285, 2018.

\bibitem[CF19]{ChettyF19}
Raj Chetty and John~N. Friedman.
\newblock A practical method to reduce privacy loss when disclosing statistics
  based on small samples.
\newblock {\em American Economic Review Papers and Proceedings}, 109:414--420,
  2019.

\bibitem[CFH{\etalchar{+}}18]{NBERw25147}
Raj Chetty, John~N Friedman, Nathaniel Hendren, Maggie~R Jones, and Sonya~R
  Porter.
\newblock The opportunity atlas: Mapping the childhood roots of social
  mobility.
\newblock Working Paper 25147, National Bureau of Economic Research, October
  2018.

\bibitem[CH12]{ChaudhuriH12}
Kamalika Chaudhuri and Daniel~J. Hsu.
\newblock Convergence rates for differentially private statistical estimation.
\newblock In {\em Proceedings of the 29th International Conference on Machine
  Learning, {ICML} 2012, Edinburgh, Scotland, UK, June 26 - July 1, 2012},
  2012.

\bibitem[CHJP19]{10.1093/qje/qjz042}
Raj Chetty, Nathaniel Hendren, Maggie~R Jones, and Sonya~R Porter.
\newblock {Race and Economic Opportunity in the United States: an
  Intergenerational Perspective*}.
\newblock {\em The Quarterly Journal of Economics}, 135(2):711--783, 12 2019.

\bibitem[CKM{\etalchar{+}}19]{CanonneKMSU19}
Cl{\'{e}}ment~L. Canonne, Gautam Kamath, Audra McMillan, Adam~D. Smith, and
  Jonathan Ullman.
\newblock The structure of optimal private tests for simple hypotheses.
\newblock In {\em STOC}, pages 310--321, 2019.

\bibitem[CKS{\etalchar{+}}19]{CouchKSBG19}
Simon Couch, Zeki Kazan, Kaiyan Shi, Andrew Bray, and Adam Groce.
\newblock Differentially private nonparametric hypothesis testing.
\newblock In {\em CCS}, pages 737--751, 2019.

\bibitem[CKS20]{Canonne0S20}
Cl{\'{e}}ment~L. Canonne, Gautam Kamath, and Thomas Steinke.
\newblock The discrete gaussian for differential privacy.
\newblock In {\em NeurIPS}, 2020.

\bibitem[CWZ20]{CWZ20}
T.~Tony Cai, Yichen Wang, and Linjun Zhang.
\newblock The cost of privacy in generalized linear models: Algorithms and
  minimax lower bounds.
\newblock {\em CoRR}, abs/2011.03900, 2020.

\bibitem[DKM{\etalchar{+}}06]{DKMMN06}
Cynthia Dwork, Krishnaram Kenthapadi, Frank McSherry, Ilya Mironov, and Moni
  Naor.
\newblock Our data, ourselves: Privacy via distributed noise generation.
\newblock In {\em EUROCRYPT}, pages 486--503, 2006.

\bibitem[DL09]{DworkL09}
Cynthia Dwork and Jing Lei.
\newblock Differential privacy and robust statistics.
\newblock In {\em STOC}, pages 371--380, 2009.

\bibitem[DMNS06]{DworkMNS06}
Cynthia Dwork, Frank McSherry, Kobbi Nissim, and Adam~D. Smith.
\newblock Calibrating noise to sensitivity in private data analysis.
\newblock In {\em TCC}, pages 265--284, 2006.

\bibitem[DSSU17]{DSSU17}
Cynthia Dwork, Adam Smith, Thomas Steinke, and Jonathan Ullman.
\newblock Exposed! a survey of attacks on private data.
\newblock {\em Annual Review of Statistics and Its Application}, 4(1):61--84,
  2017.

\bibitem[Edg11]{Edgington2011}
Eugene~S. Edgington.
\newblock {\em Randomization Tests}, pages 1182--1183.
\newblock Springer Berlin Heidelberg, Berlin, Heidelberg, 2011.

\bibitem[EKST19]{Evans2019StatisticallyVI}
G.~Evans, G.~King, Margaret Schwenzfeier, and Abhradeep Thakurta.
\newblock Statistically valid inferences from privacy protected data.
\newblock 2019.

\bibitem[FG14]{Fanaee-TG14}
Hadi Fanaee{-}T and Jo{\~{a}}o Gama.
\newblock Event labeling combining ensemble detectors and background knowledge.
\newblock {\em Prog. Artif. Intell.}, 2(2-3):113--127, 2014.

\bibitem[FWS21]{FWS20}
Cecilia Ferrando, Shufan Wang, and Daniel Sheldon.
\newblock Parametric bootstrap for differentially private confidence intervals.
\newblock {\em CoRR}, abs/2006.07749, 2021.

\bibitem[GLRV16]{RogersVLG16}
Marco Gaboardi, Hyun{-}Woo Lim, Ryan~M. Rogers, and Salil~P. Vadhan.
\newblock Differentially private chi-squared hypothesis testing: Goodness of
  fit and independence testing.
\newblock In {\em ICML}, pages 2111--2120, 2016.

\bibitem[Gut13]{gut2013probability}
A.~Gut.
\newblock {\em Probability: A Graduate Course}.
\newblock Springer Texts in Statistics. Springer New York, 2013.

\bibitem[HJ12]{horn_johnson_2012}
Roger~A. Horn and Charles~R. Johnson.
\newblock {\em Matrix Analysis}.
\newblock Cambridge University Press, 2 edition, 2012.

\bibitem[HR11]{huber2011robust}
P.J. Huber and E.M. Ronchetti.
\newblock {\em Robust Statistics}.
\newblock Wiley Series in Probability and Statistics. Wiley, 2011.

\bibitem[HTF09]{HastieTF09}
Trevor Hastie, Robert Tibshirani, and Jerome~H. Friedman.
\newblock {\em The Elements of Statistical Learning: Data Mining, Inference,
  and Prediction, 2nd Edition}.
\newblock Springer Series in Statistics. Springer, 2009.

\bibitem[Hub64]{Huber64}
Peter~J. Huber.
\newblock {Robust Estimation of a Location Parameter}.
\newblock {\em The Annals of Mathematical Statistics}, 35(1):73 -- 101, 1964.

\bibitem[Kee10]{keener2010theoretical}
R.W. Keener.
\newblock {\em Theoretical Statistics: Topics for a Core Course}.
\newblock Springer Texts in Statistics. Springer New York, 2010.

\bibitem[KOV16]{KairouzOV16}
Peter Kairouz, Sewoong Oh, and Pramod Viswanath.
\newblock Extremal mechanisms for local differential privacy.
\newblock {\em J. Mach. Learn. Res.}, 17:17:1--17:51, 2016.

\bibitem[Lei11]{Lei11}
Jing Lei.
\newblock Differentially private m-estimators.
\newblock In {\em NeurIPS}, pages 361--369, 2011.

\bibitem[Mir17]{Mironov17}
Ilya Mironov.
\newblock R{\'{e}}nyi differential privacy.
\newblock In {\em CSF}, pages 263--275, 2017.

\bibitem[RK17]{RogersK17}
Ryan Rogers and Daniel Kifer.
\newblock A new class of private chi-square hypothesis tests.
\newblock In {\em AISTATS}, pages 991--1000, 2017.

\bibitem[SGG{\etalchar{+}}19]{SwanbergGGRGB19}
Marika Swanberg, Ira Globus{-}Harris, Iris Griffith, Anna~M. Ritz, Adam Groce,
  and Andrew Bray.
\newblock Improved differentially private analysis of variance.
\newblock {\em Proc. Priv. Enhancing Technol.}, 2019(3):310--330, 2019.

\bibitem[She17]{Sheffet17}
Or~Sheffet.
\newblock Differentially private ordinary least squares.
\newblock In {\em ICML}, pages 3105--3114, 2017.

\bibitem[She18]{Sheffet18}
Or~Sheffet.
\newblock Locally private hypothesis testing.
\newblock In {\em ICML}, pages 4612--4621, 2018.

\bibitem[Sur20]{Suresh20}
Ananda~Theertha Suresh.
\newblock Robust hypothesis testing and distribution estimation in hellinger
  distance.
\newblock {\em CoRR}, abs/2011.01848, 2020.

\bibitem[Swe97]{Sweeney97}
Latanya Sweeney.
\newblock Weaving technology and policy together to maintain confidentiality.
\newblock {\em The Journal of Law, Medicine \& Ethics}, 25(2-3):98--110, 1997.

\bibitem[TC16]{TaskC16}
Christine Task and Chris Clifton.
\newblock Differentially private significance testing on paired-sample data.
\newblock In {\em Proceedings of the 2016 {SIAM} International Conference on
  Data Mining, Miami, Florida, USA, May 5-7, 2016}, pages 153--161. {SIAM},
  2016.

\bibitem[Vaa98]{vaart_1998}
A.~W. van~der Vaart.
\newblock {\em Asymptotic Statistics}.
\newblock Cambridge Series in Statistical and Probabilistic Mathematics.
  Cambridge University Press, 1998.

\bibitem[Wan18]{Wang18}
Yu{-}Xiang Wang.
\newblock Revisiting differentially private linear regression: optimal and
  adaptive prediction {\&} estimation in unbounded domain.
\newblock In {\em UAI}, pages 93--103, 2018.

\bibitem[WLK15]{WangLK15}
Yue Wang, Jaewoo Lee, and Daniel Kifer.
\newblock Differentially private hypothesis testing, revisited.
\newblock {\em CoRR}, abs/1511.03376, 2015.

\end{thebibliography}

\appendix

\clearpage

\section{$F$-Statistic for the General Linear Model}

The proofs in this section rely on insights from
~\citep{keener2010theoretical}.
In fact,
Theorem~\ref{thm:fstat} can be seen as a special case of Theorem 14.11
in~\citep{keener2010theoretical} where, under the null hypothesis, the projection onto
$\omega_0$ results in $\beta^N$ and, under the alternative hypothesis, the projection
onto $\omega$ results in $\beta$.

We present the main test statistic we will use for hypothesis testing.
This statistic is equivalent to the generalized likelihood ratio test statistic and
can be written as
\begin{equation}
T =  \left(\frac{n-r}{r-q}\right)\frac{\norm{Y - X\hat\beta^N}^2 - \norm{Y - X\hat\beta}^2}{\norm{Y-X\hat\beta}^2} = \left(\frac{n-r}{r-q}\right)\frac{\norm{X\hat\beta - X\hat\beta^N}^2}{\norm{Y-X\hat\beta}^2},
\label{eq:fstat}
\end{equation}
where $\hat\beta^N, \hat\beta$ are the least squares estimates under the null
and alternative hypothesis respectively.

The vectors $Y-X\hat\beta$ and $X\hat\beta - X\hat\beta^N$ can be shown to
be orthogonal, so that $\norm{Y-X\hat\beta^N}^2 = \norm{Y-X\hat\beta}^2 + \norm{X\hat\beta-X\hat\beta^N}^2$ by the Pythagorean theorem~\citep{keener2010theoretical}.

\begin{lemma}[Weak Law of Large Numbers,  see~\citep{keener2010theoretical}]
Let $Y_1, \ldots, Y_n$ be i.i.d. random variables with
mean $\mu$. Then
$$
\frac{1}{n}\sum_{i=1}^nY_i = \bar{Y}_n\xrightarrow{P}\mu,
$$
provided that $\E[|Y_i|] < \infty$.
\label{lem:wlln}
\end{lemma}

\begin{theorem}

For every $n\in\naturals$ with $n > r$, let
$X = X_n\in\reals^{n\times p}$ be the design matrix.
Under the general linear model 
$Y = Y_n\sim\calN(X_n\beta, \sigma_e^2I_{n\times n})$,
$$T = T_n\sim F_{r-q, n-r}(\eta_n^2),\quad \eta_n^2 = \frac{\norm{X_n\beta - X_n\beta^N}^2}{\sigma_e^2},
$$
where $F_{n, m}$ is the $F$-distribution with parameters $n$, $m$, 
$\beta^N = \E[\hat\beta^N]$,
$q$ is the dimension of $\omega_0$, and $r$ is the dimension of $\omega$ with $0\leq q < r$.

Furthermore,
\begin{enumerate}
\item
$$\norm{Y_n-X_n\hat\beta}^2\sim\calX^2_{n-r}\sigma_e^2,\quad\norm{X_n\hat\beta-X_n\hat\beta^N}^2\sim\calX^2_{r-q}(\eta_n^2)\sigma_e^2.$$
\item
If there exists $\eta\in\reals$ such that $\frac{\norm{X_n\beta - X_n\beta^N}^2}{\sigma_e^2}\rightarrow\eta^2$, then
$$
T = T_n\sim F_{r-q, n-r}(\eta_n^2)
\xrightarrow{D} \frac{\chi^2_{r-q}(\eta^2)}{r-q}.
$$
\item We have
$$\frac{\norm{Y_n-X_n\hat\beta}^2}{n-r}\xrightarrow{P}\sigma_e^2.$$
\end{enumerate}

The values $\beta = \E[\hat\beta], \beta^N = \E[\hat\beta^N]$ are
the expected values of our parameter estimates under
the alternative and null hypotheses respectively.

\label{thm:fstat}
\end{theorem}

\begin{proof}[Proof of Lemma~\ref{thm:fstat}]

First, define $\Omega_0 = \{X\beta\,:\,\beta\in\reals^p, \beta\in\omega_0\}$ 
(for null hypothesis) and
$\Omega = \text{span}\{c_1, \ldots, c_p\} = \{X\beta\,:\,\beta\in\reals^p, \beta\in\omega\}$ (for alternative) where $c_1, \ldots, c_p$ are the columns of
$X$.
Write $Y = \sum_{i=1}^nZ_iv_i$,
where $v_1, \ldots, v_n$ is an orthonormal basis chosen so that
$v_1, \ldots, v_r$ spans $\Omega$ (so that $v_{r+1}, \ldots, v_n$ lies in $\Omega^\perp$) and $v_1, \ldots, v_q$ spans $\Omega_0$.

For all $i\in[n]$, we can find $Z_i\in\reals^n$ by 
introducing an $n\times n$ matrix $O$ with columns $v_1, \ldots, v_n$.
As a result, $O$ is an orthogonal matrix (i.e., $O^TO = OO^T = I$ since $O$ is
a square matrix) such that $Z = O^TY$ (or $Y = OZ$).

As before $Y = X\beta + e$, where $e\sim\calN(0,\sigma_e^2I_{n\times n})$. As a result,
$Z = O^T(X\beta + e) = O^TX\beta + O^Te$. If we define $\tau = O^TX\beta$ and
$e^* = O^Te$, then $Z = \tau + e^*$.
And because $\E[e^*] = \E[O^Te] = O^T\E[e] = 0$ and
$\cov(e^*) = \cov(O^Te) = O^T\cov(e)O = O^T(\sigma_e^2I)O = \sigma_e^2I$,
$e^* \sim \calN(0, \sigma_e^2I_{n\times n})$. As a result,
$$
Z \sim \calN(\tau, \sigma_e^2I_{n\times n}).
$$

Next, since $c_1, \ldots, c_p$ denotes the columns of the design matrix $X$,
$X\beta = \sum_{i=1}^p\beta_ic_i$ and
$$
\tau = O^TX\beta = \begin{pmatrix}
v_1^T\\
v_2^T\\
\vdots\\
v_n^T\\
\end{pmatrix}
\sum_{i=1}^p\beta_ic_i
= \begin{pmatrix}
\sum_{i=1}^p\beta_iv_1^Tc_i\\
\sum_{i=1}^p\beta_iv_2^Tc_i\\
\vdots\\
\sum_{i=1}^p\beta_iv_n^Tc_i\\
\end{pmatrix}.
$$

And because $c_1, \ldots, c_p$ all lie in $\Omega$ and
$v_{r+1}, \ldots, v_n$ in $\Omega^\perp$, we have
$v^T_kc_i = 0$ for all $k > r$. Then,
$\tau_{r+1} = \cdots = \tau_n = 0.$

Additionally, because $\tau = O^TX\beta$,
$$
X\beta = O\tau = (v_1 \cdots v_n)\begin{pmatrix}
\tau_1\\
\tau_2\\
\vdots\\
\tau_r\\
0\\
\vdots\\
0\end{pmatrix}
 = \sum_{i=1}^r\tau_iv_i.
 $$

\textit{Essentially, we established a one-to-one relation between points
$X\beta\in\Omega$
and $(\tau_1, \ldots, \tau_r)\in\reals^r$.}

Now, since $Z\sim\calN(\tau, \sigma_e^2I_{n\times n})$, 
$Z_1, \ldots, Z_n$ are independent and
$Z_i\sim\calN(\tau_i, \sigma_e^2)$ for all $i\in[n]$. Furthermore,
$\tau_{r+1} = \cdots = \tau_n = 0$. 

Then since
$X\beta = \sum_{i=1}^r\tau_iv_i$,
$X\hat\beta = \sum_{i=1}^rZ_iv_i$ and
$X\hat\beta^N = \sum_{i=1}^qZ_iv_i$.

As a result, we get
$$
\norm{Y - X\hat\beta}^2 = \norm{\sum_{i=r+1}^nZ_iv_i}^2 = 
\sum_{i=r+1}^n\sum_{j=r+1}^nZ_iZ_jv_i^Tv_j = \sum_{i=r+1}^nZ_i^2,
$$
which follows since for all $i\neq j$, $v^T_iv_j = 0$ and for $i=j$, $v^T_iv_j = 1$.
Also, since
$\tau_{r+1} = \cdots = \tau_n = 0$,
$Z_i\sim\sigma_e\calN(\tau_i, 1)$,
$\norm{Y - X\hat\beta}^2\sim\calX^2_{n-r}\sigma_e^2$.

Similarly,
$$
\norm{Y - X\hat\beta^N}^2 = \sum_{i=q+1}^nZ_i^2.
$$
Then by Equation~(\ref{eq:fstat}),
$$
T = \frac{\frac{1}{r-q}\sum_{i=q+1}^rZ_i^2}{\frac{1}{n-r}\sum_{i=r+1}^nZ_i^2} = \frac{\frac{1}{r-q}\sum_{i=q+1}^r(Z_i/\sigma_e)^2}{\frac{1}{n-r}\sum_{i=r+1}^n(Z_i/\sigma_e)^2}.
$$

The variables $Z_i$ are independent and because $Z_i/\sigma_e\sim\calN(\tau_i/\sigma_e, 1)$, using
properties of the (non-central) chi-squared distribution,
$$
\sum_{i=q+1}^r\left(\frac{Z_i}{\sigma_e}\right)^2\sim\chi^2_{r-q}(\eta_n^2),\quad
\eta_n^2 = \sum_{i=q+1}^r\frac{\tau_i^2}{\sigma_e^2}.
$$
As a corollary,
$\norm{X\hat\beta-X\hat\beta^N}^2\sim\calX^2_{r-q}(\eta_n^2)\sigma_e^2$.

Also, since $\tau_i = 0$ for $i=r+1, \ldots, n$, $Z_i/\sigma_e\sim\calN(0, 1)$ for $i=r+1, \ldots, n$.
As a result,
$\sum_{i=r+1}^n(Z_i/\sigma_e)^2\sim\chi^2_{n-r}$.

By definition of the noncentral $F$-distribution, we have
$T\sim F_{r-q, n-r}(\eta_n^2)$ where $\eta_n^2 = \sum_{i=q+1}^r\frac{\tau_i^2}{\sigma_e^2}$.

We know that $X\beta = \E[X\hat\beta] = \sum_{i=1}^r\tau_iv_i$ and
$X\beta^N = \E[X\hat\beta^N] = \sum_{i=1}^q\tau_iv_i$. As a result,
$X\beta - X\beta^N = \sum_{i=q+1}^r\tau_iv_i$ so that
$$
\norm{X\beta - X\beta^N}^2 = \sum_{i=q+1}^r\tau_i^2.
$$

This completes the proof of the distribution
of $T$.

In the limit, by Lemma~\ref{lem:convf}
$$T_n = F_{r-q, n-r}(\eta_n^2) \xrightarrow{D} \frac{\chi^2_{r-q}(\eta^2)}{r-q}.$$

Finally, we have established that
$$\frac{\norm{Y_n-X_n\hat\beta}^2}{n-r}\sim\frac{\calX^2_{n-r}\sigma_e^2}{n-r}.$$
Applying the weak law of large numbers (Lemma~\ref{lem:wlln}),
$\frac{\chi^2_{n-r}}{n-r}\xrightarrow{P} 1$. 
As a result, 
$$
\frac{\norm{Y_n-X_n\hat\beta}^2}{n-r}\xrightarrow{P}\sigma_e^2.
$$

\end{proof}

\begin{lemma}
Let $X\sim F_{n, m}(\lambda)$ and
$Y = \lim_{m\rightarrow\infty} nX$. Then $Y\sim\chi^2_n(\lambda)$.
\label{lem:convf}
\end{lemma}

\begin{proof}
By definition of the $F$-distribution,
$X = \frac{N/n}{M/m}$, where
$N\sim\calX^2_n(\lambda)$ and $M\sim\calX^2_m$ are independent random
variables.

For mutually independent $\chi^2_1$
random variables $Y_1, \ldots, Y_m$, 
$$M = \frac{Y_1 + \cdots + Y_m}{m}.$$

By the weak law of large numbers
(Lemma~\ref{lem:wlln}),
$$
\frac{M}{m}\xrightarrow{P}\E[Y_1] = 1.
$$

As a result, by Slutsky's Theorem (Theorem~\ref{thm:slutsky}),
$$
\lim_{m\rightarrow\infty}nX = \lim_{m\rightarrow\infty}\frac{N}{M/m} = N \sim\calX^2_n(\lambda).
$$
\end{proof}

\section{Duality Between Testing and Interval Estimation}
\label{sec:dual}

In this section, we expand on the relationship between interval or region estimation
(i.e., confidence interval estimation) and hypothesis testing. We present a general
formulation, which can be specialized to linear regression.
Previous work (e.g.,~\citep{FWS20}) examines the generation
of confidence intervals for 
differentially private parametric inference.
Our work focuses on
hypothesis testing.

As before, for some unknown parameter $\theta\in\Omega, Z\sim P_\theta$ is the observed
data. Also let $f:\Omega\rightarrow\reals$ be a function on the parameter space
(i.e., for linear regression, we can compute functions of the slope and/or intercept).

\begin{definition}[Confidence Region]
A (random) set $S(Z)$ is a $1-\alpha$ confidence region for a (function of a) parameter $f(\theta)$ if
$$
\pr_\theta[f(\theta) \in S(Z)]\geq 1 - \alpha, \quad \forall\theta\in\Omega.
$$
\end{definition}

A confidence region is, essentially,
a multi-dimensional generalization of a confidence interval.

\begin{definition}[Acceptance Region]
For every $f_0\in \reals$, $A(f_0)$ is the acceptance region for a nonrandomized level $\alpha$ test of
$$
H_0: f(\theta) = f_0\,\text{ vs. }\, H_1: f(\theta) \neq f_0.
$$

$A(f_0)$ denotes the range of values that would lead to acceptance of the null hypothesis, when the null
is true, where for the
level $\alpha$ test,
$$
\pr_\theta[Z \in A(f(\theta))]\geq 1 - \alpha, \quad \forall\theta\in\Omega.
$$
\end{definition}

Define the following function:
$$
S(z) = \{f_0\,:\,\exists \theta_0\text { s.t. }f_0 = f(\theta_0)\text{ and }z\in A(f_0)\}.
$$

Then $f(\theta)\in S(Z) \iff Z\in A(f(\theta))$ which implies that
$$
\pr_\theta(f(\theta) \in S(Z)) = \pr_\theta(Z\in A(f(\theta))) \geq 1 - \alpha.
$$
We have established that $S(Z)$ is, thus, a $1-\alpha$ confidence region for $f$.
Essentially, we have shown that
\textbf{we can construct confidence regions from a family of nonrandomized tests}.

For any function $f:\Omega\rightarrow\reals$ and
$f_0 = f(\theta_0)$, we could seek to obtain a $1-\alpha$ confidence region $S(Z)$
for the parameter $f_0$ (e.g., the mean or median).

Now, consider a test $\phi$ defined by
$$
\phi(z) = \begin{cases} 1 &\mbox{if } f_0 \notin S(z) \\
0 & \mbox{otherwise}
\end{cases}.
$$

Then if $f(\theta) = f_0$, then
\begin{align}
\E_\theta\phi &= \pr_\theta[f_0 \notin S(Z)] \\
&= \pr_\theta[f(\theta)\notin S(Z)] \leq \alpha.
\end{align}

This test has level at most $\alpha$ for testing
$$H_0: f(\theta) = f_0\,\text{ vs. }\, H_1: f(\theta) \neq f_0.$$

Then if the coverage probability for $S(Z)$ is exactly $1-\alpha$, then
$$
\pr_\theta[f(\theta) \in S(Z)] = 1-\alpha,\quad \forall\theta\in\Omega,
$$
so that $\phi$ will have level of exactly $\alpha$.

To summarize, we have shown that we can: (1) Construct confidence regions from a family of nonrandomized tests. (2) Construct a family of nonrandomized tests from a $1-\alpha$ confidence region.

Since (nontrivial) DP tests are randomized, to
apply the duality framework, we could de-randomize by
giving the DP procedure the random bits to be used for
DP. Or we can also define the DP procedure as a family
of nonrandomized tests.

We can still turn DP procedures for confidence region estimation into
procedures for testing although, as we show below, the power of the
corresponding tests is likely to be very low if the area of the
confidence regions are too large.

\subsection{More Details on Experimental Evaluation of DP Confidence Intervals}

We now proceed to construct a hypothesis test based
on DP parametric bootstrap confidence intervals
(e.g., using the work of~\citep{FWS20}).
Then we will experimentally compare to our linear relationship
tester based on the DP $F$-statistic.

Suppose that $\theta$ is the set of parameters (e.g., standard deviation of
the dependent and independent variables) and
$f = f(\theta)$ is the estimation target (e.g., the slope in the dataset).
The goal is to obtain a
$1-\alpha$ confidence interval $[\hat{a}_n, \hat{b}_n]$ for $f(\theta)$ via
an end-to-end differentially private procedure. In other words, we want
$$
\pr\left[\hat{a}_n \leq f(\theta) \leq \hat{b}_n\right] = 1-\alpha,
$$
where the probability is taken over both $\theta$ and
$f$.

Because of the randomized nature of (non-trivial) DP procedures, the 
finite-sample coverage
of the interval might not exactly be close to $1-\alpha$.
Ferrando, Wang, and Sheldon~\citep{FWS20} show the consistency of these
intervals (in the large-sample, asymptotic regime).

Algorithm~\ref{alg:mctci} follows the same framework as
Algorithm~\ref{alg:mct}, except that instead of simulating test statistics
under the null hypothesis, the
goal is to calculate a confidence interval for the slope.
$P_{(\tilde\theta_0, \tilde\theta_1)}$ 
denotes the distribution from which
we shall generate our bootstrap samples and from which a confidence interval
can be estimated.
For example, for taking bootstrap samples for the slope,
$P_{(\tilde\theta_0, \tilde\theta_1)}$ would
approximately be distributed as
$\calN(\tilde\beta_1, \frac{\widetilde{S^2}}{\widetilde{\textrm{nvar}}})$
where $\widetilde{\textrm{nvar}} = n\cdot\widetilde{\XXm} - n\cdot\tilde{x}^2$ and
$\widetilde{S^2}$ is as defined in Algorithm~\ref{alg:t1}.
Note that a crucial difference between tests
based on the parametric bootstrap confidence intervals
and our tests is the following:
our tests
only use $\tilde\theta_0$, a subset of the estimated DP
statistics, to simulate the null distribution and decide to
reject the null while the
other approach uses 
$(\tilde\theta_0, \tilde\theta_1)$ to decide to reject the null.

The target slope is $b$. For example, if we seek to test for a linear
relationship, we set $b = 0$ since under the null hypothesis, the slope will be 0.
$\dpstats$ is a $\rho$-zCDP procedure for estimating
DP sufficient statistics for a parametric model.
In Algorithm~\ref{alg:mctci}, $(s_{(l)}, s_{(r)})$ is the parametric
bootstrap confidence interval for the slopes under the null hypothesis.

\begin{algorithm}
\KwData{$X\in\reals^{n\times p}; Y\in\reals^n$}
\KwIn{$n\text{ (dataset size)}; \rho\text{ (privacy-loss parameter)}; \alpha\text{ (target significance)}; b\text{ (target slope)}$}

$(\tilde\theta_0, \tilde\theta_1) = \dpstats(X, Y, n, \rho)$

\If {$\tilde\theta_0 =  \tilde\theta_1 = \perp$} {
 \Return Fail to Reject the null
}

\

Select $K > 1/\alpha$

\For {$k=1\ldots K$} {

   Sample slope $s_k \sim P_{(\tilde\theta_0, \tilde\theta_1)}$

}

\

Sort $s_{(1)} \leq\cdots\leq s_{(K)}$

\

Set $l = \lceil (K+1)(\alpha/2)\rceil$

Set $r = \lceil (K+1)(1-\alpha/2)\rceil$

\If { $b\notin (s_{(l)}, s_{(r)})$ } {
 \Return Reject the null
} \Else {
 \Return Fail to Reject the null
}

\caption{DP Test Framework via Parametric Bootstrap Confidence Intervals.}
\label{alg:mctci}
\end{algorithm}

Without privacy, by Lemma~\ref{lem:nonprivateconv}, 
under the null hypothesis,
we know that the slopes will be distributed
as the following distribution:
$\hat\beta_1\sim\calN\left(\beta, \frac{\sigma_e^2}{n\cdot\widehat{\sigma^2_x}}\right)\sim\calN\left(0, \frac{\sigma_e^2}{n\cdot\widehat{\sigma^2_x}}\right)$.
Even
as $n$ increases and as we take fresh samples of
$\hat\beta_1$, the parametric bootstrap confidence interval
around $\hat\beta_1$ gets smaller and more concentrated
around the true value 0. We expect to observe similar behavior
when applying DP.

\section{Experimental Framework for Monte Carlo Evaluation}
\label{sec:expframework}

In Algorithm~\ref{alg:estrej}, we present a generic
procedure showing how we obtain significance
and power on our experimental evaluation of our
DP tests. 
$\EstRejection$ is a meta-procedure that uses $\DataSampler$ to sample a dataset
$D$ either from the null or alternative distribution we are testing. Then it runs
$\MCTester$ to decide whether to reject or fail to reject the null.
$\MCTester$ can be any of the private Monte Carlo tests defined above or their non-private
versions.
The fraction of times (amongst $M$ trials)
a reject decision is returned is estimated and can be used to calculate the
significance or the power of the test.

\paragraph{$\DataSampler$:}
This procedure is used to sample from a user-specified distribution for testing the null
or the alternative hypothesis. For example, for two groups $(X_1, Y_1)$ and
$(X_2, Y_2)$ with slopes $\beta_1$ and $\beta_2$ respectively,
the user could specify
$\beta_1 = \beta_2$ as parameters to the data sampler. 
Other parameters include the size
of the groups, the noise distribution in the dependent or
independent variable, and so on.

\paragraph{$\MCTester$:}
Examples of this procedure are instantiations of
Algorithm~\ref{alg:mct}.

\paragraph{$\CompareAlgs$:}
For user-specified  data samplers and Monte Carlo test procedures, this procedure
(Algorithm~\ref{alg:comparealgs}) collates significance and power (to be plotted, for example).

\begin{algorithm}
\KwIn{$\DataSampler; \MCTester; M (\# trials)$}

$r = 0$

\For {$m = 1, \ldots, M$} {
    $D \leftarrow \DataSampler()$
    
    \If {$\MCTester(D)$ = Reject the null} {
        $r = r + 1$
    }
}

// Compute empirical probability that the test rejects

\Return $r/M$

\caption{$\EstRejection$: Meta-Procedure for Estimating significance and power.}
\label{alg:estrej}
\end{algorithm}

\begin{algorithm}
\KwIn{$\DataSamplerList; \MCTesterList; M (\# trials)$}

$R = []$

\For {$\DataSampler\in\DataSamplerList$} {

    \For {$\MCTester\in\MCTesterList$} {
        $e = \EstRejection(\DataSampler, \MCTester, M)$

        Append $(\DataSampler, \MCTester, M, e)$ to $R$
    }
}

\Return $R$

\caption{$\CompareAlgs$: Compares statistical performance of DP tests.}
\label{alg:comparealgs}
\end{algorithm}

\end{document}